\documentclass[11pt]{amsart}
\input{preamble.tex}
\usepackage{etoolbox}
\usepackage{scalerel}
\usepackage{arydshln}
\tracingpatches
\makeatletter
\patchcmd{\@setref}{\bfseries ??}{\bfseries\color{red} OWE A COFFEE/BEER}{}{}
\makeatother

\newcommand{\kb}{\Bbbk}
\newcommand{\rk}{\mathrm{rk}}

\renewcommand{\MM}{\EuScript{F}}

\newcommand{\Cat}{\mathrm{Cat}}        % category of (small) categories
\newcommand{\Alg}{\mathrm{Alg}}        % category of algebra objects
\newcommand{\Eone}{\mathbb{E}_1}   
\newcommand{\Etwo}{\mathbb{E}_2}   % little 2-disks operad
\newcommand{\labdata}{\mathcal{C}}
\newcommand{\Dec}{\mathrm{Dec}}
\newcommand{\labelfunction}{\mathrm{l}}

\begin{document}

\author{Leon J. Goertz}
\address{Fachbereich Mathematik, Universit\"at Hamburg, 
Bundesstra{\ss}e 55, 
20146 Hamburg, Germany
\href{https://leonjgoertz.github.io/}{leonjgoertz.github.io}
}
\email{leon.goertz@uni-hamburg.de}

\author{Laura Marino}
\address{Fachbereich Mathematik, Universit\"at Hamburg, 
Bundesstra{\ss}e 55, 
20146 Hamburg, Germany
\href{https://sites.google.com/view/laura--marino/home-page}{sites.google.com/view/laura--marino}}
\email{laura.marino@uni-hamburg.de}

\author{Paul Wedrich}
\address{Fachbereich Mathematik, Universit\"at Hamburg, 
Bundesstra{\ss}e 55, 
20146 Hamburg, Germany
\href{https://paul.wedrich.at}{paul.wedrich.at}}
\email{paul.wedrich@uni-hamburg.de}

\title{Monoidal 2-categories from foam evaluation}

\begin{abstract}
In this paper we describe a general framework for constructing examples of
locally linear semistrict monoidal 2-categories covering many examples appearing
in link homology theory. The main input datum is a closed foam evaluation formula. As
examples, we rigorously construct semistrict monoidal 2-categories based on
$\glN$-foams, which underlie the general linear link homology theories, and further examples based on Bar-Natan's decorated cobordisms, related to Khovanov homology. These monoidal 2-categories are typically non-semisimple, have duals for all objects, adjoints for all 1-morphisms, and carry a canonical spatial duality structure expressing oriented 3-dimensional pivotality and sphericality.
\end{abstract}

\maketitle

\setcounter{tocdepth}{2}

\tableofcontents

\section{Introduction}
Monoidal and braided monoidal categories have played a central role in quantum topology
for several decades, providing the algebraic framework underlying quantum invariants of knots and links and the construction of associated topological field theories. Prominent examples include the Temperley--Lieb category, which underlies the Jones polynomial \cite{MR0766964,MR1185809} via the Kauffman bracket, and categories of webs associated with tilting modules of quantum
groups of type~$A$ \cite{MOY, CKM, TVW}.
These categories are \emph{diagrammatic} in the sense that they admit presentations as linear monoidal categories generated by certain collections of generating objects and morphisms, modulo a monoidal ideal spanned by \emph{skein relations}. Furthermore, the presentations can be defined over the integers, which gives these categories meaningful reductions in positive characteristic and makes them amenable to categorification. As a result, they have become fundamental objects not only in representation theory, but also in low-dimensional topology and related areas.

\subsection{Foams and higher categories}
Diagrammatic methods in quantum topology extend beyond the above-mentioned monoidal categories to their categorifications, naturally giving rise to higher-categorical structures. Specifically, categorifications of the Temperley--Lieb category and of type A web categories feature an additional layer of 2-morphisms, which can be described in terms of decorated surfaces or \emph{foams}---singular
surfaces interpolating between webs---often considered embedded in ambient 3-dimensional (or even 4-dimensional) space.
Such higher categories of foams form the algebraic backbone of \emph{Khovanov homology} \cite{Kho, BN2} and the \emph{general linear link homologies} pioneered by Khovanov--Rozansky \cite{KR, MSV, QR, RoW}, for which they appear throughout the related literature in many
variants and refinements; see \cite{khovanov2025lecturessl3foamslink} for recent lecture notes on the topic.
From a categorification perspective, such constructions are most naturally organized\footnote{In \S\ref{sec:hol} we comment on alternatives.} as instances of
monoidal $2$-categories, in which webs appear as $1$-morphisms and foams
as $2$-morphisms, and which decategorify to diagrammatic monoidal 1-categories.
\smallskip

\subsection{Combinatorial vs.\ topological foams}
In much of the existing literature, however, the higher-categorical nature of foam
constructions is only partially made explicit.
In one part of the literature, foam categories are described in a purely \emph{combinatorial} manner,
but typically only at the $1$-categorical level, with foams appearing as
morphisms between webs with fixed boundary conditions, leaving monoidal 2-category considerations aside. Alternatively, foams are sometimes treated as genuinely \emph{topological} objects by viewing
them as embedded singular surfaces in a three-dimensional ball or cube, with composition defined
by gluing and considered up ambient isotopy and certain local relations.
While both perspectives have proved useful, each captures only part of the
underlying structure: the former misses the higher-categorical operations which are essential for TQFT considerations, while the
latter relies essentially on geometric embeddings which obscure the combinatorial nature of foams.

\smallskip
The approach taken in this paper brings these two viewpoints together.
We give a fully algebro-combinatorial construction of the foam formalism as genuine
higher-categorical object, avoiding embedded representatives while retaining the
composition and gluing structures that reflect their three-dimensional origin.
In particular, this allows us to realize the foam categories relevant to link homology
as monoidal $2$-categories in a way that makes their higher-categorical features explicit and
amenable to rigorous and formal verification.

\smallskip
The main result of this paper is the explicit and rigorous construction of the \emph{Bar-Natan monoidal 2-categories} and the \emph{monoidal 2-categories of $\glN$-foams} as two prototypical classes of examples of locally graded-linear semistrict monoidal 2-categories. These are 2-categories equipped with an additional monoidal structure, whose hom-categories are enriched in graded $\kb$-modules for a suitable graded commutative ring $\kb$ in a way compatible with horizontal composition and monoidal structure; see e.g.\ \cite{liu2024braided}. Adopting higher-categorical notation, we see these examples as objects of $\Alg_{\Eone}(\Cat[\Cat[\kb\grmod]])$, i.e.\ $\Eone$-algebras in categories enriched in graded $\kb$-linear categories, the latter modeled as categories enriched in graded $\kb$-modules.

\phantomsection \newcommand{\MainThm}{\hyperref[thm:A]{Main Theorem}} % for referring back here
\begin{mainthm}\label{thm:A}~
\begin{enumerate}
\item
For every commutative ring $\kb$ and every commutative Frobenius algebra $A$ free over
$\kb$, the formalism of Bar--Natan’s decorated cobordisms \cite{BN2} extends to provide a locally $\kb$-linear semistrict monoidal 2-category
\[
\bn_A \in \Alg_{\Eone}(\Cat[\Cat[\kb\Mod]]),
\]
generated by a 2-dualizable self-dual object. If, furthermore, $\kb$ is graded and $A$ a graded Frobenius algebra over $\kb$, then $\bn_A$ can be promoted to carry a graded $\kb$-linear structure, yielding an object of $\Alg_{\Eone}(\Cat[\Cat[\kb\grmod]])$.

\item
For every commutative ring $\kb$ and every positive integer $N$, the Robert--Wagner closed foam evaluation formula \cite{RoW} for $\glN$-foams over $\kb$ extends to provide a locally graded $\kb$-linear semistrict monoidal 2-category
\[
\Foamssmonbicat{N} \in \Alg_{\Eone}(\Cat[\Cat[\kb\grmod]]),
\]
generated by 2-dualizable objects indexed by the set $\{1,\dots, N\}$.
\end{enumerate}
In both types of monoidal 2-categories, all objects admit duals and all 1-morphisms admit adjoints. Furthermore, these can be packaged into canonical \emph{spatial}\footnote{\cite[Definition 4.8]{barrett2018gray}.} \emph{duality structures}\footnote{\cite[Definition 3.10]{barrett2018gray}.} in the sense of Barrett--Meusburger--Schaumann \cite{barrett2018gray}. 
\end{mainthm}

\begin{remnono}[On link homology]
Our \MainThm\ rigorously constructs the monoidal 2-categories which respectively feature in the extension of Khovanov homology (1) and $\glN$ link homology (2) to the case of tangles. In this context,  we consider tangles as properly embedded in the cube $[0,1]^3$ with endpoints on $[0,1] \times \{0,1\} \times \{1/2\}$. Stacking tangles in the direction of the second coordinate corresponds to the horizontal composition of 1-morphisms and placing tangles side-by-side in the direction of the first coordinate corresponds to the monoidal structure. The 2-morphisms play the dual roles of supplying the differentials in \emph{cube of resolution} chain complexes of tangles as well as the components of chain maps associated to tangle cobordisms. Cubes of resolutions associated to tangle diagrams can be accommodated after changing the base of enrichment from linear categories to pretriangulated dg categories of chain complexes, which can be interpreted as a completion operation\footnote{After idempotent and additive completion, the formation of the $\infty$-categorical bounded homotopy category of chain complexes was identified in \cite[\S 3.4]{liu2024braided} as left adjoint to the forgetful functor from stable idempotent-complete categories to additive idempotent-complete categories.}. For example, a locally linear semistrict monoidal 2-category of Soergel bimodules was constructed in \cite{stroppel2024braidingtypesoergelbimodules}. Upon completion to a locally stable, locally linear $(\infty,2)$-category of chain complexes of Soergel bimodules, the Rouquier braiding \cite{Rou2} extends to an $\mathbb{E}_2$-algebra structure \cite{liu2024braided}. We expect that a similar completion for $\glN$-foams can be promoted to a locally stable, locally linear $\mathbb{E}_2$-monoidal $(\infty,2)$-category that controls the general linear link homology theories on the chain level. In this paper, link homology and $\Etwo$-structures only play a motivational role. Instead, we focus on describing the combinatorial nature of the underlying $\Eone$-structures rigorously.
\end{remnono}

\begin{remnono}[On skein theory]
The existence of duals for objects, adjoints for 1-morphisms, the canonical duality structure and spatiality are crucial features of the monoidal 2-categories constructed in the \MainThm. Alternatively, in the language of \cite[Def.\ 2.2.3 and 2.2.4]{douglas2018fusion}, these amount to a pivotal structure on a monoidal planar pivotal 2-category in which every object has a left and a right dual. 
These properties and structures are higher-categorical analogues of the rigidity property and pivotality in monoidal 1-categories, and they ensure that the resulting monoidal 2-categories
admit a well-behaved graphical calculus of 2-dimensional skeins embedded in \emph{oriented} 3-dimensional space---we may call this $\mathrm{SO}(3)$-pivotality.

Additionally, our examples are \emph{spherical} in the sense of \cite[Def.\ 2.3.2]{douglas2018fusion}, allowing for unambiguous evaluation of skeins in the oriented 3-sphere\footnote{This is conjecturally related to $\mathrm{SO}(4)$-pivotality.}.
This allows our monoidal 2-categories to be used as input for fully local, partially defined 4-dimensional topological quantum field theories \cite{HRWBorderedInvariantsKh,WedrichSurveyLinkHomTQFT,leonsurfaceskein}. 
\end{remnono}

\begin{remnono}[Graded $\kb$-linear vs. $\kb$-linear with $\Z$-action] In the context of our \MainThm\,  the local enrichment in graded $\kb$-modules can also be traded for $\kb$-linear hom-categories with $\Z$-actions by grading-shift autoequivalences, yielding objects of $\Alg_{\Eone}(\Cat[\Cat[\kb\Mod]^{B\Z}])$. In this setting, both families of examples from \MainThm\, still have duals for all objects and adjoints for all 1-morphisms. However, since left and right adjoints of 1-morphisms typically differ by a non-trivial grading shift, these examples do not inherit the duality structures from their locally graded $\kb$-linear counterparts.
\end{remnono}

\subsection{Approach}

While our \MainThm\ focuses on the two most prominent families of examples arising
in link homology, the constructions presented in this paper are considerably more
general.
At a conceptual level, the only inputs required are:
\begin{enumerate}
\item A choice of \emph{combinatorial type} of foam. For example, for \MainThm\ (2), we consider foams with oriented facets, which meet cyclically ordered at trivalent seams and which intersect in tetravalent vertices, such that facets are labeled by natural numbers, which satisfy a flow condition at seams, and decorated by certain $\kb$-algebra elements. We describe such trivalent foams, a subclass of nilvalent foams, and appropriate labeling data in \S\ref{sec:closedfoams}. Other combinatorial types, such as those discussed in \cite{khovanov2025lecturessl3foamslink} for $\mathrm{SL}(3)$ and  $\mathrm{SO}(3)$, see also \cite{MR3880205}, can be treated mutatis mutandis.
\item A choice of compatible closed foam evaluation formula, satisfying certain multiplicativity requirements. For example, the combinatorial Robert--Wagner foam evaluation formula \cite{RoW} for \MainThm\ (2). An alternative would be the Kapustin-Li formula \cite{MR2039036} adapted by Khovanov--Rozansky to the setting of foams~\cite{MR2322554}, see also Mackaay--Sto\v{s}i\'{c}--Vaz~\cite{MSV}.
\item A rule for allowed embeddings of the chosen type of foams in oriented 3-manifolds. In our running example we require a compatibility between cyclic ordering of facets, induced seam orientation and ambient 3-manifold orientation via the left-hand rule.
\end{enumerate} 
From such data, we obtain a locally linear semistrict monoidal $2$-category in two stages, which are carried out in Sections~\ref{sec:TQFT} and \ref{sec:2cat} respectively, and which we now describe informally.

\subsection{From closed foam evaluation to a TQFT for webs and foams}
The first stage is to extract a TQFT for webs and foams from the closed foam evaluation formula by means of the \emph{universal construction}, see e.g.~\cite{BHMV}. 
\smallskip

In \S\ref{sec:boundaryweb} we start by extending the chosen combinatorial type of closed foams to \emph{foams with boundary}, the latter encoded as appropriate type of \emph{closed web}. After partitioning the boundary web into a disjoint union of \emph{source} and \emph{target webs}, we can consider foams as morphisms between closed webs in an appropriately defined symmetric monoidal category; see Proposition~\ref{prop:foamcat}. The composition of foams is given by a combinatorial gluing construction and the symmetric monoidal structure is disjoint union. 
\smallskip 

In \S\ref{sec:labeledfoams} we incorporate labelings and decorations on foams provided by a labeling datum\footnote{As defined in Definition~\ref{def:labdata}.} $\labdata$ to obtain a $\kb$-linear symmetric monoidal category $\PreFoamCat[\labdata]$ of \emph{$\labdata$-prefoams} between closed  $\labdata$-webs in Proposition~\ref{prop:labfoamcat}. The prefix \emph{pre} indicates that this freely generated category of prefoams is still missing relations.
\smallskip

The relations are obtained from the closed foam evaluation formula in \S\ref{sec:univ} by means of the universal construction, as we now describe. For closed $\labdata$-webs $W$ and $W'$ we denote the $\kb$-module of $\labdata$-prefoams between them by $\PreFoamCat[\labdata](W,W')$. Prefoam composition and the closed foam evaluation formula give rise to a $\kb$-bilinear pairing
\[\PreFoamCat[\labdata](W,W') \times \PreFoamCat[\labdata](W',W) \to \kb\]
whose radical $I(W,W')\subset \PreFoamCat[\labdata](W,W')$ we interpret as subspace of foam relations. In fact, these subspaces define an ideal, and the quotient 
\[
\FoamCat[\labdata]:= \PreFoamCat[\labdata]/I
\]
is the desired symmetric monoidal $\kb$-linear foam category.  

\smallskip
We then study additional properties of the closed foam evaluation formula concerning multiplicativity. In favorable cases, the representable functor for the empty web $\emptyset$

\[
\webeval\colon \Hom(\emptyset, -) \colon \FoamCat[\labdata] \to \kb\Mod 
\]
is strongly symmetric monoidal. After precomposing with the canonical quotient functor from $\PreFoamCat[\labdata]$, we obtain a symmetric monoidal functor 
\begin{equation}
\label{eq:foamTQFT}
\unicon \colon \PreFoamCat[\labdata] \to \kb\Mod,
\end{equation} which, alongside $\webeval$, is the second incarnation of the desired TQFT for webs and foams with labeling data $\labdata$. Cases in which the foam evaluation is compatible with reversing all orientations on foams are studied in \S\ref{sec:pairing}.

\smallskip 

In \S\ref{sec:Frob} we observe that the TQFT for webs and foams contains for each fixed label in $\labdata$ an ordinary 2-dimensional TQFT, defined for oriented cobordisms between closed oriented 1-manifolds. Correspondingly, there is an associated commutative Frobenius algebra over $\kb$. In the graded case, we extract the homogeneous degrees of the unit and counit map, which play an important role in the following as grading shifts employed when defining 2-morphism spaces in the associated monoidal 2-category.

\subsection{Via holography to monoidal 2-categories}\label{sec:hol}
A guiding principle behind our approach is a holographic viewpoint. Spaces of 2-morphisms in a monoidal 2-category are indexed by a pair of parallel 1-morphisms $S,T\colon s\to t$, serving as source $S$ and target $T$. In the graphical calculus of monoidal 2-categories, it is convenient to consider $S$ and $T$ to decorate the bottom and top face of a 3-dimensional cube, with their target and source objects $t$ and $s$ connected along the left and right faces, respectively. In the case our our desired monoidal 2-category $\Foamssmonbicat{\labdata}$, these decorations amount to a closed $\labdata$-web, as depicted in Figure~\ref{fig:cubecoordinatedconvention}.c). Thus embedded in the boundary of the cube, the web may serve as boundary condition for properly embedded $\labdata$-foams, as pictured in Figure~\ref{fig:cubecoordinatedconvention}.c). Alternatively, it can also be considered as \emph{abstract} closed $\labdata$-web, pictured in Figure~\ref{fig:cubecoordinatedconvention}.d), which we denote by $\abs{\inthom(S,T)}$.

\begin{figure}[ht]
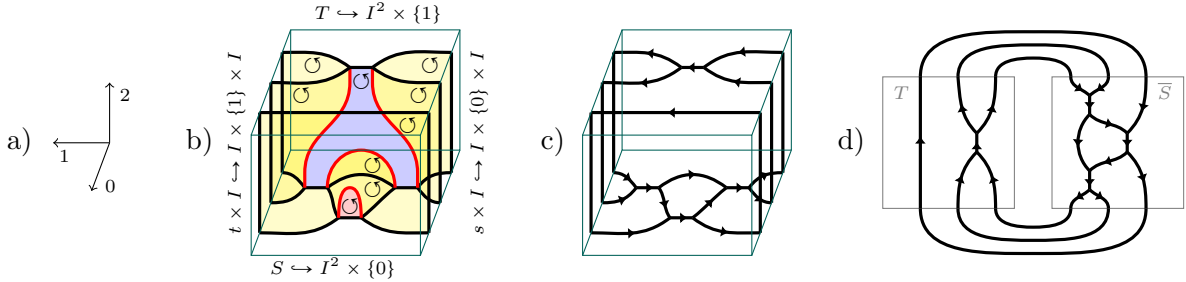

    $\mathrm{a)}$ \foamcoord$\quad$
    $\mathrm{b)}$\foamGraphsimp$\quad$
    $\mathrm{c)}$ \foamGraphsimpnofill$\quad$
    $\mathrm{d)}$ \foamGraphsimpflat
    \caption{a) Coordinate system for the graphical calculus, b) An embedded trivalent foam as 2-morphism between parallel webs $S,T\colon s \to t$. c) The associated boundary web. d) A flattened depiction of the boundary web. Throughout we suppress labels and cyclic orderings.}\label{fig:cubecoordinatedconvention}
\end{figure}

In our construction of $\Foamssmonbicat{\labdata}$ we \emph{define} the space of 2-morphisms from $S$ to $T$ by setting
\[
\Hom_{\Foambicathom{\labdata}{t}{s}}(S,T):= \unicon(\abs{\inthom(S,T)})
\]
where $\unicon$ is the TQFT for abstract $\labdata$-labeled webs and foams from \eqref{eq:foamTQFT}. This construction is described in detail in \S\ref{sec:vertcomp} once the sets of objects and 1-morphisms of $\Foamssmonbicat{\labdata}$ have been defined in \S\ref{subsec:objsAndOneMorphs}.

\begin{remnono}[Skeins and skein relations]
Given any $\labdata$-foam embedded in the 3-dimensional cube with boundary conditions $S,T\colon s\to t$ as described above, we can obtain an element of the $\kb$-module $\Hom_{\Foambicathom{\labdata}{t}{s}}(S,T)$ by \emph{evaluating} the foam as follows. First we forget the embedding and regard the foam as an abstract $\labdata$-prefoam, namely as a morphism from the empty web $\emptyset$ to the boundary web $\abs{\inthom(S,T)}$ in $\PreFoamCat[\labdata]$. Applying the TQFT $\unicon$ yields a $\kb$-linear map $\kb\to \Hom_{\Foambicathom{\labdata}{t}{s}}(S,T)$, whose value at $1\in \kb$ can be regarded as the 2-morphism represented by the foam.

In the examples of \MainThm, the values of embedded foams span $\Hom_{\Foambicathom{\labdata}{t}{s}}(S,T)$ for any pair $S,T\colon s\to t$. This gives rise to a skein-theoretic description of the 2-morphism modules. Specifically, $\Hom_{\Foambicathom{\labdata}{t}{s}}(S,T)$ can be presented as quotient of the free $\kb$-module with basis given by all embedded $\labdata$-foams with boundary conditions $S,T\colon s\to t$---the \emph{skeins}---modulo the kernel of the $\kb$-linear map to $\Hom_{\Foambicathom{\labdata}{t}{s}}(S,T)$ defined on the basis by the evaluation above---the \emph{skein relations}.

Note that such a skein-theoretic description is a consequence of our construction, not part of the construction---$\labdata$-foams embedded in the 3-dimensional cube only play a motivational part in this paper.
\end{remnono}

\smallskip

The vertical composition, horizontal composition, and monoidal structure of $\Foamssmonbicat{\labdata}$ are constructed in \S\ref{sec:vertcomp}, \S\ref{sec:horcomp}, and \S\ref{sec:tencomp} respectively. The non-trivial part of these three constructions is on the level of $\kb$-modules of 2-morphisms and in all three cases follows the same underlying idea, which is inspired by Khovanov's construction of arc rings \cite{MR1928174}, and the graphical calculus from Figure~\ref{fig:cubecoordinatedconvention}. It is useful to consider the coordinate system displayed in Figure~\ref{fig:cubecoordinatedconvention}.a). The 0-direction is the direction of the tensor product $\boxtimes$, the 1-direction is the direction of horizontal composition $\hcomp$ in which 1-morphisms compose, the 2-direction is the direction of vertical composition $\vcomp$. On the level of 2-morphisms, these three operations correspond, respectively, to stacking cubes (containing foams) behind each other, left-and-right of each other, and on top of each other. 

\smallskip
For example, the vertical composition of 2-morphisms $S\To T$ and $T\To U$ corresponds to stacking cubes on top of each other along the common facet containing the web $T$, and rescaling the vertical coordinate to obtain a new cube with bottom web $S$ and top web $U$. On the level of abstract boundary webs, this corresponds to a transformation
\[
    \abs{\inthom(T,U)} \sqcup \abs{\inthom(S,T)} \to \abs{\inthom(S,U)},
    \]
which we can realize by a carefully constructed abstract $\labdata$-foam, which contracts the two copies of $T$. Under the (symmetric) monoidal TQFT $\unicon$ for $\labdata$-labeled webs and foams, this corresponds precisely to the desired $\kb$-linear vertical composition map:
\[
\Hom_{\Foambicathom{\labdata}{t}{s}}(T,U) \otimes_{\kb} \Hom_{\Foambicathom{\labdata}{t}{s}}(S,T) \to \Hom_{\Foambicathom{\labdata}{t}{s}}(S,U)
\]

After explicitly constructing abstract foams that implement the horizontal composition, the monoidal structure, the identity 2-morphisms, and certain other structural 2-morphisms for the semistrict monoidal structure, all verifications of the axioms of a semistrict monoidal 2-category \cite[Lemma 4]{BaezNeuchlHDAI} can be reduced to exhibiting explicit isomorphisms between abstract $\labdata$-prefoams. For example, the vertical composition turns out to be associative because both ways of composing a triple of composable 2-morphisms can be implemented by the map induced under $\unicon$ by the same underlying $\labdata$-prefoam.
\smallskip

In \S\ref{subsec:locallygraded}, we discuss the necessary modifications in the graded case. These amount to certain grading shifts in the definition of the graded $\kb$-modules of 2-morphism, which ensure that the three composition operations $\boxtimes$, $\hcomp$, and $\vcomp$ end up grading-preserving.

\begin{remnono}[Alternative 3-categorical formalizations] By delooping, any semistrict monoidal 2-category can be considered as a weak 3-category with a single object, more precisely, a \emph{Gray monoid} \cite{MR3076451}\footnote{Conversely, every weak 3-category is equivalent to a Gray category by \cite{GPStricat}.}. 

Figure~\ref{fig:fourmodels} shows alternative models of 3-categorical structures, for which a holographic approach using a TQFT of webs and foams can be used to construct explicit examples. They have in common that the top-dimensional morphisms spaces are parametrized by a topological 3-ball with a web embedded transversely to a given stratification in the boundary 2-sphere. For a disklike 3-category in the sense of Morrison--Walker~\cite{MWBlob}, the stratification can be interpreted as trivial. For a canopolis in the sense of Bar-Natan~\cite{BN2}, the stratification is induced by realizing the 3-ball as a  \emph{can} $B^2\times I$ and one restricts to webs that are vertical on $S^1 \times I$. For a monoidal 2-category or a triple category (the 3-fold analog of a double category), the stratification is inherited from the product stratification of the cube $I^3$ with only difference that in a monoidal 2-category webs avoid the front and back faces and are vertical in the side faces.

\begin{figure}[ht]
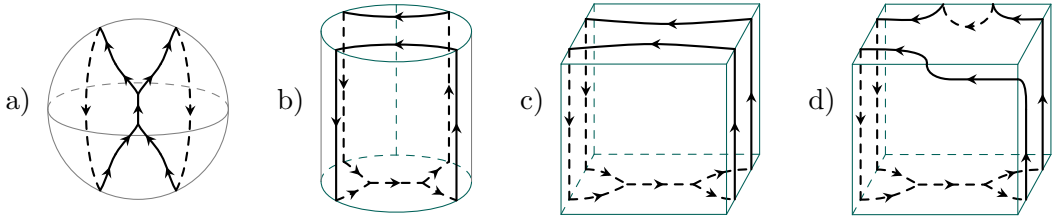
 
\[
\mathrm{a)} \;\;\disklikeThreeCat\;\;\;\;\;\;
\mathrm{b)} \;\;\canopolis\;\;\;\;\;\;
\mathrm{c)} \;\;\monoidalTwoCat\;\;\;\;\;\;
\mathrm{d)}  \;\;\tripleCat\]
  \caption{Example shapes of top-dimensional morphisms in a) a disklike 3-category, b) a canopolis \cite{BN2}, c) a monoidal 2-category, d) a triple category.}
  \label{fig:fourmodels}
\end{figure}
\noindent Other stratifications also have their use cases, for example those coming from cylinders over polygons, see \cite{HRWBorderedInvariantsKh}. 

\end{remnono}

Finally, in \S\ref{subsec:duals} we show that all objects of $\Foamssmonbicat{\labdata}$ admit duals and all 1-morphisms admit left and right adjoints. Furthermore, we show that taking duals and adjoints can be packaged into two duality operations $\#$ and $*$ on $\Foamssmonbicat{\labdata}$, which roughly\footnote{See Remark~\ref{rem:strict}.} correspond in the graphical calculus to rotating 2-morphisms by an angle of $\pi$ in the $(0,1)$- and $(1,2)$-planes respectively. These operations satisfy compatibilities encoded by the notion of a \emph{spatial duality structure} in the sense of \cite{barrett2018gray}. This enables a graphical calculus for $\Foamssmonbicat{\labdata}$ involving oriented (rather than framed) ambient 3-cubes and skein theory in ambient oriented 3-manifolds.

\begin{remnono}
	Relative to the graphical language of \cite{barrett2018gray}, in Figure~\ref{fig:cubecoordinatedconvention} we relabeled the coordinate
	axes $w\mapsto 0$, $x\mapsto 1$, $y\mapsto 2$ and display our cubes rotated
	by $\pi$ in our $(1,2)$-plane. That is: the first factor of a monoidal product appears in front of the second one. Horizontal composition puts the later 1-morphism to the left of the previous one and vertical composition puts the later 2-morphism on top of the previous one. We also adopt the convention \cite[Definition 2.1.3]{DSPS} that tensor product corresponds to opposite composition upon delooping. With respect to the three composition directions, we have thus implemented the right-hand rule. Still, we use the graphical language adapted to the tensor product, following \cite{barrett2018gray}.
	\end{remnono}

\subsection*{Acknowledgments}
We would like to thank David Reutter and Kevin Walker for helpful discussions.

\subsection*{Funding}

The authors acknowledge support from the Deutsche
Forschungsgemeinschaft (DFG, German Research Foundation) under Germany's
Excellence Strategy - EXC 2121 ``Quantum Universe'' - 390833306 and the
Collaborative Research Center - SFB 1624 ``Higher structures, moduli spaces and
integrability'' - 506632645.

\section{Closed foam evaluation}
\label{sec:closedFoamEval}

	\subsection{Closed foams}
    \label{sec:closedfoams}
    \newcommand{\powerset}{\mathrm{P}}
	
\begin{definition}
\label{def:combclosedfoam}
    A \emph{combinatorial closed foam} consists of the following data:
    \begin{itemize}
        \item A finite set $V$ of \emph{vertices}. 
        \item A finite set $E$ of \emph{seams}, partitioned as $E=E_{\text{loop}}\sqcup E_{\text{arc}}$ into \emph{loops} and \emph{oriented arcs}. \item Associated to each arc $e\in E_{\text{arc}}$, a set $V(e)\subset V$ of exactly two \emph{adjacent} vertices. Each arc $e$ is given an orientation by designating one of the vertices of $V(e)$ as \emph{source} $s(e)$ and the other one as \emph{target} $t(e)$. For every vertex $v\in V$, we write $E(v)\subset E$ for the subset of arcs adjacent to $v$.

        For each $e\in E_{\text{arc}}$, we define $e^-$ to be $e$ with opposite orientation, i.e.\ $s(e^-)=t(e)$ and $t(e^-)=s(e)$, and we set $e^+:=e$. Then, we let $W(E)$ denote the set of cyclic words of positive length formed as follows: the words of length $1$ consist of the loops in $E_{\text{loop}}$. For length $\ell>1$, $W(e)$ contains the tuples $(e_1^{\varepsilon_1},\dots e_\ell^{\varepsilon_\ell})$, considered up to cyclic permutation, where $\varepsilon_i\in\{-,+\}$, $e_i\in E_{\text{arc}}$ for all $i$, and the $e_i^{\varepsilon_i}$ are composable in the sense that $t(e_i^{\varepsilon_i})=s(e_{i+1}^{\varepsilon_{i+1}})=:v_{i+1}$ for $1\leq i<\ell$ and $t(e_\ell^{\varepsilon_\ell})=s(e_1^{\varepsilon_1})=:v_1$, and the vertices $v_1,\ldots,v_\ell$ are all distinct.
        We call elements of $W(E)$ of length $\ell >1$ \emph{oriented}, and the remaining ones \emph{unoriented}.
        
        \item A finite set $F$ of 
        \emph{facets}, partitioned as $F=\bigsqcup F_{g,n}$ according to genus $g\in \N_0$ and number of boundary components $n\in \N_0$. 
        \item Associated with every $f\in F_{g,n}$ is a set of \emph{boundary words} $\partial(f)\in \powerset(W(E))$ of cardinality $|\partial(f)|=n$. The elements of $\partial(f)$ are cyclic words $\partial_i(f)\in W(E)$, such that every seam $e\in E$ and every vertex $v\in V$ 
        is traversed in $\partial_i(f)$ for at most one $i$.
        If $e\in E$ is traversed exactly once in this way, we write $f\in F(e)\subset F$ and $e\in E(f)\subset E$. Likewise, if $v\in V$ is traversed exactly once, we write $f\in F(v)\subset F$ and $v\in V(f)\subset V$.
        \item For every seam $e\in E$ a cyclic ordering of the set $F(e)$ of facets adjacent to $e$.
    \end{itemize}
\end{definition}

\begin{remark} \label{rem:cyclicSign}
    Note that, when $\ell>1$, the cyclic word $(e_1^{\varepsilon_1},\dots e_\ell^{\varepsilon_\ell})\in W(E)$ is completely determined by the tuple $(e_1,\ldots,e_\ell)$ and the sign $\varepsilon_i$ of any of its arcs $e_i$. If $\ell>2$, even more is true: $(e_1^{\varepsilon_1},\dots e_\ell^{\varepsilon_\ell})\in W(E)$ is completely determined by just the tuple $(e_1,\ldots,e_\ell)$. Therefore, we will omit the signs $\varepsilon_1,\ldots,\varepsilon_\ell$ in this case.
\end{remark}

In the following, we gather a few useful concepts relating to combinatorial closed foams:
\begin{definition}\label{def:combclosedfoam2}~
\begin{itemize}
    \item Consider an oriented boundary word $\partial_i(f)=(e_1^{\varepsilon_1},\ldots ,e_\ell^{\varepsilon_\ell})\in W(E)$ associated to a facet $f$. For any $1\leq j\leq\ell$, we say that $f$ \emph{induces the orientation} of the arc $e_j$ if $\varepsilon_j=+$, and that $f$ \emph{induces the opposite orientation} on $e_j$ if $\varepsilon_j=-$.
    \item Given a combinatorial closed foam, the \emph{link of a vertex} $v\in V$ is the graph $\mathrm{Link}(v)$ with nodes $e\in E(v)$ and edges $f\in F(v)$, which are declared to be adjacent if and only if $f\in F(e)$. 
    \item A combinatorial closed foam is called \emph{trivalent} if the following conditions are satisfied:
    \begin{enumerate}[(i)]
        \item For every $e\in E$ we have $|F(e)|=3$.
        \item For every $v\in V$ the graph $\mathrm{Link}(v)$ is isomorphic to the 1-skeleton of a tetrahedron. In particular, we have $|E(v)|=4$ for every $v\in V$.
        \item For every arc $e\in E_{\text{arc}}$, exactly two out of the three adjacent facets, denoted by $f_1$ and $f_2$, induce the orientation of $e$, while the third, denoted by $f_3$, induces the opposite orientation.
    \end{enumerate} 
    \item A combinatorial closed foam is called \emph{nilvalent} if $E=\emptyset$ and $V=\emptyset$.
    \item Given a combinatorial closed foam $\Sigma$, its \emph{mirror} $\Sigma^m$ is obtained by reversing all cyclic orderings and its \emph{reverse} $\Sigma^r$ by reversing the orientations of all arcs and reversing boundary words, namely 
    \[
    E_{\text{arc}}\ni e \;\overset{r}{\mapsto}\; e^- \quad \text{ and } \quad W(E)\ni (e_1^{\varepsilon_1},\ldots ,e_\ell^{\varepsilon_\ell}) \;\overset{r}{\mapsto}\; (e_\ell^{-\varepsilon_\ell},\ldots ,e_1^{-\varepsilon_1}).
    \]
\end{itemize}   
\end{definition}

In the following, we will mostly consider nilvalent and trivalent foams. In a trivalent foam, there are exactly six facets adjacent to every vertex.

\begin{definition}[Labeling data for trivalent and nilvalent foams]
\label{def:labdata}
    Let $\kb$ be a commutative ring. A \emph{$\kb$-linear labeling datum} $\labdata$ for trivalent combinatorial foams consists of a semigroup $\labdata_1$ with associative multiplication denoted $+\colon \labdata_1\times \labdata_1 \to \labdata_1$, and for each $a\in \labdata_1$ the data of a commutative $\kb$-algebra $\Dec(a)$, called the \emph{algebra of decorations on an $a$-facet}. If the operation has a neutral element $0\in \labdata_1$, we additionally require $\Dec(0)\cong \kb$.

    A \emph{$\kb$-linear labeling datum} $\labdata$ for nilvalent foams consists of a set $\labdata_1$ and for each $a\in \labdata_1$ the data of a commutative $\kb$-algebra $\Dec(a)$.
\end{definition}

\begin{definition}[Combinatorial closed $\labdata$-foam] \label{def:labclosedfoam} Let $\kb$ be a commutative ring and $\labdata$ as in Definition~\ref{def:labdata}. A \emph{combinatorial closed $\labdata$-foam} consists of the following data:
    \begin{enumerate}
        \item A combinatorial closed foam as in Definition~\ref{def:combclosedfoam} with sets of vertices, seams, and facets denoted by $V$, $E$, and $F$ respectively.
        \item A labeling $\labelfunction\colon F \to \labdata_1$ of all facets by elements of $\labdata_1$.
         \item A decoration of the facets by an element $\Dec(F)\in \bigotimes_{f\in F} \Dec(\labelfunction(f))$. 
    \end{enumerate}
    In the nilvalent case, this finishes the definition. In the trivalent case, the data is subject to the following requirements:
    \begin{itemize}
         \item For every arc seam $e\in E_{\text{arc}}$, we refine the cyclic order on $F(e)$ to a total order $(f_1,f_2,f_3)$ by requiring that $f_3$ is the edge with orientation incompatible with the others. Then, we require that the labels of $f_1,f_2,f_3$ satisfy:
         \[\labelfunction(f_1) + \labelfunction(f_2) = \labelfunction(f_3).\]
         Moreover, the same requirement must hold for the facets adjacent to loop seams $e\in E_{\text{loop}}$, for appropriate assignments of names $f_1,f_2,f_3$.
         
        \item For every vertex $v\in V$ exactly two of the adjacent seams $e_1,e_2$ are oriented into the vertex and two seams $e_3,e_4$ are oriented out of the vertex, so that the six adjacent facets $f_1,\dots, f_6$ have one of the following two label configurations: 
        \begin{align*}
        \labelfunction(f_1) + \labelfunction(f_2) &\overset{e_1}{=} \labelfunction(f_4)    && 
        \labelfunction(f_2) + \labelfunction(f_3) \overset{e_1}{=} \labelfunction(f_4)\\ 
        \labelfunction(f_4) + \labelfunction(f_3) &\overset{e_2}{=} \labelfunction(f_6)&& 
        \labelfunction(f_1) + \labelfunction(f_4) \overset{e_2}{=} \labelfunction(f_6)\\
        \labelfunction(f_2) + \labelfunction(f_3) &\overset{e_3}{=} \labelfunction(f_5) && 
        \labelfunction(f_1) + \labelfunction(f_2) \overset{e_3}{=} \labelfunction(f_5)\\
        \labelfunction(f_1) + \labelfunction(f_5) &\overset{e_4}{=} \labelfunction(f_6) && 
        \labelfunction(f_5) + \labelfunction(f_3) \overset{e_4}{=} \labelfunction(f_6)
        \end{align*}
        Here, $f_1,f_2,f_3$ induce the orientations of $e_1,e_2,e_3,e_4$, while $f_6$ induces the opposite orientation. Note that, in each equation, the ordering of the summands represents the total order of the three facets adjacent to the corresponding seam, as per the first point. 
    \end{itemize}    
\end{definition}

\begin{remark}
Definitions~\ref{def:labdata} and~\ref{def:labclosedfoam} can be straightforwardly generalized to classes of foams of other valency or of other orientation behavior. Since the trivalent and nilvalent cases suffice to cover our intended main examples, we will not pursue such generalizations in this paper. 
\end{remark}

\subsection{Evaluation}
Fix a commutative ring $\kb$ and a $\kb$-linear labeling datum $\labdata$. The set of combinatorial closed $\labdata$-foams with the same underlying data (1) and (2) in Definition~\ref{def:labclosedfoam} is parametrized by the $\kb$-module of decorations on the facets (3). We let 
\[\HomPreFoamCat[\labdata](\emptyset,\emptyset)\] 
denote the $\kb$-module obtained as the coproduct of these $\kb$-modules as (1) and (2) range over all combinatorial closed foams and labelings. The notation foreshadows \Cref{def:labfoammorphism}.

    \begin{definition} \label{def:eval} 
    Let $\kb$ be a commutative ring and $\labdata$ a $\kb$-linear labeling datum.       An \emph{evaluation} of closed $\labdata$-foams is a $\kb$-linear function 
         \[
         \eval \colon \HomPreFoamCat[\labdata](\emptyset,\emptyset) \to \kb
         \]
         associating to every combinatorial closed $\labdata$-foam an element of the commutative ring $\kb$.
         
         An evaluation is called \emph{multiplicative} if:
         \begin{enumerate}
             \item $\eval(\emptyset)=1\in \kb$, where $\emptyset$ denotes the empty combinatorial closed $\labdata$-foam;
             \item $\eval(\Sigma_1\sqcup \Sigma_2)=\eval(\Sigma_1)\cdot\eval(\Sigma_2)$ for any pair of combinatorial closed $\labdata$-foams $\Sigma_1,\Sigma_2$.
         \end{enumerate}
         
    \end{definition}

        \begin{definition}
            \label{def:involutive}
Let $\iota\colon \kb \to \kb$ be an involution of $\kb$ (i.e.~a ring automorphism squaring to the identity) and suppose that the labeling data $\labdata$ includes for every $a\in \labdata_1$ the data of a compatible involution on the $\kb$-algebra $\Dec(a)$, also denoted by $\iota$. 

        Then we define an involution on combinatorial closed $\labdata$-foams 
        \[
        -^{r}\colon \HomPreFoamCat[\labdata](\emptyset,\emptyset) \to \HomPreFoamCat[\labdata](\emptyset,\emptyset)
        \]
        acting by $\Sigma\to \Sigma^r$ (from \Cref{def:combclosedfoam2}) on underlying foams and by the involution $\iota$ on all decorations. The evaluation $\eval$ is called ($\iota$-)\emph{involutive} if for any combinatorial closed $\labdata$-foam $\Sigma$ we have
         \[
            \eval(\Sigma^r)=\iota(\eval(\Sigma)).
         \]
        \end{definition}

        \begin{example}[Khovanov--Bar-Natan decorated surfaces] \label{ex:BNeval} Let $A$ be a commutative Frobenius algebra over a commutative ring $\kb$ with multiplication $m\colon A\otimes A \to A$, comultiplication $\Delta\colon A \to A \otimes A$, Frobenius trace $\epsilon \colon A \to \kb$ and handle operator $H:= m\circ \Delta\colon A \to A$. 
        
        A $\kb$-linear labeling datum $\labdata^A$ for nilvalent foams is specified by letting $\labdata^A_1:=\{*\}$ be the one-element set and declaring $\Dec(*):=A$.

        The data of a combinatorial closed $\labdata^A$-foam is then nothing but a finite family $F$ of (closed) facets, decorated by an element of the tensor power $A^{\otimes F}$. The only invariant of a closed facet $f\in F$ is its genus $g_f\in \N_0$, so we can equivalently encode such a foam as a pair $\left( (g_f)_{f\in F},a\right) \in \N_0^F\times A^{\otimes F}$. An evaluation is given by the $\kb$-linear extension of:
        \[
        \eval \colon \labclosedFoamModule{\labdata^A} \to \kb, \quad 
        \left( (g_f)_{f\in F},\otimes_{f\in F} a_f)\right)  \mapsto \prod_{f\in F} \epsilon(H^{g_f} a_f)\in \kb
        \]
    \end{example}

    \begin{example}[Robert--Wagner $\mf{gl}_N$-foams] \label{ex:glneval} Let $N\in \N$ and $\kb=\Z$. The labeling datum $\labdata^N$ for trivalent foams has as underlying semigroup $\labdata^N_1:=(\N_0,+)$ the natural numbers with addition, and the algebra of decorations on an $n$-labeled facet for $n\in\N$ is given by the ring $\Dec(n)=\Z[x_1,\ldots,x_n]^{S_n}$ of symmetric polynomials in $n$ variables.
    
        Then combinatorial closed $\labdata^{N}$-foams correspond to the $\mf{gl}_N$-foams of Robert--Wagner and a closed foam evaluation formula is provided by the main theorem of \cite{RoW}.

        The evaluation of $\mf{gl}_2$-foams is due to Blanchet~\cite{Bla}.
        
    \end{example}

        \begin{example}[Robert--Wagner equivariant $\mf{gl}_N$-foams] \label{ex:GLneval} Let $N\in \N$ and $\mathbb{K}=\Z$. The Robert--Wagner foam evaluation generalizes to the equivariant setting, which is defined over the base ring $\kb:=H^*_{\mathrm{GL}(N)}(*)\cong \Z[X_1,\dots,X_N]^{S_N}$, the $\mathrm{GL}(N)$-equivariant cohomology of the point. The corresponding labeling datum $\labdata^{\mathrm{GL}(N)}$ for trivalent foams again has $\labdata^{\mathrm{GL}(N)}_1:=(\N_0,+)$ and $\Dec(n)=\kb[x_1,\ldots,x_n]^{S_n}$ for $n\in \N_0$. A closed foam evaluation formula is provided by the main theorem of \cite{RoW}. Indeed, the evaluation from \Cref{ex:glneval} results from this equivariant evaluation by setting the variables $X_1,\dots, X_n$ to zero.
        
    \end{example}

\section{TQFT for webs and foams}
\label{sec:TQFT}

\subsection{Foams with boundary and webs}
\label{sec:boundaryweb}

We now extend Definition~\ref{def:combclosedfoam} to the case of foams with boundary.

\begin{definition}
\label{def:combboundaryfoam}
    A \emph{combinatorial foam with boundary} consists of:
    \begin{itemize}
        \item A finite set $V$ of \emph{vertices} as in Definition~\ref{def:combclosedfoam}. 
        \item A finite set $E=E_{\text{loop}}\sqcup E_{\text{arc}}$ of \emph{seams} as in Definition~\ref{def:combclosedfoam}.
        \item An additional finite set $V^{\partial}$ of \emph{boundary vertices}. We set $\overline{V}=V\sqcup V^{\partial}$.
        
        \item An additional finite set $E^{\partial}$ of \emph{boundary edges}, partitioned into \emph{boundary loops} and \emph{oriented boundary arcs}, $E^{\partial}=E^{\partial}_{\text{loop}}\sqcup E^{\partial}_{\text{arc}}$, so that we can write:
        \[
        \overline{E}_{\text{loop}} := E_{\text{loop}}\sqcup E^{\partial}_{\text{loop}}
        , \quad
        \overline{E}_{\text{arc}} := E_{\text{arc}}\sqcup E^{\partial}_{\text{arc}}
        ,\quad 
        \overline{E}:=E\sqcup E^{\partial}.\]
        
        \item For each $e\in E_{\text{arc}}$, a set $V(e)\subset \overline{V}$ of exactly two vertices \emph{adjacent} to $e$, one designated as source and the other as target. For every $v\in \overline{V}$, we then write $E(v)\subset E_{\text{arc}}$ for the subset of seams adjacent to $v$.
        \item For each $e\in E^{\partial}_{\text{arc}}$, a set $V^{\partial}(e)\subset V^{\partial}$ of exactly two boundary vertices \emph{adjacent} to $e$, one designated as source $s(e)$ and the other as target $t(e)$. 
        For every boundary vertex $v\in V^{\partial}$, we write $E^{\partial}(v)\subset E^{\partial}_{\text{arc}}$ for the subset of boundary edges adjacent to $v$ and $\overline{E}(v)= E(v)\sqcup E^{\partial}(v)\subset \overline{E}_{\text{arc}}$ for the subset of all seams and edges adjacent to $v$.

        As in Definition~\ref{def:combclosedfoam}, we consider the set $W(\overline{E})$ of cyclic words of positive length~$\ell$ formed by composable seams and edges, equipped with signs $+$ or $-$ when $\ell >1$. 
        More precisely, words of length one are formed by loops in $\overline{E}_{\text{loop}}$, while words of length $\ell>1$ are tuples $(e_1^{\varepsilon_1},\dots e_\ell^{\varepsilon_\ell})$, considered up to cyclic permutation, where $\varepsilon_i\in\{-,+\}$ and $e_i\in \overline{E}_{\text{arc}}$ for all $i$, and the $e_i^{\varepsilon_i}$ are composable and all have distinct sources. We additionally require that if $e_i\in E^\partial_{\text{arc}}$, it appears with sign $\varepsilon_i=+$.
        \item A finite set $F$ of \emph{facets}, partitioned as $F=\bigsqcup F_{g,n}$ as in Definition \ref{def:combclosedfoam}. 
        \item Associated with every $f\in F_{g,n}$ is a set of \emph{boundary words} $\partial(f)\in \powerset(W(\overline{E}))$ of cardinality $|\partial(f)|=n$. The elements of $\partial(f)$ are cyclic words $\partial_i(f)\in W(\overline{E})$, such that every $e\in \overline{E}$ and every $v\in \overline{V}$ is traversed in $\partial_i(f)$ for at most one $i$. If $e\in \overline{E}$ is traversed exactly once in this way, we write $f\in F(e)\subset F$ and $e\in \overline{E}(f)\subset \overline{E}$.
        \item For every seam $e\in E$ a cyclic ordering of the set $F(e)$ of facets adjacent to $e$.
        \item For $v\in V^\partial$ and $e\in E^\partial$ we require $|E(v)|=1$ and $|F(e)|=1$.
    \end{itemize}
    The set of combinatorial foams with boundary is denoted $\boundaryFoamSet$.
\end{definition}

Note that for each boundary vertex $v\in V^\partial$ the adjacent boundary edges $E^\partial(v)$ are in canonical bijection with the facets $F(e(v))$ adjacent to the unique seam $e(v)\in E$ that is adjacent to $v$, and hence also cyclically ordered. 
\smallskip

\begin{definition}~ \label{def:combboundaryfoam2}
    \begin{itemize}
        \item A combinatorial foam with boundary is called \emph{trivalent} if the following conditions hold:
        \begin{enumerate}[(i)]
            \item $|F(e)|=|E^\partial(v)|=3$ for every $e\in E$ and $v\in V^\partial$.
            \item For every $v\in V$ the graph $\mathrm{Link}(v)$ is isomorphic to the 1-skeleton of a tetrahedron.
            \item For every arc $e\in E_{\text{arc}}$, exactly two out of the three adjacent facets, denoted by $f_1$ and $f_2$, induce the orientation of $e$, while the third, denoted by $f_3$, induces the opposite orientation. Note that for an appropriate assignment of the names $f_1$ and $f_2$, this condition refines the cyclic order on $F(e)$ to a total order $(f_1,f_2,f_3)$.
        \end{enumerate}
        \item A combinatorial foam with boundary is \emph{nilvalent} if $E=\emptyset$, $V=\emptyset$, and $V^\partial=\emptyset$.
    \end{itemize}
\end{definition}

\begin{remark}\label{rem:trivalent}
    In a trivalent combinatorial foam with boundary, consider a facet $f$ with boundary $\partial(f)=\{\partial_1(f),\ldots,\partial_n(f)\}$ and suppose that a boundary component $\partial_i(f)$ traverses a boundary edge $e\in E^\partial$. If we call $s_1$ and $s_2$ the unique seams in $E(s(e))$ and $E(t(e))$ respectively, then it is an easy exercise to check that $(s_1^{\varepsilon_1},e,s_2^{\varepsilon_2})$ is a subword of $\partial_i(f)$, for some $\varepsilon_1,\varepsilon_2\in\{-,+\}$.
\end{remark}

\newcommand{\webs}[1]{\mathrm{Webs}_{#1}}

We now describe the boundaries of foams from Definition~\ref{def:combboundaryfoam}.

\begin{definition}\label{def:combclosedweb}
	A \emph{combinatorial closed web} consists of:
	\begin{itemize}
		\item A finite set $V^\partial$ of \emph{vertices}.
		\item A finite set $E^{\partial}$ of \emph{edges}, partitioned into \emph{loops} and \emph{oriented arcs}, $E^{\partial}=E^{\partial}_{\text{loop}}\sqcup E^{\partial}_{\text{arc}}$.
        \item  For each $e\in E^{\partial}_{\text{arc}}$, a set $V^{\partial}(e)\subset V^{\partial}$ of exactly two vertices \emph{adjacent} to $e$, one designated as source $s(e)$, the other as target $t(e)$. 
        For every vertex $v\in V^{\partial}$, we write $E^{\partial}(v)\subset E^{\partial}_{\text{arc}}$ for the subset of edges adjacent to $v$.
        \item For every vertex $v\in V^{\partial}$ a cyclic ordering of the set $E^{\partial}(v)$ of edges adjacent to $v$.
	\end{itemize}
Note that the disjoint union of two webs is again a web.
The set of combinatorial closed webs is denoted $\WebSet$.
\end{definition}
\smallskip

\begin{definition}~ \label{def:combclosedweb2}
    \begin{itemize}
        \item A combinatorial closed web is called \emph{trivalent} if, for every $v\in V^\partial$, $|E^\partial(v)|=3$ and exactly two out of the three adjacent arc edges have compatible orientation w.r.t.\ $v$. If they both have $v$ as their source, then $v$ is called a \emph{split vertex}; otherwise they both have $v$ as their target and $v$ is called a \emph{merge vertex}. 
        \item A combinatorial closed web is called \emph{nilvalent} if $V^\partial=\emptyset$ and hence also $E^{\partial}_{\text{arc}}=\emptyset$.
        \item For a given web $W$, we consider the \emph{mirror} $W^m$, obtained by reversing all cyclic orderings, and the \emph{reverse} $W^r$, obtained by reversing the orientations on all arc edges.
        Note that these operations are involutions and that split vertices in $W$ turn into merge vertices in 
        and vice versa.
    \end{itemize}
\end{definition}

Every combinatorial foam with boundary from Definition~\ref{def:combboundaryfoam} defines a combinatorial closed web. 
Indeed, we have already observed, that the boundary edges adjacent to a fixed boundary vertex inherit a cyclic ordering from the cyclic ordering of the corresponding facets around their shared seam.
Moreover, note that, by construction, for every boundary arc $e\in E_{\text{arc}}^\partial$, the unique adjacent facet $f\in F(e)$ induces the orientation of $e$. Along with condition (iii) of \Cref{def:combboundaryfoam2}, this implies that the boundary of a trivalent foam with boundary is a trivalent web.
The map that associates to each combinatorial foam with boundary its boundary web will be denoted:
\[
\partial\colon  \boundaryFoamSet \to \WebSet
\]

\begin{definition}
\label{def:foammorphism}
    Given webs $W_1,W_2\in \WebSet$, a foam $\Sigma\colon W_1\to W_2$ is defined to be a foam $\Sigma\in \boundaryFoamSet $ with boundary $\partial(\Sigma)=W_2 \sqcup W_1^r$.
\end{definition}

Note that a foam $\Sigma$ with boundary $\partial(\Sigma)=W_2 \sqcup W_1^r$ can be considered both as a foam $W_1\to W_2$ or as a foam $W_2^r\to W_1^r$.

\begin{construction}[Composition of foams]
\label{constr:foamcomposition}
    
    Given two foams $\Sigma_1\colon W_1\to W_2$ and $\Sigma_2\colon W_2\to W_3$ that are both either trivalent or nilvalent, let $\overline{V}(\Sigma_i)$, $\overline{E}(\Sigma_i)$ and $F(\Sigma_i)$, for $i=1,2$, be their sets of vertices, seams and facets. Note that $V^\partial(\Sigma_i)=V^\partial_i\sqcup V^\partial_{i+1}$ and $E^\partial(\Sigma_i)=E^\partial_i\sqcup E^\partial_{i+1}$, where $V^\partial_j$ and $E^\partial_j$ are respectively the sets of vertices and edges of $W_j$.
    We construct the trivalent (respectively nilvalent) combinatorial foam $\Sigma_2\circ\Sigma_1\colon W_1\to W_3$ as follows:
    \begin{itemize}
        \item The set of vertices is $V=V(\Sigma_1)\sqcup V(\Sigma_2)$.
        \item The sets of boundary vertices and boundary edges are, respectively, $V^\partial=V^\partial_1\sqcup V^\partial_3$ and $E^\partial=E^\partial_1\sqcup E^\partial_3$.
        \item The set $E=E_{\text{loop}}\sqcup E_{\text{arc}}$ of seams consists of the following:
         \begin{align*} 
         E_{\text{loop}} &:= E_{\text{loop}}(\Sigma_1)\sqcup E_{\text{loop}}(\Sigma_2) \sqcup N_{\text{loop}}
         \\
         E_{\text{arc}} &:= \{e\in E_{\text{arc}}(\Sigma_1)\sqcup E_{\text{arc}}(\Sigma_2)\mid V(e)\cap V^\partial_2=\emptyset \} \sqcup N_{\text{arc}}
         \end{align*}
         for sets $N_{\text{loop}}$ and $N_{\text{arc}}$ of \emph{new loops} and \emph{new arcs} to be described as follows.
             
             Consider a vertex $v\in V^\partial_2$. By Definition \ref{def:combboundaryfoam}, there are exactly one seam in $E_{\text{arc}}(\Sigma_1)$ and one in $E_{\text{arc}}(\Sigma_2)$ that are adjacent to $v$. Therefore, $v$ is part of a unique maximal path $[v]:=(v_1 \xleftrightarrow{e_1} v_2 \xleftrightarrow{e_2} \cdots \xleftrightarrow{e_k} v_{k+1})$ of pairwise distinct seams $e_i\in E_{\text{arc}}(\Sigma_1)\sqcup E_{\text{arc}}(\Sigma_2)$ through boundary vertices $v_2,\ldots,v_{k}\in V^\partial_2$. Appearing in the same such path defines an equivalence relation on $V^\partial_2$. We define the set of new seams to be the set of equivalence classes:
             \[ N_{\text{loop}}\sqcup N_{\text{arc}}:=\{[v]\mid v\in V^\partial_2\}
             \]
            To determine the partition into loops and arcs we consider two cases for the path $[v]=(v_1 \xleftrightarrow{e_1} v_2 \xleftrightarrow{e_2} \cdots \xleftrightarrow{e_k} v_{k+1})$:
            If $v_1,v_{k+1}\notin V^\partial_2$, then we consider the path a new arc, i.e.~ set $[v]\in N_{\text{arc}}$, and declare $V([v]):=\{v_1,v_{k+1}\}$. The seams $e_1\dots e_k$ carry consistent orientations and we equip $V([v])$ with the induced orientation. If, on the other hand $v_1\in V^\partial_2$ or $v_{k+1}\in V^\partial_2$, then by maximality we have $v_1=v_{k+1}$ and we consider the new seam a loop $[v]\in N_{\text{loop}}$. 
            
        \item  The set $F=\bigsqcup F_{g,n}$ of facets is given by the quotient $(F(\Sigma_1)\sqcup F(\Sigma_2))/\sim$ with respect to the equivalence relation on facets generated by
        \[f\sim f' \iff  E_2^\partial(f)\cap E_2^\partial(f')\neq \emptyset \]
        i.e.\ sharing a common boundary edge in $E_2^\partial$. We denote the equivalence class of a facet $f\in F(\Sigma_1)\sqcup F(\Sigma_2)$ by $[f]$ and its induced partition into facets from $F(\Sigma_1)$ and $F(\Sigma_2)$ by $[f]= [f]_1 \sqcup [f]_2$. 
        
        Motivated by Remark~\ref{rem:trivalent} we like to think of $[f]$ as a new facet obtained from the sets of facets $[f]_1$ and $[f]_2$ by gluing along adjacent loops (in the nilvalent and trivalent cases) and adjacent arcs (in the trivalent case) in $E^\partial_2$. More specifically, Remark~\ref{rem:trivalent} guarantees that facets never have to be glued along paths of boundary arcs of length greater than one.

        \item The set of boundary words $\partial([f])\subset \powerset(W(\overline{E}))$ of $[f]\in F$ is obtained from the set 
        \[
        \bigsqcup_{f'\in [f]_1} \partial(f') \sqcup \bigsqcup_{f''\in [f]_2} \partial(f'') \subset \powerset(W(\overline{E}(\Sigma_1) \cup \overline{E}(\Sigma_2))) 
        \]
        by the following algorithm, which iteratively applies the following steps:
        \begin{enumerate}[(a)]
            \item We first remove all pairs $\partial_i(f'), \partial_j(f'')$ of occurrences of 1-element boundary words given by common loops in $E_2^\partial$
            \[
            \hspace{2.3cm}
            \partial_i(f') \in W(E_{2,\text{loop}}^\partial \cap \bigcup_{f'\in [f]_1} E^\partial_2(f')) \quad \text{ and } \quad \partial_j(f'')\in W(E_{2,\text{loop}}^\partial \cap  \bigcup_{f''\in [f]_2} E^\partial_2(f'')) \]
            Let $g_+([f])$ denote the number of times this operation is performed.
            \item[(b1)] We consider all edges $e\in E^\partial_{2,\text{arc}}$ that appear twice among the remaining words. The two occurrences originate from boundary words of the two unique facets adjacent to $e$ in $\Sigma_1$ and $\Sigma_2$ respectively. The edge $e$ is thus traversed in opposite orientations. First we consider the case when $e$ appears in two \emph{distinct} cyclic words of lengths $k,l\geq 2$:
            \[
            \hspace{2.3cm}(v_1 \xrightarrow{e_1^{\epsilon_1}} \cdots  \xrightarrow{e_{k-1}^{\epsilon_{k-1}}} v_k \xrightarrow{e^{\epsilon}} v_{1}) \text{ and }
            (v_1 \xrightarrow{e^{-\epsilon}} v_{k} \xrightarrow{e_{k}^{\epsilon_{k}}} \cdots  \xrightarrow{e_{k+l-2}^{\epsilon_{k+l-2}}} v_{1}) \]
            Here we remove the two cyclic words and append the grafted cyclic word of length $k+l-2$ given by
            \[
            (v_1 \xrightarrow{e_1^{\epsilon_1}} \cdots  \xrightarrow{e_{k-1}^{\epsilon_{k-1}}} v_k \xrightarrow{e_{k}^{\epsilon_{k}}} \cdots  \xrightarrow{e_{k+l-2}^{\epsilon_{k+l-2}}} v_{1})\]
            Let $n_-([f])$ denote the number of times this operation is performed.
             \item[(b2)] For the second case, we assume that the edge $e\in E^\partial_{2,\text{arc}}$ appears twice in one cyclic word of length $k+l\geq 4$:
            \[\hspace{2.3cm}
            (v_1 \xrightarrow{e_1^{\epsilon_1}} \cdots  \xrightarrow{e_{k-1}^{\epsilon_{k-1}}} v_k \xrightarrow{e^{\epsilon}} v_{k+1} \xrightarrow{e_{k+1}^{\epsilon_{k+1}}} \cdots  
            \xrightarrow{e_{k+l-3}^{\epsilon_{k+l-3}}} v_{k+l-2} \xrightarrow{e_{k+l-2}^{\epsilon_{k+l-2}}} v_{k+1}\xrightarrow{e^{-\epsilon}} v_{k}\xrightarrow{e_{k+l}^{\epsilon_{k+l}}} v_{1}) 
            \]
             We now remove the cyclic word and append the two cyclic words given by
            \[
            \hspace{2.3cm}(v_1 \xrightarrow{e_1^{\epsilon_1}} \cdots  \xrightarrow{e_{k-1}^{\epsilon_{k-1}}} v_k \xrightarrow{e_{k+l}^{\epsilon_{k+l}}} v_{1}) \text{ and }
            (v_{k+1} \xrightarrow{e_{k+1}^{\epsilon_{k+1}}} \cdots  \xrightarrow{e_{k+l-3}^{\epsilon_{k+l-3}}} v_{k+l-2} \xrightarrow{e_{k+l-2}^{\epsilon_{k+l-2}}} v_{k+1}) 
            \]
            Let $n_+([f])$ denote the number of times this operation is performed.
            \item[(c)] For every cyclic word traversing at least one boundary vertex $v\in V_2^\partial$, the unique maximal path $[v]$ described above appears as a subword; contract this subword to the single letter $[v]\in N_{\text{loop}}\sqcup N_{\text{arc}}\subset \overline{E}$. 
        \end{enumerate}
            It is clear that this algorithm terminates after finitely many steps, since every step reduces the total number of letters in all boundary words of facets. We leave it to the reader to check that the result does not depend on the order in which the steps of type (b1) and (b2) are applied.
           For a finite set of facets $F'$ we let $g(F'), n(F')\in \N_0$ denote the sum total genus resp. number of boundary components of all facets in $F'$. For $i\in\{1,2\}$, we set $g_i([f]):=g([f]_i),n_i([f]):=n([f]_i)\in \N_0$. The number of boundary components and genus of the new facet $[f]$ are defined as
        \begin{align*}
        g([f])&:=g_1([f])+g_2([f]) + n_-([f]) + g_+([f]) + 1 - |[f]_1| -|[f]_2| \\
        n([f])&:= |\partial([f])| =  n_1([f]) + n_2([f]) - 2\cdot g_+([f]) - n_-([f]) + n_+([f]).
        \end{align*} 
       \item  One can check that every edge every $e\in \overline{E}$ and every $v\in \overline{V}$ is traversed at most once in $\partial_i([f])$ and for at most one $i$. This uses the assumed trivalency or nilvalency. In particular, each seam $e\in E$ is adjacent to exactly three facets and these inherit the cyclic ordering from their constituent facets in $\Sigma_1$ and $\Sigma_2$. We leave it as an exercise to show that the cyclic orderings are well-defined for seams in $N_{\text{loop}}\sqcup N_{\text{arc}}$.
        \end{itemize}

\end{construction}

        \begin{remark} \label{rem:Euler} The elements $f\in F_{g,n}$ can informally be thought of as connected oriented surfaces of genus $g$ and $n$ boundary components and, thus, Euler characteristic $\chi(f) = 2 - n - 2g$. As a plausibility check, we now compare the update rule for genus for a glued facet in \Cref{constr:foamcomposition} with the gluing behavior of the Euler characteristic. Recall that $\chi$ is additive under disjoint union and constant under self-gluing along circles, while it decreases by $1$ under gluing along an interval. The Euler characteristic of the glued facet is, on the one hand $2- n([f]) - 2\cdot  g([f])$, or expanded:
           \[
             2 -n_1([f]) - n_2([f]) + 2\cdot g_+([f]) + n_-([f]) - n_+([f]) - 2\cdot  g([f]) 
           \]
           On the other hand, we compute it from the Euler characteristics of all facets in $[f]_1$ and $[f]_2$ and the number of interval gluings:
           \[
           2|[f]_1| + 2|[f]_2| -n_1([f]) - n_2([f]) - 2\cdot g_1([f]) - 2\cdot g_2([f]) - n_+([f]) -n_-([f])
           \]
           Equating the two expressions and solving for  $g([f])$, we obtain:
           \[
            g([f]) =  g_1([f])+g_2([f])+ n_-([f]) + g_+([f]) + 1 - |[f]_1| -|[f]_2|
           \]
    Similarly, one can informally think of the elements $e\in \overline{E}$ as connected oriented 1-manifolds, closed if $e\in \overline{E}_{\text{loop}}$ and with boundary if $e\in \overline{E}_{\text{arc}}$. Consequently, we set $\chi(e)=0$ in the first case and $\chi(e)=1$ in the second. Lastly, we set $\chi(v)=1$ for a vertex $v\in \overline{V}$, and informally view it as a point.
        \end{remark}
        
\begin{definition}[Identity foam] \label{def:idcombfoam}
    Given a (trivalent or nilvalent) combinatorial closed web $W$, we define the \emph{identity foam on $W$} to be the (trivalent or nilvalent) combinatorial foam $\id_W\colon W \to W$ constructed as follows:
    \begin{itemize}
        \item We rename the two webs constituting the boundary of the desired foam $\id_W$ by $W_1:=W^r$ and $W_2:=W$. For $i\in\{1,2\}$, denote by $V^\partial_i$ and $E^\partial_i$ the sets of vertices and edges of $W_i$ and note that we have canonical bijections $V^\partial_1\cong V^\partial_2$ and $E^\partial_1\cong E^\partial_2$. The set $\overline{V}$ of vertices of $\id_W$ consists, by definition, entirely of the boundary vertices $\overline{V}=V^\partial=V^\partial_1\sqcup V^\partial_2$ and no internal vertices $V=\emptyset$.
        \item The set of boundary edges is declared to be $E^\partial=E^\partial_1\sqcup E^\partial_2$.
        \item The set $E=E_{\text{loop}}\sqcup E_{\text{arc}}$ of seams is given by $E_{\text{loop}}=\emptyset$ and $E_\text{arc}=\{e_{v}\}_{v\in V^\partial_2}$, where, for each merge/split vertex $v\in V^\partial_2$, we let $v^r\in V^\partial_1$ be the corresponding split/merge vertex of $W_1$ under the canonical bijection and define $e_{v}$ as the seam with $V(e_{v})=\{v,v^r\}$, oriented from the merge vertex to the split vertex. 
        \item The set $F$ of facets is given by $F=F_{0,1}\sqcup F_{0,2}=\{f_{e}\}_{e\in E^\partial_2}$ (in particular, $F_{g,n}=\emptyset$ for $g>0$ or $n>2$), where $f_{e}$ is defined as follows: for each $e\in E^\partial_2$, let $e^r\in E^\partial_1$ be the corresponding edge of $W_1$ under the canonical bijection. If $e$ is a loop, let $f_{e}\in F_{0,2}$ be the facet with boundary $\{\partial_1(f_{e}),\partial_2(f_{e})\}=\{(e),(e^r)\}$. If $e$ is an arc with source $v$ and target $v'$, let $f_{e}\in F_{0,1}$ be the facet with boundary $\partial(f_{e})=(e,e_{v'},e^r,e_v)$.
        \item Given a seam $e_v\in E$ with $v\in V^\partial_2$, the three facets $f_{e},\, f_{e'}$ and $f_{e''}$ adjacent to it inherit a cyclic ordering from the cyclic ordering of the edges $e,\,e',\,e''\in E^\partial_2$ adjacent to $v$.
    \end{itemize}
\end{definition}

\begin{prop}
\label{prop:foamcat} The following data defines a small symmetric monoidal category, which we will denote by $\PreFoamCat$ and call the \emph{category of closed webs and prefoams}: 
\begin{itemize}
   \item The set of objects is $\WebSet$, the set of combinatorial closed webs as in Definition~\ref{def:combclosedweb}.
   \item The set of morphisms $W_1\to W_2$ is the set of combinatorial foams as in Definition~\ref{def:foammorphism}.
   \item The composition law for morphisms is given by Construction \ref{constr:foamcomposition}.
   \item For each $W\in \WebSet$, the identity morphism $\id_W$ on $W$ is given in Definition~\ref{def:idcombfoam}.
   \item The symmetric monoidal structure is given on webs and foams by disjoint union, with the empty web serving as tensor unit.
\end{itemize}

\end{prop}

\begin{proof}
    By inspection of Construction \ref{constr:foamcomposition}, composition is associative and has identity foams acting neutrally on both sides.
\end{proof}

\subsection{Labeled foams with boundary}
\label{sec:labeledfoams}

The definition of combinatorial $\labdata$-foams with boundary is a direct extension of Definition \ref{def:labclosedfoam}.

\begin{definition}[Combinatorial $\labdata$-foam with boundary] \label{def:labboundaryfoam} Let $\kb$ be a commutative ring and $\labdata$ as in Definition~\ref{def:labdata}. A (trivalent or nilvalent) \emph{combinatorial $\labdata$-foam with boundary} consists of the following data:
\begin{enumerate}
    \item A (trivalent or nilvalent) combinatorial foam with boundary as in Definitions \ref{def:combboundaryfoam} and \ref{def:combboundaryfoam2} with sets of vertices, seams and facets denoted by $\overline{V},\,\overline{E}$ and $F$ respectively.
    \item A labeling $\labelfunction\colon F \to \labdata_1$ of all facets by elements of $\labdata_1$.
    \item A decoration of the facets by an element $\Dec(F)\in \bigotimes_{f\in F} \Dec(\labelfunction(f))$. 
    \item The labeling must satisfy the same conditions on the sets of facets $F$, internal seams $E$ and internal vertices $V$ as formulated for combinatorial closed $\labdata$-foams in Definition~\ref{def:labclosedfoam}.
\end{enumerate}
The set of combinatorial foams with boundary is denoted $\labboundaryFoamSet{\labdata}$.
\end{definition}

We now describe the boundaries of $\labdata$-foams from Definition~\ref{def:labboundaryfoam}.

\begin{definition}[Combinatorial closed $\labdata$-web]
\label{def:labclosedweb}
    Let $\labdata$ be as in Definition~\ref{def:labdata}. A (trivalent or nilvalent) \emph{combinatorial closed $\labdata$-web} consists of:
    \begin{itemize}
        \item A (trivalent or nilvalent) combinatorial closed web as in Definitions~\ref{def:combclosedweb} and \ref{def:combclosedweb2} with sets of vertices and oriented edges denoted by $V^\partial$ and $E^\partial$ respectively.
        \item A labeling $\labelfunction\colon E^\partial \to \labdata_1$ of all facets by elements of $\labdata_1$.
    \end{itemize}
    In the nilvalent case, this finishes the definition. In the trivalent case, the data is subject to the following requirement: given a vertex $v\in V^\partial$, we refine the cyclic order on $E(v)$ to a total order $(e_1,e_2,e_3)$ by requiring that $e_3$ is the edge with orientation incompatible with the others. Then, the labels of $e_1,e_2,e_3$ are required to satisfy:
    \[\labelfunction(e_1) + \labelfunction(e_2) = \labelfunction(e_3).\]
    The set of combinatorial foams with boundary is denoted $\labWebSet{\labdata}$.
\end{definition}

Combinatorial closed $\labdata$-webs arise as boundaries of the combinatorial $\labdata$-foams with boundary from Definition~\ref{def:labboundaryfoam}. Note that every boundary edge $e\in E^\partial$ inherits an induced label $\labelfunction(e)$ from the unique adjacent facet $f\in F(e)$. The map that associates to each $\labdata$-foam with boundary its boundary $\labdata$-web will be denoted:
\[
\partial\colon  \labboundaryFoamSet{\labdata} \to \labWebSet{\labdata}
\]

\begin{definition}
\label{def:labfoammorphism}
    Given $\labdata$-webs $W_1,W_2\in \labWebSet{\labdata}$, a $\labdata$-foam $\Sigma\colon W_1\to W_2$ is defined to be a foam $\Sigma\in \labboundaryFoamSet{\labdata}$ with boundary $\partial(\Sigma)= W_2 \sqcup W_1^r$. 
    The set of all $\labdata$-foam $\Sigma\colon W_1\to W_2$ with the same underlying data (1) and (2) in Definition~\ref{def:labboundaryfoam} is parametrized by the $\kb$-module of decorations on the facets (3). We let 
\[\HomPreFoamCat[\labdata](W_1,W_2)\] denote the $\kb$-module obtained as the coproduct of these $\kb$-modules as (1) and (2) range over all $\labdata$-foam $F\colon W_1\to W_2$.
\end{definition}

Note that we have canonical identifications:
\begin{equation}
\label{eq:absdual}
\begin{split}
\HomPreFoamCat[\labdata](W_1,W_2)
=
\HomPreFoamCat[\labdata](\emptyset,W_2\sqcup W_1^r)
\\=
\HomPreFoamCat[\labdata](W_1\sqcup W_2^r,\emptyset) 
=
\HomPreFoamCat[\labdata](W_2^r,W_1^r).
\end{split}
\end{equation}

\begin{remark}
    \label{rem:involutivetwo}
    In the case of involutive labeling data as in \Cref{def:involutive}, we obtain an involution on $\labboundaryFoamSet{\labdata}$ sending $\labdata$-labeled foams with boundary $W$ to such with boundary $W^r$. This involution again acts by reversing the underlying foam and acting by $\iota$ on all decorations. In particular, we obtain isomorphisms:
    \begin{equation}
\label{eq:prefoaminvo}
\begin{split}
\HomPreFoamCat[\labdata](W_1,W_2)
=
\HomPreFoamCat[\labdata](\emptyset,W_2\sqcup W_1^r)
\\\xrightarrow{-^r}
\HomPreFoamCat[\labdata](\emptyset,W_2^r\sqcup W_1) 
=
\HomPreFoamCat[\labdata](W_2,W_1).
\end{split}
\end{equation}

\end{remark}

\begin{construction}[Composition of $\labdata$-foams]
\label{constr:labfoamcomposition}
Let $W_1,\,W_2,\,W_3\in \labWebSet{\labdata}$ be $\labdata$-webs and consider two combinatorial (trivalent or nilvalent) $\labdata$-foams $\Sigma_1\colon W_1 \to W_2$ and $\Sigma_2\colon W_2 \to W_3$. We construct a $\labdata$-foam 
\[\Sigma_2\circ \Sigma_1\colon W_1 \to W_3\] as follows:
\begin{itemize}
    \item The underlying combinatorial (unlabeled) foam is given by the composition of the combinatorial foams underlying $\Sigma_1$ and $\Sigma_2$, as in Construction \ref{constr:foamcomposition}. Recall that the set of facets is given by $F=(F(\Sigma_1)\sqcup F(\Sigma_2))/\sim$.
    \item Each facet $[f]\in F$ inherits a label $\labelfunction([f])$ from the label of $f\in F(\Sigma_1)\sqcup F(\Sigma_2)$: indeed, if $f \sim f'$, then $\labelfunction(f)=\labelfunction(f')$, because they share a common boundary edge in the $\labdata$-web $W_2$.
    \item Let $\Dec(F(\Sigma_1))$ and $\Dec(F(\Sigma_2))$ be the decorations on $\Sigma_1$ and $\Sigma_2$ respectively. A decoration $\Dec(F)$ on $F$ is given by the image of $\Dec(F(\Sigma_1))\otimes\Dec(F(\Sigma_2))$ under the map 
    \[
       \hspace{1cm} \bigotimes_{f\in F(\Sigma_1)} \Dec(\labelfunction(f)) \otimes \bigotimes_{f\in F(\Sigma_2)} \Dec(\labelfunction(f)) \cong 
        \bigotimes_{[f]\in F} \bigotimes_{f'\in [f]} \Dec(\labelfunction([f])) \longrightarrow
        \bigotimes_{[f]\in F} \Dec(\labelfunction([f]))
    \]
    induced by the multiplication map $\bigotimes_{f'\in [f]} \Dec(\labelfunction([f])) \to \Dec(\labelfunction([f]))$.
\end{itemize}
It follows directly that $(\Sigma_2,\Sigma_1)\to \Sigma_2\circ \Sigma_1$ defines a $\kb$-bilinear map
\[
\circ\colon \HomPreFoamCat[\labdata](W_2,W_3) \times \HomPreFoamCat[\labdata](W_1,W_2) \to \HomPreFoamCat[\labdata](W_1,W_3)
\]

\end{construction}

\begin{definition} \label{def:idlabfoam}
    Given a (trivalent or nilvalent) combinatorial closed $\labdata$-web $W$, we define the \emph{identity foam on $W$} to be the (trivalent or nilvalent) $\labdata$-foam $\id_W\colon W \to W$ constructed as follows:
    \begin{itemize}
        \item The underlying combinatorial foam is given by the identity foam on $W$ from Definition~\ref{def:idcombfoam}. Recall that the set of facets of $\id_W$ is given by $F=\{f_e\}_{e\in E^{\partial}(W)}$.
        \item Each facet $f_e\in F$ inherits a label from $e\in E^\partial(W)$: $\labelfunction(f_e)=\labelfunction(e)$.
        \item The decoration on the facets is given by $\Dec(F)=1$.
    \end{itemize}
\end{definition}  

\begin{prop}
\label{prop:labfoamcat}
The following data defines a symmetric monoidal $\kb$-linear category, that we denote by $\PreFoamCat[\labdata]$ and call the \emph{category of closed $\labdata$-webs and $\labdata$-prefoams}: 
\begin{itemize}
   \item The objects set $\labWebSet{\labdata}$ is the set of combinatorial closed $\labdata$-webs from Definition~\ref{def:labclosedweb}.
   \item The $\kb$-module of morphisms $W_1\to W_2$ is formed by the $\labdata$-foams as in Definition~\ref{def:labfoammorphism}.
   \item The composition law for morphisms is given by Construction \ref{constr:labfoamcomposition}.
   \item For each $W\in \labWebSet{\labdata}$, the identity morphism $\id_W$ on $W$ is given in Definition~\ref{def:idlabfoam}.
   \item The symmetric monoidal structure is given on $\labdata$-webs by disjoint union and on foams by the $\kb$-bilinear extension of the disjoint union of underlying combinatorial foams to $\labdata$-foams. The empty web serves as tensor unit.
\end{itemize}
\end{prop}
\begin{proof} This follows from Proposition~\ref{prop:foamcat} and the unitality and associativity of algebras of decorations on the facets.
\end{proof}

\subsection{Universal construction} 
\label{sec:univ}
Fix a commutative ring $\kb$ and a $\kb$-linear labeling datum $\labdata$ as in Definition~\ref{def:labdata} and a closed foam evaluation $\eval$ as in Definition \ref{def:eval}, not yet assumed to be multiplicative. We will use the universal construction of \cite{BHMV} to impose relations on $\labdata$-foams.

For an object $W\in \PreFoamCat[\labdata]$ we use \eqref{eq:absdual} to define the \emph{trace}:
\[
\mathrm{tr}_W \colon \HomPreFoamCat[\labdata](W,W) =  
\HomPreFoamCat[\labdata](W\sqcup W^r,\emptyset) \xrightarrow{\id_W \circ - } \HomPreFoamCat[\labdata](\emptyset,\emptyset)
\]
where we have used \eqref{eq:absdual} again to consider $\id_W\in \HomPreFoamCat[\labdata](\emptyset,W\sqcup W^r)$. For another object $W'\in \PreFoamCat[\labdata]$ and $\Sigma\in \HomPreFoamCat[\labdata](W,W')$ and $\Sigma'\in \HomPreFoamCat[\labdata](W',W)$ we have
\[\tr_W(\Sigma'\circ \Sigma) = \tr_{W'}(\Sigma\circ \Sigma') \]

\begin{lemma}
\label{lem:foamideal}
    For objects $W,W'\in \PreFoamCat[\labdata]$, we consider the symmetric and $\kb$-linear pairing:
    \begin{gather*}
         \HomPreFoamCat[\labdata](W',W) \otimes \HomPreFoamCat[\labdata](W,W')  \longrightarrow \kb\\
         \Sigma' \otimes \Sigma \mapsto \langle\Sigma',\Sigma\rangle:=\eval(\mathrm{tr}_W(\Sigma'\circ \Sigma))
    \end{gather*}
    and the $\kb$-linear map:
    \begin{gather*}
         \HomPreFoamCat[\labdata](W',\emptyset)\otimes \HomPreFoamCat[\labdata](W,W')\otimes \HomPreFoamCat[\labdata](\emptyset,W) \longrightarrow \kb\\
         \Sigma'' \otimes \Sigma \otimes \Sigma' \mapsto \langle\Sigma'',\Sigma,\Sigma'\rangle:=\eval(\Sigma'' \circ \Sigma\circ \Sigma')
    \end{gather*}
    Then the submodules
    \begin{equation}
    \label{eq:foamideal}
    I_1(W,W'):= \bigcap_{\Sigma'} \ker\, \langle\Sigma',-\rangle \subset \HomPreFoamCat[\labdata](W,W') 
    \end{equation}
    and
    \begin{equation}
    \label{eq:foamidealtwo}
    I_2(W,W'):= \bigcap_{\Sigma',\Sigma''} \ker\, \langle\Sigma'',-,\Sigma'\rangle \subset \HomPreFoamCat[\labdata](W,W') 
    \end{equation}
    satisfy $I_1(W,W')\subset I_2(W,W')$ and both families define ideals of morphisms in $\PreFoamCat[\labdata]$.
\end{lemma}
\begin{proof}
For the first assertion, note that $\langle\Sigma'',-,\Sigma'\rangle = \langle\Sigma'\circ \Sigma'',-\rangle$ and the indexing set of the intersection in \eqref{eq:foamidealtwo} corresponds to the subset of the indexing set of the intersection in \eqref{eq:foamideal} formed by those $\Sigma'$ that factor through $\emptyset$.

\noindent To check that $I_1$ forms an ideal, we consider morphisms $W'\xrightarrow{ \Phi' } V \xrightarrow{ \Phi } W \xrightarrow{ \Sigma } W' \xrightarrow{ \Psi } V' \xrightarrow{ \Psi'' }W$
in $\PreFoamCat[\labdata]$ and find
\begin{align*}
\Sigma \in \ker\, \langle\Psi''\circ \Psi,-\rangle 
&\implies 
\Psi\circ \Sigma \in \ker\, \langle\Psi'',-\rangle \\
\Sigma \in \ker\, \langle\Phi\circ \Phi' ,-\rangle 
&\implies 
\Sigma\circ \Phi \in \ker\, \langle\Phi',-\rangle
\end{align*}
where we have used cyclicity of the trace twice in the second line.
Since $\Psi''\circ \Psi$ and $\Phi\circ \Phi'$ appear in the indexing of the intersection in \eqref{eq:foamideal}, we have that $\Sigma \in I_1(W,W')$ implies $\Psi\circ \Sigma \in I_1(W,V')$ and $\Sigma\circ \Phi \in I_1(V,W')$.

For $I_2$ we instead consider morphisms $\emptyset \xrightarrow{ \Phi' } V \xrightarrow{ \Phi } W \xrightarrow{ \Sigma } W' \xrightarrow{ \Psi } V' \xrightarrow{ \Psi'' } \emptyset$
in $\PreFoamCat[\labdata]$ and find:
\begin{align*}
\Sigma \in \ker\, \langle\Psi''\circ \Psi,-,\Sigma'\rangle 
&\implies 
\Psi\circ \Sigma \in \ker\, \langle\Psi'',-,\Sigma'\rangle \\
\Sigma \in \ker\, \langle\Sigma'',-,\Phi\circ \Phi'\rangle 
&\implies 
\Sigma\circ \Phi \in \ker\, \langle\Sigma'',-,\Phi'\rangle
\end{align*}
Since $\Psi''\circ \Psi$ and $\Phi\circ \Phi'$ appear in the indexing of the intersection in \eqref{eq:foamidealtwo}, we have that $\Sigma \in I_2(W,W')$ implies $\Psi\circ \Sigma \in I_2(W,V')$ and $\Sigma\circ \Phi \in I_2(V,W')$.
\end{proof}

\begin{lemma}
\label{lem:homfromempty}
    For any $W\in \PreFoamCat[\labdata]$, we have $I_1(\emptyset,W)=I_2(\emptyset,W)$.
\end{lemma}
\begin{proof}
For any morphism $\Sigma\in I_2(\emptyset,W)$, we have $\eval(\Sigma'\circ \Sigma \circ C)=0$ for all $\Sigma'\in \PreFoamCat[\labdata](W, \emptyset)$ and $C\in \PreFoamCat[\labdata](\emptyset, \emptyset)$. This includes the case $C=\emptyset$, so that we have
\[
\langle \Sigma', \Sigma\rangle = \langle \Sigma', \Sigma, \emptyset\rangle =0 
\]
which implies $\Sigma\in I_1(\emptyset,W)$. The reverse inclusion was shown in \Cref{lem:foamideal}.
\end{proof}

\begin{definition} We introduce the notation  
$\PreFoamCat[\labdata] \xrightarrow{\pi} \FoamCata[\labdata] \xrightarrow{\pi} \FoamCatb[\labdata]$
 for the $\kb$-linear quotient functors associated with the ideals from \Cref{lem:foamideal} and their codomains. Further notation concerning representable functors $\webeval^i=\Hom(\emptyset,-)$ is defined in terms of the following diagram of $\kb$-linear functors and natural transformations:

\[\begin{tikzcd}
	{\PreFoamCat[\labdata]} \\
	{\FoamCata[\labdata]} &&&& {\kb\Mod} \\
	{\FoamCatb[\labdata]}
	\arrow["\pi"', two heads, from=1-1, to=2-1]
	\arrow[""{name=0, anchor=center, inner sep=0}, curve={height=6pt}, from=1-1, to=2-5]
	\arrow[""{name=1, anchor=center, inner sep=0}, "{\Hom(\emptyset,-)}", curve={height=-18pt}, from=1-1, to=2-5]
	\arrow[""{name=2, anchor=center, inner sep=0}, "{\webeval^1}"'{pos=0.7}, from=2-1, to=2-5]
	\arrow[""{name=3, anchor=center, inner sep=0}, curve={height=24pt}, from=2-1, to=2-5]
	\arrow["\pi"', two heads, from=2-1, to=3-1]
	\arrow["{\webeval^2}"'{pos=0.7}, curve={height=18pt}, from=3-1, to=2-5]
	\arrow[shorten <=4pt, shorten >=4pt, Rightarrow, from=1, to=0]
	\arrow[shorten <=4pt, shorten >=4pt, Rightarrow, from=2, to=3]
\end{tikzcd}\]
Here the unlabeled arrows are defined by commutativity of the unlabeled triangles. Additionally, we define the $\kb$-linear functors:
\[\unicon^1, \unicon^2 \colon \PreFoamCat[\labdata]\to \kb\Mod,\quad \unicon^1:=\webeval^1\circ \pi,\quad \unicon^2:=\webeval^2\circ \pi\circ \pi\]
\end{definition}

\begin{corollary}
We have canonical natural isomorphisms of $\kb$-linear functors \[
\webeval^1\xRightarrow{\cong} \webeval^2 \circ \pi \qquad \text{and} \qquad \unicon^1\xRightarrow{\cong} \unicon^2.\]
\end{corollary}
\begin{proof}
    This is a direct consequence of \Cref{lem:homfromempty}.
\end{proof}
In the following, we may write $\unicon:=\unicon^1=\unicon^2$ and $I(\emptyset,W):=I_1(\emptyset,W)=I_2(\emptyset,W)$.

\begin{lemma}
\label{lem:faithful} The $\kb$-linear functor $\webeval^2$ is faithful.
\end{lemma}
\begin{proof}  Let $\underline{\Sigma}\in \FoamCatb[\labdata](W,W')$ be such that $\webeval^2(\underline{\Sigma})=0$. This means the induced map:
\[
\FoamCatb[\labdata](\emptyset,W) \xrightarrow{\underline{\Sigma}\circ - } \FoamCatb[\labdata](\emptyset,W')
\]
is zero. Let $\Sigma\in \HomPreFoamCat[\labdata](W,W')$ be such that $\pi\circ\pi(\Sigma)=\underline{\Sigma}$. This means for every $\Sigma'\in \HomPreFoamCat[\labdata](\emptyset,W)$ we have:
\[
 \Sigma \circ \Sigma' \in I_2(\emptyset, W')=I_1(\emptyset, W')
\]
Then, we must have for every $\Sigma''\in  \HomPreFoamCat[\labdata](W',\emptyset)$
\[
\eval(\Sigma'' \circ \Sigma \circ \Sigma') = 0 
\]
and hence $\Sigma \in I_2(W,W')$, which implies $\underline{\Sigma}=\pi\circ\pi(\Sigma)=0$ in $\FoamCatb[\labdata](W,W')$.
\end{proof}

\begin{lemma}
\label{lem:endempty} The closed foam evaluation $\eval$ factors as:
    \[
    \begin{tikzcd}
	{\PreFoamCat[\labdata]}(\emptyset, \emptyset) \\
	{\FoamCata[\labdata]}(\emptyset, \emptyset)  &&& {\kb} \\
	{\FoamCatb[\labdata](\emptyset, \emptyset)}
	\arrow[""{name=0, anchor=center, inner sep=0}, "{{\eval}}", curve={height=-18pt}, from=1-1, to=2-4]
	\arrow["\pi"', two heads, from=1-1, to=2-1]
        \arrow[double, no head, from=2-1, to=3-1]
        \arrow["{{\eval'}}"', hook, from=2-1, to=2-4]
	\arrow["{{\eval'}}"',hook, curve={height=18pt}, from=3-1, to=2-4]
\end{tikzcd}
\] 
If $\eval$ is multiplicative in the sense of \Cref{def:eval}, then the following hold:
\begin{enumerate}
    \item The map $\eval'$ is a $\kb$-algebra isomorphism.
    \item The functor $\unicon$ is laxly symmetric monoidal with respect to disjoint union.
\end{enumerate}
\end{lemma} 
\begin{proof}
For any endomorphism $C\in I(\emptyset,\emptyset)$ and $C'\in \PreFoamCat[\labdata](\emptyset, \emptyset)$, the definition of $I_1$ implies $\eval(C'\circ C)=0$. This includes the case $C'=\emptyset$, which implies $I(\emptyset,\emptyset)  \subset \ker(\eval)$. Thus, the induced evaluation $\eval'$ is well-defined. It is injective by \Cref{lem:faithful}. 

(1) If $\eval$ satisfies $\eval(\emptyset)=1\in \kb$ as required in \Cref{def:eval}.(1), then $1\mapsto \emptyset$ defines a section $\kb \to \PreFoamCat[\labdata](\emptyset, \emptyset)$ of $\eval$, showing that $\eval$ and $\eval'$ are surjective. 
\Cref{def:eval}.(2) then expresses the required compatibility with the algebra multiplication. 

(2) For objects $W,W'\in \PreFoamCat[\labdata]$ consider the tensor product
\begin{equation}
\label{eq:tensfromempty}
\PreFoamCat[\labdata](\emptyset, W) \otimes \PreFoamCat[\labdata](\emptyset,V) \to \PreFoamCat[\labdata](\emptyset,W\sqcup V)
\end{equation}
induced by disjoint union of foams and denote it $\sqcup$. Let $\Sigma\in I(\emptyset,W)\subset \PreFoamCat[\labdata](\emptyset,W)$ and $\tau\in \PreFoamCat[\labdata](\emptyset,V) = \PreFoamCat[\labdata](V^r,\emptyset)$.

Let $X\in \PreFoamCat[\labdata](W\sqcup V,\emptyset)= \PreFoamCat[\labdata](W,V^r)$ be arbitrary. Then
\begin{align*}
\langle X, \Sigma \sqcup \tau\rangle 
&= \eval(\mathrm{tr}_\emptyset (X\circ (\Sigma \sqcup \tau))) \\
& = \eval(X\circ (\Sigma \sqcup \tau) 
\\
&= \eval(\tau \circ X \circ \Sigma) 
\\
&=\eval(\mathrm{tr}_\emptyset (\tau \circ X \circ \Sigma))
=\langle \tau \circ X, \Sigma \rangle =0
\end{align*}
which shows $\Sigma \sqcup \tau\in I_1(\emptyset,W\sqcup V)=I(\emptyset,W\sqcup V)$. Thus, the tensor product \eqref{eq:tensfromempty} descends to the quotient, written in three equivalent ways as:
\begin{align} \label{eq:tensorprodFoam}
\FoamCata[\labdata](\emptyset, W) \otimes \FoamCata[\labdata](\emptyset,V) &\to \FoamCata[\labdata](\emptyset,W\sqcup V)\\
\nonumber \FoamCatb[\labdata](\emptyset, W) \otimes \FoamCatb[\labdata](\emptyset,V) &\to \FoamCatb[\labdata](\emptyset,W\sqcup V)\\
\nonumber\unicon(W) \otimes \unicon(V) &\to \unicon(V\sqcup W).
\end{align}
These maps are natural in both arguments and hence equip the functor $\unicon$ with a lax monoidal structure. Furthermore, the functor $\unicon$ is clearly compatible with the symmetric braidings on $\PreFoamCat[\labdata]$ and $\kb\Mod$ given by swapping components.
\end{proof}

\begin{lemma}
\label{lem:strong}
    If $\eval$ is multiplicative in the sense of \Cref{def:eval} and if, furthermore, $\unicon$ is a strong symmetric monoidal functor, then:
\begin{enumerate}
\item For every $W,W'\in \FoamCata[\labdata]$, the map
\[
\circ\colon \FoamCata[\labdata](\emptyset,W')\otimes \FoamCata[\labdata](W,\emptyset) \to \FoamCata[\labdata](W,W')
\]
is surjective, $I_1(W,W')=I_2(W,W')=I(W,W')$ and so the quotient functor witnesses a canonical isomorphism $\pi\colon \FoamCata[\labdata] \to \FoamCatb[\labdata]$ of $\kb$-linear categories. 
\newline
In this case we write $\FoamCat[\labdata]:=\FoamCata[\labdata]=\FoamCatb[\labdata]$ and $\webeval:=\webeval^1=\webeval^2$. 
\item The ideal of morphisms $I(W,W')$ in $\PreFoamCat[\labdata]$ is monoidal, so that the monoidal structure descends to the quotient $\FoamCat[\labdata]$ and the functor $\webeval$ becomes strongly symmetric monoidal.
\end{enumerate}
\end{lemma}

\begin{remark}
\label{rem:Khprops}
    In \cite{khovanov2020universal}, Khovanov discusses a number of properties of relevance concerning the universal construction (see also \cite{BHMV}):
    \begin{itemize}
        \item[(M)] is the strong monoidality of the functor $\unicon$.
        \item[(M')] is the surjectivity expressed in Lemma~\ref{lem:strong}, usually referred to as \emph{neck cutting}.
        \item[(M'')] is a weaker form of (M) that requires $\unicon(W) \otimes \unicon(V) \to \unicon(V\sqcup W)$ only to be an isomorphism onto a $\kb$-module direct summand, for all $V,W\in \PreFoamCat[\labdata]$.
    \end{itemize}
    The proof of Lemma~\ref{lem:strong} starts with observing the implication (M)$\implies$(M'), whose converse also holds if $\kb$ is a field, see \cite{khovanov2020universal}.
\end{remark}

\begin{proof}[Proof of \Cref{lem:strong}]
    (1) The claimed surjectivity (indeed, bijectivity) of $\circ$ follows from strong monoidality of $\unicon$, since $\circ$ corresponds to the map of \eqref{eq:tensorprodFoam} after translating through the commutative square:
        \[
    \begin{tikzcd}
	\FoamCata[\labdata](\emptyset,W')\otimes \FoamCata[\labdata](W,\emptyset) & \FoamCata[\labdata](W,W')\\
	\FoamCata[\labdata](\emptyset,W')\otimes \FoamCata[\labdata](\emptyset,W^r)   & \FoamCata[\labdata](\emptyset,W' \sqcup W^r)
	\arrow["\circ",from=1-1, to=1-2]
    \arrow[from=2-1, to=2-2]
    \arrow[from=1-1, to=2-1]
    \arrow[from=1-2, to=2-2]
\end{tikzcd}
\] 

    The inclusion $I_1(W,W')\subset I_2(W,W')$ follows from Lemma \ref{lem:foamideal}. Let $\Sigma\in I_2(W,W')$ and consider an arbitrary $\tau\in\PreFoamCat[\labdata](W',W)$ and $\pi(\tau)=\underline{\tau}\in \FoamCata[\labdata](W',W)$. By surjectivity of $\circ$, one has
    \[
        \underline{\tau}=\sum_i a_i \underline{\Sigma}'_i\circ\underline{\Sigma}''_i
    \]
    for $\underline{\Sigma}'_i\in\FoamCata[\labdata](\emptyset,W)$, $\underline{\Sigma}''_i\in\FoamCata[\labdata](W',\emptyset)$ and $a_i\in\kb$. For all $i$,
    \begin{align*}
        \langle \Sigma'_i\circ\Sigma''_i,\Sigma\rangle 
        &=
        \eval(\tr_W(\Sigma'_i\circ\Sigma''_i\circ\Sigma)) \\
        &= \eval(\tr_\emptyset(\Sigma''_i\circ\Sigma\circ\Sigma'_i)) \\
        &= \eval(\Sigma''_i\circ\Sigma\circ\Sigma'_i) \\
        &= \langle\Sigma''_i,\Sigma,\Sigma'_i\rangle
        =0.
    \end{align*}
    Therefore, $\langle\tau,\Sigma\rangle=0$ and one concludes that $I_2(W,W')\subset I_1(W,W')$.

    (2) Given objects $V,V',W,W'\in \PreFoamCat[\labdata]$ consider the tensor product
    \begin{equation} \label{eq:tensfoamcat}
    \PreFoamCat[\labdata](V, V') \otimes \PreFoamCat[\labdata](W,W') \to \PreFoamCat[\labdata](V\sqcup W,V'\sqcup W')
    \end{equation}
    induced by disjoint union of foams, denoted by $\sqcup$. Let $\Sigma\in I(V,V')\subset \PreFoamCat[\labdata](V,V')$ and $\Theta\in \PreFoamCat[\labdata](W,W')$.
    Let $\tau\in \PreFoamCat[\labdata](\emptyset,V\sqcup W)$ and $\tau'\in \PreFoamCat[\labdata](V'\sqcup W',\emptyset)$ be arbitrary.
    By point (1), $\underline{\tau}=\pi(\tau)\in \FoamCat[\labdata](\emptyset,V\sqcup W)=\FoamCat[\labdata](V^r,W)$ decomposes as 
    \[
        \underline{\tau}=\sum_i a_i \underline{\Sigma}'_i\circ\underline{\Sigma}''_i
    \]
    for $\underline{\Sigma}'_i\in\FoamCat[\labdata](V^r,\emptyset)=\FoamCat[\labdata](\emptyset,V)$, $\underline{\Sigma}''_i\in\FoamCat[\labdata](\emptyset,W)$ and $a_i\in\kb$. Then, for all $i$,
    \begin{align*}
    \langle \tau', \Sigma \sqcup \Theta,\Sigma'_i\sqcup \Sigma''_i \rangle 
    &= \eval(\tau'\circ (\Sigma \sqcup \Theta)\circ (\Sigma'_i\sqcup \Sigma''_i)) \\
    & = \eval(\tau'\circ (\Sigma \circ \Sigma'_i) \sqcup (\Theta\circ \Sigma''_i))
    \\
    &= \eval(\Sigma''_i \circ \Theta \circ \tau'\circ \Sigma \circ \Sigma'_i)
    \\
    &=\langle\Sigma''_i \circ \Theta \circ \tau', \Sigma, \Sigma'_i \rangle =0
    \end{align*}
    which shows that $\langle \tau', \Sigma \sqcup \Theta,\tau \rangle =0$ and, in turn, that $\Sigma \sqcup \Theta\in I_2(V\sqcup W,V'\sqcup W')=I(V\sqcup W,V'\sqcup W')$. Thus, the tensor product \eqref{eq:tensfoamcat} descends to $\FoamCat[\labdata]$. As $\unicon=\webeval\circ \pi$, the strong monoidality of $\webeval$ follows from the strong monoidality of $\unicon$.
\end{proof}

\begin{lemma}
\label{lem:A_imply_M}
    Suppose that $\eval$ is multiplicative and the following property holds:
    \begin{itemize}
    \item[(A)] Every $W\in \FoamCata[\labdata]$ is a biproduct  $W\cong \bigoplus_{i=1}^m \emptyset$ for some $m\in \N_0$.
    \end{itemize}
    Then:
    \begin{enumerate}
\item Property (M) holds, i.e.\ the functor $\unicon$ is strongly (symmetric) monoidal.
\item All morphism $\kb$-modules in $\FoamCat[\labdata]=\FoamCata[\labdata]$ are free of finite rank.
\item The functor $\webeval$ from Lemma~\ref{lem:strong} induces an equivalence of categories, again denoted by $\webeval$:
    \[
        \webeval\colon \Mat(\FoamCat[\labdata]) \longrightarrow \kb\Mod^{\textnormal{fr}}
    \]from the additive closure $\Mat(\FoamCat[\labdata])$ of $\FoamCat[\labdata]$ to the full subcategory $\kb\Mod^{\textnormal{fr}}$ of $\kb\Mod$ on the free $\kb$-modules of finite rank.
\end{enumerate}
 \end{lemma}
\begin{proof}
    For $V,W\in \FoamCata[\labdata]$ we may assume $V=\bigoplus_{i=1}^m \emptyset$ and $W=\bigoplus_{i=1}^n$ and use that by Lemma~\ref{lem:endempty} we have $\FoamCata[\labdata](\emptyset,\emptyset)) \cong \kb$ and thus:
       \[
        \FoamCata[\labdata](\emptyset,V)\otimes \FoamCata[\labdata](\emptyset,W) \cong \kb^{m}\otimes \kb^{n} \cong \kb^{m,n}\cong \FoamCata[\labdata](\emptyset,V\sqcup W)
        \]
        Similar, we have that
        \[
        \FoamCata[\labdata](V,W) \cong \mathrm{Mat}(m\times n,\kb)
    \]
    is free of finite rank. The stated equivalence of categories follows.
\end{proof}

\begin{example}
\label{ex:glnProperties}
    In the Robert--Wagner $\mathfrak{gl}_N$-foam construction from \Cref{ex:glneval}, the closed foam evaluation map (denoted by $\langle - \rangle$ in \cite{RoW}) is multiplicative and $\iota$-involutive, with involution $\iota=\id$. Moreover, the following properties hold:
    \begin{itemize}
        \item (A) follows from \cite[proof of Theorem 3.30]{RoW}.
        \item (M) follows from (A) and \Cref{lem:A_imply_M}.
        \item (M') follows from (M), by \Cref{lem:strong} and \Cref{rem:Khprops}, and is expressed by relation~(10) of \cite[Proposition 3.32]{RoW}.
        \item (F), introduced below in \Cref{rem:additionalprops}, follows from (A) and \Cref{lem:FM'iffA}.
    \end{itemize}
\end{example}

\subsection{Pairing and sesquilinearity}
\label{sec:pairing}
In this section, we investigate an alternative characterization of property (A) from \Cref{lem:A_imply_M} in case of an involutive foam evaluation in the sense of \Cref{def:involutive}. This subsection is independent of the remainder of the paper.
\smallskip

For each pair of objects $W,V\in\PreFoamCat[\labdata]^1$, the pairing $\langle -,-\rangle$ from \Cref{lem:foamideal} descends to nondegenerate $\kb$-linear pairing on $\FoamCata[\labdata]$, that we denote again by $\langle -,-\rangle$ 
\begin{gather*}
    \FoamCata[\labdata](V,W)\otimes\FoamCata[\labdata](W,V)\longrightarrow\kb \\
    \underline{\Sigma}'\otimes\underline{\Sigma}\mapsto \langle\underline{\Sigma}',\underline{\Sigma}\rangle\coloneqq \langle \Sigma',\Sigma\rangle
\end{gather*}
where $\Sigma$ and $\Sigma'$ are any representatives of $\underline{\Sigma}$ and $\underline{\Sigma}'$.

\begin{lemma}
    The symmetry of \eqref{eq:prefoaminvo} respects the ideals  $I_1(-,=)$ and $I_2(-,=)$ in $\PreFoamCat[\labdata]$ and thus descends to the quotients $\FoamCata[\labdata]$ and $\FoamCatb[\labdata]$.
\end{lemma}
\begin{proof}
    Let $V,W\in \PreFoamCat[\labdata]$ and $\Phi\in I_2(V,W)$. Then for $\Sigma \in \PreFoamCat[\labdata](\emptyset,W)$ and $\Sigma' \in \PreFoamCat[\labdata](V,\emptyset)$ we have  
 \[
 \langle \Sigma', \Phi^r, \Sigma \rangle 
 =
  \iota \langle \Sigma^r, \Phi, (\Sigma')^r \rangle 
 = \iota(0) = 0 
 \]
 and so $\Phi^r\in I_2(W,V)$. The proof for $I_1$ is analogous.
\end{proof}

Consider now, for any pair of objects $W,V\in \FoamCata[\labdata]$, the $\kb$-linear map
\begin{align*}
    \phi_{W,V}\colon\; \FoamCata[\labdata](V,W) &\longrightarrow \FoamCata[\labdata](W,V)^*
    \\
    \underline{\Sigma}' &\longmapsto \langle \underline{\Sigma}',-\rangle
\end{align*} 

    We also consider the $\iota$-sesquilinear ($\iota$-antilinear in the first argument) pairing:
\begin{gather*}
    \langle -^r,-\rangle\colon \FoamCata[\labdata](W,V)\times\FoamCata[\labdata](W,V)\longrightarrow\kb \\
   (\underline{\Sigma}',\underline{\Sigma})\mapsto \langle{\underline{\Sigma}'}^r,\underline{\Sigma}\rangle = \phi_{W,V}({\underline{\Sigma}'}^r)(\underline{\Sigma})
\end{gather*}
where ${\underline{\Sigma}'}^r\in\FoamCata[\labdata](W^r,V^r)=\FoamCata[\labdata](V,W)$.

\begin{remark}
\label{rem:additionalprops}
In addition to the properties from \Cref{rem:Khprops}, Khovanov~\cite{khovanov2020universal} discusses the following two properties (see also \cite{BHMV}):
    \begin{itemize}
        \item[(I)]  For all $V,W\in \FoamCata[\labdata]$ the map $\phi_{W,V}$ is an isomorphism.
        \item[(F)] For all $V,W\in \FoamCata[\labdata]$ the module $\FoamCata[\labdata](W,V)$ is free of finite rank and the sesquilinear form  
        \[\langle -^r,-\rangle \colon \FoamCata[\labdata](W,V) \times \FoamCata[\labdata](W,V) \to \kb \]
        is unimodular.

    \end{itemize}
\end{remark}

\begin{lemma}
    Suppose that $\eval$ is involutive in the sense of \Cref{def:involutive}, with respect to an involution $\iota$. Then:
    \begin{enumerate}
        \item The map $\phi_{W,V}$ is injective for all $W,V\in \FoamCata[\labdata]$.
        \item If (F) holds, then $\phi_{W,V}$ is an isomorphism for all $W,V\in \FoamCata[\labdata]$, i.e.\ (I) holds.
    \end{enumerate}
\end{lemma}

\begin{proof}
    (1) Injectivity of $\phi_{W,V}$ follows immediately from nondegeneracy of the pairing $\langle -,-\rangle$.

    (2) The conditions imply that the pairing $\langle-^r,-\rangle$ induces an isomorphism 
    \[
        \FoamCata[\labdata](W,V) \xrightarrow{\cong} \FoamCata[\labdata](W,V)^*
    \]
    given by $\phi_{W,V}\circ -^r$. 
    It follows that $\phi_{W,V}=(\phi_{W,V}\circ -^r)\circ -^r$ is an isomorphism.
\end{proof}

\begin{lemma}\label{lem:FM'iffA}
    Suppose that $\eval$ is multiplicative and involutive. The following are equivalent:
    \begin{enumerate}
        \item[(FM')] Properties (F) and (M') hold.
        \item[(A)] Every $W\in \FoamCata[\labdata]$ is a biproduct  $W\cong \bigoplus_{i=1}^m \emptyset$ for some $m\in \N_0$.
    \end{enumerate}
\end{lemma}

\begin{proof}
    We first prove the implication (FM') $\Rightarrow$ (A) and, to this end, assume (F) and (M').
    Given $W\in \FoamCata[\labdata]$, let $m=\rk_{\kb}\FoamCata[\labdata](\emptyset,W)$ and choose a basis $\Sigma_1,\ldots,\Sigma_m$ of $\FoamCata[\labdata](\emptyset,W)$ with dual basis $\Sigma^1,\ldots,\Sigma^m$ of $\FoamCata[\labdata](\emptyset,W)^*$.  We have $\Sigma^i=\langle \tau_i,-\rangle$, for some $\tau_i\in \FoamCata[\labdata](\emptyset,W^r)=\FoamCata[\labdata](W,\emptyset)$ forming a basis that is dual in the sense that $\langle \tau_i,\Sigma_j\rangle=\delta_{ij}$.
    By property (M'), we can write $\id_W=\sum_{i,j=1}^m \lambda_{ij}\Sigma_i\circ \tau_j$ for some scalars $\lambda_{ij}\in \kb$. For $a,b\in\{1,\ldots,m\}$ we then compute:
    \begin{align*}
        \delta_{ab} &= \langle \Sigma_a, \tau_b \rangle = \langle \id_W, \Sigma_a\circ\tau_b\rangle = \langle \sum_{i,j=1}^m \lambda_{ij}\Sigma_i\circ \tau_j, \Sigma_a\circ\tau_b\rangle \\
        &= \sum_{i,j=1}^m \lambda_{ij} \langle \Sigma_i, \tau_b \rangle\, \langle \tau_j, \Sigma_a \rangle = \sum_{i,j=1}^m \lambda_{ij} \delta_{ib}\delta_{ja} =\lambda_{ab}.
    \end{align*}
Thus we obtain inverse isomorphisms
\begin{equation}
\label{eq:biproduct}
 \begin{tikzcd}
	W & {\bigoplus_{i=1}^m \emptyset}
	\arrow["f", shift left, from=1-1, to=1-2]
	\arrow["g", shift left, from=1-2, to=1-1]
\end{tikzcd}
, \quad f=(\tau_1,\ldots,\tau_m)^t , \quad g=(\Sigma_1,\ldots,\Sigma_m).
\end{equation}
that present $W$ as a biproduct as required by (A).
\smallskip 

    For the implication (A) $\Rightarrow$ (FM') we assume that (A) holds and already know that it implies (M) and hence (M'). To see (F), we may assume $W=\bigoplus_{i=1}^m \emptyset$ and $V=\bigoplus_{i=1}^n$ in $\FoamCata[\labdata]$ and recall from \Cref{lem:A_imply_M} that
       \[
        \FoamCata[\labdata](V,W) \cong \mathrm{Mat}(m\times n,\kb)
    \] is free of finite rank.
   The pairing corresponds to the pairing of matrices $(\phi_1,\phi_2)\mapsto \mathrm{tr}(\phi_1^t\circ \phi_2)$, which is clearly unimodular.\qedhere
\end{proof}

\subsection{Frobenius algebras and gradings}
\label{sec:Frob}
In this subsection we comment on how to incorporate gradings into the previous discussion. 

\newcommand{\Cob}{\mathrm{Cob}_2}

We let $\Cob$ denote the symmetric monoidal category of oriented 2-dimensional cobordisms between oriented closed 1-manifolds. Consider a $\kb$-linear labeling datum $\labdata$ as in Definition~\ref{def:labdata} and a closed foam evaluation, denoted by $\eval$. Let $x\in \labdata_1$ be a label. We have symmetric monoidal functors: 
\begin{equation}\label{eq:extract2DTQFT}
    \Cob \to \PreFoamCat^{\mathrm{nil}} \xrightarrow{\text{label by } x} 
    \PreFoamCat[\labdata] \to 
    \FoamCata[\labdata] \to
    \FoamCatb[\labdata]\to
    \kb\Mod
\end{equation}
given by interpreting compact oriented 1-manifolds as webs and 2-dimensional oriented cobordisms as foams, and subsequently labeling everything by $x$. Here, $\PreFoamCat^{\mathrm{nil}}\subset \PreFoamCat$ is the subcategory of nilvalent webs and prefoams. We let $\mathcal{Z}_x$  denote the composite functor. 

\begin{lemma}\label{lem:extractFAperLabel} Suppose that the foam evaluation is multiplicative and satisfies assumption (A)\footnote{This assumption is perhaps stronger than necessary, but will be satisfied in our examples of interest.} from Lemma~\ref{lem:FM'iffA}, then for every label $x\in \labdata_1$, the functor $\mathcal{Z}_x$ is symmetric monoidal, i.e.\ a 2-dimensional \emph{topological quantum field theory}. Consequently, the invariant $A_x:=\mathcal{Z}_x(\bigcirc)$ of the $x$-labeled loop web is a commutative Frobenius algebra free over $\kb$. 
\end{lemma}

\newcommand{\degr}{d}
\newcommand{\shift}{c}

We extend the labeling data to a graded setting.

\begin{definition}[Graded labeling data and degrees for trivalent and nilvalent foams]
\label{def:grlabdata}
    Let $\kb$ be a graded commutative ring. A \emph{graded\footnote{Throughout we only consider $\Z$-gradings.} $\kb$-linear labeling datum} for trivalent foams consists of a labeling datum $\labdata$ in the sense of \Cref{def:labdata}, such that for each $a\in \labdata_1$ the algebra of decorations $\Dec(a)$ is equipped with the structure of a grading, together with weights
    \[\alpha(k)\in \Z,\quad \alpha(k,l)\in \Z,\quad \alpha_1(k,l,h)\in \Z,\quad \alpha_2(k,l,h)\in \Z \]
    for, respectively, facets of label $k\in\labdata_1$, seams of labels $k,l,k+l\in \labdata_1$, and the two types\footnote{Compare \Cref{def:labclosedfoam}.} of vertices involving labels $k,l,h,k+l,l+h,k+l+h\in \labdata_1$. If $f$, $s$, $v$ are a facet, a seam, and a vertex with these labels, we also set
    \[\alpha(f):=\alpha(k),\quad \alpha(s):=\alpha(k,l),\quad \alpha(v):=\alpha_i(k,l,h),\]
    the latter with $i$ depending on the type of the vertex, subject to compatibility conditions expressed below.

    A \emph{graded $\kb$-linear labeling datum} $\labdata$ for nilvalent foams consists of a set $\labdata_1$ and for each $a\in \labdata_1$ the data of a graded commutative $\kb$-algebra $\Dec(a)$ and a weight $\alpha(a)\in \Z$ for facets of label $a$.

 The \emph{degree} of a combinatorial $\labdata$-foam (closed or with boundary) $\Sigma$ with facet set F and homogeneous decoration $\Dec(F)$ is then defined as:
\begin{equation}
\label{eq:foamdeg}
    \deg(\Sigma)= \deg(\Dec(F)) + \sum_f \alpha(f)\chi(f) - \sum_s \alpha(s)\chi(s) + \sum_v \alpha(v)\chi(v),
\end{equation}
 where $\chi(-)$ is the Euler characteristic from \Cref{rem:Euler}.
 \smallskip
 
 \Cref{def:labfoammorphism} now naturally yields a graded $\kb$-module $\HomPreFoamCat[\labdata](W_1,W_2)$ spanned by $\labdata$-prefoams with homogeneous decorations between closed combinatorial $\labdata$-webs $W_1,W_2$. 
The weights are required to be compatible so the composition operation from Construction \ref{constr:labfoamcomposition} is grading-preserving and, in particular, identity $\labdata$-foams have degree zero.
\smallskip 

Likewise, by considering only homogeneous decorations, $\HomPreFoamCat[\labdata](\emptyset,\emptyset)$ becomes a graded $\kb$-module. A closed foam evaluation for foams with graded $\kb$-linear labeling data is \emph{graded}, if 
\[
         \eval \colon \HomPreFoamCat[\labdata](\emptyset,\emptyset) \to \kb
         \]
is a map of graded $\kb$-modules, i.e.\ a grading-preserving $\kb$-linear map.
\end{definition}

Starting from a graded labeling data and a graded foam evaluation, the constructions of \Cref{sec:univ} can be carried out mutatis mutandis. In particular, we obtain symmetric monoidal graded $\kb$-linear categories $\PreFoamCat[\labdata]$, $\FoamCata[\labdata]$ and $\FoamCatb[\labdata]$. Here, as in the introduction, \emph{graded $\kb$-linear} means enriched in graded $\kb$-modules. When invoking property (A) from \Cref{lem:A_imply_M} in the graded setting, we assume that the structure morphisms of the biproduct are homogeneous. The functor $\webeval$ then also extends to the graded setting.

\begin{example}
For graded labeling data $\labdata$, consider labels $k,l\in \labdata_1$ and a $\labdata$-web $\Theta_{k,l}$ consisting of three edges labeled $k,l,k+l$, connected by a split vertex and a merge vertex. The identity foam on $\Theta_{k,l}$ being of degree zero imposes the compatibility condition:
\[
0 = 0 + \alpha(k) + \alpha(l) + \alpha(k+l) - \alpha(k,l) - \alpha(l,k)
\]
\end{example}

\begin{definition}\label{def:gradingFrob}
    Let $\kb$ be a graded commutative ring. A \emph{graded Frobenius algebra over $\kb$ of degree $d\in \Z$} is a Frobenius algebra $(A,m,\eta, \Delta,\varepsilon)$ over $\kb$ whose underlying $\kb$-module $A$ is graded and where the structure maps $m$, $\eta$, $\Delta$ and $\varepsilon$ are homogeneous of the following degrees
    \begin{align*}
        &\deg(m) = 0, &&\deg(\Delta)=d,\\
        &\deg(\eta)=0, &&\deg(\varepsilon)=-d.
    \end{align*}
\end{definition}

\begin{remark} We often consider non-negatively graded Frobenius algebras, for which the degree $d$ will also be non-negative. By shifting the underlying graded $\kb$-module \emph{down} by $c\in \Z_{\geq 0}$, one arrives at the notion of a $(c,d-c)$-graded Frobenius algebra discussed in \cite{ClivioGrFA}, for which the degrees of the structure maps are
    \begin{align*}
        &\deg(m) = c, &&\deg(\Delta)=d-c,\\
        &\deg(\eta)=-c, &&\deg(\varepsilon)=c-d.
    \end{align*}
Conversely, a $(c',d')$-graded Frobenius algebra corresponds to a graded Frobenius algebra of degree $d=c'+d'$ which has been shifted down by $c=c'$.

\end{remark}

With the above definitions, we can refine \Cref{lem:extractFAperLabel} as follows.

\begin{cor}
    Consider a graded foam evaluation for combinatorial $\labdata$-foams with graded $\kb$-linear labeling data. Suppose that the graded foam evaluation is multiplicative and satisfies property (A) from Lemma~\ref{lem:FM'iffA}. Then for every label $x\in \labdata_1$, it follows that $A_x=\mathcal{Z}_x(\bigcirc)$ is a shifted graded commutative Frobenius algebra free over $\kb$. Its degree is $\degr(x)=-2\alpha(x)$ and its shift is $\shift(x)=-\alpha(x)$. 
\end{cor}
\begin{proof}
    We consider the cup and cap cobordisms $\mathrm{cup} \colon \emptyset\to \bigcirc$ and $\mathrm{cap}\colon \bigcirc \to \emptyset$, each labeled by $x\in\labdata_1$, as morphisms of $\PreFoamCat[\labdata]$ and record their degrees $\alpha(x)\chi(\mathrm{cup})=\alpha(x)\chi(\mathrm{cap})=\alpha(x)$. The unit maps $\eta_x=\mathcal{Z}_x(\mathrm{cup})$ and counit maps $\varepsilon_x=\mathcal{Z}_x(\mathrm{cap})$ inherit these degrees and we have $\alpha(x)=\deg(\eta_x)=-\shift(x)$ and $\alpha(x)=\deg(\varepsilon_x)=\shift(x)-\degr(x)$.
\end{proof}

\begin{example}[Khovanov--Bar-Natan decorated surfaces]\label{ex:FAforBN}
    Recall from Example~\ref{ex:BNeval} the case of Khovanov--Bar-Natan decorated surfaces with $\kb$-linear labeling data $\labdata^A$ consisting of $\labdata_1^A=\{*\}$ and a commutative Frobenius algebra $\Dec(*)=A$. Then, by construction, we have $\mathcal{Z}_{*}(\bigcirc)=A$. In the graded case, we obtain a graded commutative Frobenius algebra over $\kb$ of degree $\degr(*)=-2\alpha(*)$, with downward shift $\shift(*)=\degr(*)/2=-\alpha(*)$. The example underlying the original construction of Khovanov homology is given by the commutative Frobenius algebra $\Z[x]/(x^2)$, which is free of rank 2 over $\kb=\Z$, with counit and comultiplication given by
    \begin{align*}
        \varepsilon\colon 1\mapsto 0,\; x\mapsto 1 \quad\quad\text{and}\quad\quad \Delta\colon 1\mapsto 1\otimes x + x\otimes 1,\; x\mapsto x\otimes x.
    \end{align*}
    Setting $\deg(1)=0$ and $\deg(x)=2$, the Frobenius algebra becomes graded of degree $\degr(*)=2$, and is then shifted down by $\shift(*)=1$. Indeed, $\mathcal{Z}_{*}(\bigcirc)$ then equals the unknot invariant of Khovanov homology, and every (undecorated) nilvalent foam is graded by its negative Euler characteristic.
    
    Another case of interest, for fixed $N\in \N$, has graded ground ring $\kb=\Z[a_1, \dots, a_{N}]$ with $\deg(a_i)=2i$, so that $A=\kb[x]/(x^N+a_{1}x^{N-1}+\dots +a_{N-1}x+a_{N})$ with $\deg(x)=2$ becomes a graded Frobenius algebra of degree $2N-2$ with respect to the $\kb$-linear counit $\varepsilon\colon x^k\mapsto \delta_{k,N-1}$, which is shifted down by $\shift=N-1$. This recovers the uncolored unknot invariant of the $\mathrm{GL}(N)$-equivariant $\glN$ link homology. Nilvalent foams (without decorations) are graded by their Euler characteristic scaled by $1-N$.
\end{example}

\begin{example}[Robert--Wagner $\glN$-foams]\label{ex:FAforRobertWagner}
    Let $N\in \N$ and recall Robert--Wagner $\glN$-foams from Example~\ref{ex:glneval} over the commutative ring $\kb=\Z$. Recall that the $\kb$-linear labeling data $\labdata^N$ has the underlying semigroup $\labdata_1^{N}=(\N_0,+)$. A grading on the ring of decorations $\Dec(n)=\Z[x_1,\ldots,x_n]^{S_n}$ is given by declaring that all variables $x_i$ have degree $2$. Weights for facets, seams, and vertices are given respectively by:
    \begin{align*}
        \alpha(k) =&\ -k(N-k), \\
         \alpha(k,l)=&\ -((k+l)(N-k-l)+kl), \\
        \alpha_i(k,l,h) =&\ -((k+l+h)(N-l-k-h)+ kl + lh + kh).
    \end{align*}
    These determine the degrees of combinatorial $\labdata$-foams by \eqref{eq:foamdeg} (cf.\ \cite[Definition 2.3]{RoW}). Observe that $k(N-k)$ is the complex dimension of $\mathrm{Gr}_{\mathbb{C}}(k,N)$, the Grassmannian of $k$-planes in $\mathbb{C}^N$, while the other weights are negative complex dimensions of 2- and 3-step partial flag varieties.

    The commutative Frobenius algebra $A_k$ associated to the label $k\in\N_0$, is given by the integer cohomology ring
    \begin{align*}
        A_k = H^*(\mathrm{Gr}_{\mathbb{C}}(k,N);\Z)
    \end{align*}
    With the above choice of grading of foams, the Frobenius algebra $A_k$ is graded, with degree $\degr(k)=\dim_\mathbb{R}(\mathrm{Gr}_{\mathbb{C}}(k,N))=2k(N-k)$ and downward shift by $c(k)=\frac{\degr(k)}{2}$. 
\end{example}

\section{A semistrict monoidal 2-category}\label{sec:2cat}

In this section we fix a commutative ring $\kb$,
a $\kb$-linear labeling datum $\labdata$ as in Definition~\ref{def:labdata} and a closed foam evaluation, denoted by $\eval$, 
which is multiplicative and such that property (A) from Lemma~\ref{lem:A_imply_M} is satisfied\footnote{And thus also
(M), (M'), (M'').}.
We will only consider the trivalent and nilvalent cases.

The goal of this section is to rigorously construct a locally $\kb$-linear semistrict monoidal 2-category $\Foambicat{\labdata}$ inspired by the following coarse description:
\begin{itemize}
\item Objects should be finite sequences of elements of $\labdata_1$, each equipped with orientation data \emph{inwards} or \emph{outwards}, on which the monoidal structure $\boxtimes$ acts by concatenation. We like to think of such sequences as encoding configurations of $\labdata_1$-labeled points in the unit interval $[0,1]$.
\item 1-morphisms should encode $\labdata$-labeled webs with boundary in the unit square $[0,1]^2$, mapping between two boundary sequences. These webs should be generated under $\boxtimes$ and composition $\hcompw$  by oriented identity webs, cups and caps as well as, in the trivalent case, trivalent merge and split vertices.
\item 2-morphisms between two parallel 1-morphisms given by webs $S$ and $T$ should encode linear combinations of $\labdata$-foams with corners, embedded in the unit cube $[0,1]^3$ with top boundary $T$ and bottom boundary $S$ with opposite orientation. 
\end{itemize}
We refer to \cite[\S 6]{2019arXiv190712194M} for a similar construction in the context of braided monoidal 2-categories.

\subsection{Objects and 1-morphisms} \label{subsec:objsAndOneMorphs}

We describe the objects and 1-morphisms of the to-be-constructed locally $\kb$-linear strict 2-category $\Foambicat{\labdata}$.

\begin{definition}[Objects] \label{def:obj2cat}
    The \emph{objects} of $\Foambicat{\labdata}$ are finite sequences of pairs in $\labdata_1\times \{\uparrow,\downarrow\}$, including the empty sequence $\emptyset$, whose entries we abbreviate to $\hat{x}:= (x,\uparrow)$ and $\widecheck{x} := (x,\downarrow)$ for $x\in \labdata_1$.
    The associative operation of concatenating sequences will be denoted $\boxtimes$. 
\end{definition}
For example, we have $
    \left(\hat{x}, \widecheck{y} \right) \boxtimes \left( \widecheck{z}\right):= \left(\hat{x}, \widecheck{y}, \widecheck{z}\right)$.

\newcommand{\capweb}{\mathrm{cap}}
\newcommand{\lcapweb}{\overleftarrow{\capweb}}
\newcommand{\rcapweb}{\overrightarrow{\capweb}}
\newcommand{\cupweb}{\mathrm{cup}}
\newcommand{\lcupweb}{\overleftarrow{\cupweb}}
\newcommand{\rcupweb}{\overrightarrow{\cupweb}}
\newcommand{\splitweb}{\widehat{\mathrm{split}}}
\newcommand{\mergeweb}{\widehat{\mathrm{merge}}}
\newcommand{\downsplitweb}{\widecheck{\mathrm{split}}}
\newcommand{\downmergeweb}{\widecheck{\mathrm{merge}}}

\newcommand{\genweb}[1]{#1\mhyphen\mathrm{GenWeb}}
\newcommand{\basweb}[1]{#1\mhyphen\mathrm{BasicWeb}}

\begin{definition}[1-morphisms]
\label{def:foamonemorphisms}
In the following, we define the \emph{1-morphisms} in $\Foambicat{\labdata}$. The \emph{source object} of a 1-morphism $W$ will be denoted $s(W)$ and the \emph{target object} $t(W)$.
\begin{enumerate}[(i)]
    \item The set $\genweb{\labdata}$ of \emph{generating $\labdata$-webs} consists of the following 
for all $x\in \labdata_1$
\begin{align*}
    \lcapweb_x\colon \left( \widecheck{x},\widehat{x}\right) \to \emptyset  && \lcupweb_x\colon \emptyset \to \left( \widehat{x},\widecheck{x}\right) \\
    \rcapweb_x\colon \left( \widehat{x},\widecheck{x}\right) \to \emptyset &&  \rcupweb_x\colon \emptyset \to \left( \widecheck{x},\widehat{x}\right)
\end{align*}
    if $\labdata$ is nilvalent. If $\labdata$ is trivalent, then for all $x,y\in \labdata_1$, the set $\genweb{\labdata}$ additionally contains:
\begin{align*}
   \splitweb_{x,y}\colon \left( \hat{x+y}\right) \to \left( \hat{x},\hat{y}\right) && \mergeweb_{x,y}\colon \left( \hat{x},\hat{y}\right) \to \left( \hat{x+y}\right)
\end{align*}

    \item The set $\basweb{\labdata}$ of \emph{basic $\labdata$-webs} consists of the following symbols for all objects $o_l,o_r\in \Foambicat{\labdata}$ and generating webs $W'\in \genweb{\labdata}$ with $W'\colon s(W')\to t(W')$:

\begin{align*}
    o_l\boxtimes W' \boxtimes o_r \colon o_l\boxtimes s(W') \boxtimes o_r \to o_l\boxtimes t(W') \boxtimes o_r.
\end{align*}

    \item A \emph{1-morphism} of $\Foambicat{\labdata}$ from $s$ to $t$ is a finite sequence $\left(W_k,\dots,W_1\right)$ of length $k\geq 1$ with entries given by 
     basic webs $W_i\in \basweb{\labdata}$ that are composable in the sense  that $s(W_{i+1})=t(W_i)$ for $1\leq i<k$ and $t=t(W_k)$ and $s=s(W_1)$. If $s=t$, we also allow the empty sequence as (identity) endomorphism and denote it by $\id_{s}$. 
    \item The composition $\hcompw$ of 1-morphisms in $\Foambicat{\labdata}$ is given by concatenating sequences. It is strictly associative and strictly unital.
     \end{enumerate}
\end{definition}

\begin{conv}[Graphical calculus]\label{conv:graphcal}
    We have the following graphical calculus for generating $\labdata$-webs, to be read from bottom to top. Here and in the following, when we depict $\labdata$-webs, we use the following coordinate axes:
    \begin{align*}
        \axes
    \end{align*}
    The 0-axis is the direction of $\boxtimes$, and the 1-axis is the direction of composition. For all $x\in \labdata_1$, we have
    \begin{align*}
        \lcapx{x} = \lcapweb_x && \lcupx{x}=\lcupweb_x\\
        \rcapx{x} = \rcapweb_x && \rcupx{x}=\rcupweb_x
    \end{align*}
    if $\labdata$ is nilvalent. If $\labdata$ is trivalent, we have additionally for $x,y\in \labdata_1$
    \begin{align*}
         \splitxy{x}{y} = \splitweb_{x,y} && \mergexy{x}{y} = \mergeweb_{x,y}.
    \end{align*}
Sometimes, it is useful to label strands not just by elements of $\labdata_1$, but by objects $s \in \Foambicat{\labdata}$. For instance, we draw the identity 1-morphism $\id_s\in \Hom_{\Foambicat{\labdata}}(s,s)$ as
    \begin{align*}
        \idonobj{s} = \id_s.
    \end{align*}
For objects $o_l,o_r\in \Foambicat{\labdata}$ and generating webs $W'\in \genweb{\labdata}$, we draw basic $\labdata$-webs as
    \begin{align*}
        \basicweb = o_l\boxtimes W' \boxtimes o_r.
    \end{align*}
\end{conv}

\begin{remark}
    As emphasized in \Cref{conv:graphcal} graphically, \Cref{def:foamonemorphisms} includes only \emph{upward-oriented} or \emph{progressive} merge and split webs as generators alongside cups and caps. We chose this presentation for reasons of compatibility with the formalism of factorizing families of type A perverse schobers~\cite{dyckerhoff2025perverseschoberscoxetertype}. Rotated merge and split webs can be obtained as composites, see e.g.\ \Cref{constr:mates} below. An alternative definition including the downward-oriented merge and split vertices as generators is commented on in \Cref{rem:strict}.
\end{remark}

\begin{conv}
\label{conv:vertedges}
    In the following, we introduce a few useful concepts for working with 1-morphisms in $\Foambicat{\labdata}$. Let $W=\left(W_k,\dots,W_1\right)\colon s\to t$ be an arbitrary 1-morphism. Here we allow $k=0$ for $s=t$.
    \begin{itemize}
        \item The set $V(W)$ of \emph{vertices} is the collection of basic $\labdata$-webs appearing in $W$, whose underlying generating $\labdata$-web is of type $\splitweb_{*,*}$ or $\mergeweb_{*,*}$.
        \item The set $E(W)$ of \emph{edges} is defined as follows. 
        Consider the set $\mathcal{E}$ of positions in the sequences forming source and target objects of the basic webs $W_i$. Its elements are encoded as pairs of the form $(p,s(W_i))$ with $p\in\{1,\dots, \mathrm{len}(s(W_i))\}$ or $(p,t(W_i))$ with $p\in\{1,\dots, \mathrm{len}(t(W_i))\}$ where $\mathrm{len}(-)$ computes the length of a sequence.

       We define $E(W)$ as the set of equivalence classes of the equivalence relation on $\mathcal{E}$ generated by:
            \begin{itemize}
                \item $(p,s(W_i)) \sim (p+1,s(W_i))$ (resp.\ $(p,t(W_i)) \sim (p+1,t(W_i))$) if they appear paired in any generator of the form $\lcapweb_x,\rcapweb_x$ (resp.\ $\lcupweb_x, \rcupweb_x$) for some $x\in \labdata_1$.
                \item $(p,s(W_{i+1}))\sim (p,t(W_i))$ identifying source and target of composed basic webs.
                 \item $(p,s(W_i))\sim(p,t(W_i))$ if $p$ is within the joint $o_l$ or $o_r$ parts of $s(W_{i})$ and $t(W_i)$. 
            \end{itemize}
        \item 
        For an edge $e$, we write $n_\partial(e)$ for the number of representing positions in $s$ or $t$. Note that $n_\partial(e)\in\{0,1,2\}$ for any edge $e$. We distinguish types of edges by $n_\partial(e)$. An edge $e$ is considered an \emph{internal edge} if $n_\partial(e)=0$. Otherwise it is considered a \emph{boundary edge}.
        A boundary edge $e$ is called a \emph{boundary arc} if $n_\partial(e)=1$. If this position is in $s$ resp.\ $t$, we call it a \emph{source arc} resp.\ \emph{target arc}. Note that then there is also a representing position in the source or target of a generating $\labdata$-web of type $\splitweb_{*,*}$ or $\mergeweb_{*,*}$. 
        If $n_\partial(e)=2$, we distinguish the following cases. It is considered a \emph{source turnback} resp.\ \emph{target turnback} if both representing positions are in $s$ resp.\ $t$. If $e$ has one representing position in $s$ and one in $t$ it is called a \emph{passing edge}. An internal edge is called an \emph{internal arc} if a representing position appears in the source or target of a generating $\labdata$-web of type $\splitweb_{*,*}$ or $\mergeweb_{*,*}$. Otherwise it is called an \emph{(internal) loop}. 
        Internal arcs are oriented from the vertex where they are represented by a position in the target object to the vertex where they are represented by a position in the source object. Boundary arcs, passing edges, and source turnbacks resp.\ target turnbacks are oriented from the position corresponding to $\hat{x}$ to the position corresponding to $\widecheck{x}$ in $s$ resp.\ $t$, where $x\in \labdata_1$.
        
        The following picture features examples of all types of edges mentioned above:
        \[
            \webconventionLabels
        \]
        \item The \emph{label} of an edge $e$ is defined to be the underlying label $x\in \labdata_1$ of the symbol $\hat{x}$ or $\widecheck{x}$ appearing at a representing position of the edge. We write $\labelfunction(e)=x$.
    \end{itemize}
\end{conv}

\begin{construction}\label{constr:absCWebOfMorphism}
    To any 1-morphism $W=\left(W_k,\dots,W_1\right) \colon \emptyset \to \emptyset$ in $\Foambicat{\labdata}$ we associate a combinatorial closed $\labdata$-web $\abs{W}\in \labWebSet{\labdata}$ in the sense of Definition~\ref{def:labclosedweb}.

    \begin{itemize}
        \item The set of vertices is $V^\partial:=V(W)$.
        \item The set of edges is $E^\partial:=E(W)$, partitioned into arcs and loops as in \Cref{conv:vertedges}.
        \item The labels on edges and orientations on arcs are defined as in \Cref{conv:vertedges}.
        \item Around vertices of type $\splitweb_{x,y}$ and $\mergeweb_{x,y}$, the edges are cyclically ordered by $([\hat{x}],[\hat{y}],[\hat{x+y}])$\footnote{In the embedded context, this is compatible with the left-hand rule relating seam orientations, cyclic orderings and the ambient orientation---as done by Khovanov \cite{Kho3} and Robert--Wagner \cite{RoW}.}.
    \end{itemize}
    
    If $W'\colon \emptyset \to \emptyset$ is another 1-morphism in $\Foambicat{\labdata}$, then it is clear that
    \[
    \abs{W' \hcompw W} = \abs{W'}\sqcup \abs{W}.
    \]
\end{construction}
\smallskip

In Section \ref{subsec:duals}, we will show that the semistrict monoidal 2-category $\Foambicat{\labdata}$ has duals and adjoints following \cite{barrett2018gray}. For this, we introduce two operations on $\Foambicat{\labdata}$ denoted by $\#$ and $*$.
\begin{construction}[$\#,*$ on objects] \label{constr:dualobj}
    We define two involutions $\#,*$ on the set of objects of $\Foambicat{\labdata}$. Given an object $s\in\Foambicat{\labdata}$, the involution $\#\colon s \mapsto s^\#$ reverses the order of sequences and interchanges entries $\hat{x}\leftrightarrow\widecheck{x}$, while the involution $*$ is defined as the identity on objects.
\end{construction}

\begin{remark}
\label{rem:reverseedge}
    For an object $s\in \Foambicat{\labdata}$, we draw the identity $\id_{s^\#}$ on $s^\#$ in two ways as
    \begin{align*}
        \idonobjdown{s}=\id_{s^\#} = \;\;\idonobj{s^\#}.
    \end{align*}
\end{remark}

\begin{construction} \label{constr:fold}
    Let $s\in \Foambicat{\labdata}$ be an object. We define the 1-morphism $\lcupweb_s\colon \emptyset \to s \boxtimes s^\#$ by induction on the length of the sequence underlying $s$. For length $0$ we declare $\lcupweb_\emptyset= \id_\emptyset\colon \emptyset \to \emptyset$. For length $1$ we set:
    \[ \lcupweb_s = 
    \begin{cases} 
    \lcupweb_x & \text{if } s=(\hat{x}) \text{ for } x\in \labdata_1,\\
    \rcupweb_x & \text{if } s=(\widecheck{x}) \text{ for } x\in \labdata_1.
    \end{cases}
    \] 
    For length $\ell\geq 2$ we write $s=s_1\boxtimes s'$ for $s'$ of length $\ell-1$ and declare:
    \[ \lcupweb_s = 
         (s_1 \boxtimes \lcupweb_{s'} \boxtimes s^\#_1 ) \hcompw \lcupweb_{s_1}
    \] 
Additionally, we induce the notation $\rcupweb_s:= \lcupweb_{s^\#}\colon \emptyset \to s^\# \boxtimes s$ and note that this is consistent with our previous notation for words of length 1. Completely analogously, we define 
\begin{align*}
    \lcapweb_s\colon  s^\# \boxtimes s\to \emptyset\quad \text{ and }\quad \rcapweb_s:=\lcapweb_{s^\#}\colon s \boxtimes s^\#\to \emptyset.
\end{align*}
We draw the above 1-morphisms as follows:
\begin{align*}
    \lcupx{s} = \lcupweb_s && \lcapx{s}=\lcapweb_s\\
    \rcupx{s} = \rcupweb_s && \rcapx{s}= \rcapweb_s
\end{align*}
Labels $x\in \labdata_1$ convert by default to labels by the object $\widehat{x}$ with the same orientation and, equivalently, to labels by the object $\widecheck{x}$ with opposite orientation.
\end{construction}

\newcommand{\rmate}[1]{#1^\#}
\newcommand{\lmate}[1]{{}^\##1}

\begin{construction}[Mates]\label{constr:mates}
Let $W\colon s\to t$ be a 1-morphism in $\Foambicat{\labdata}$. Then we define the (right-)mate of $W$ to be the 1-morphism\footnote{Note that our terminology agrees with the one of Douglas--Reutter in \cite[Section 2.2.2]{douglas2018fusion}. Their convention uses the same composition, but the opposite tensor product (see \cite[p.11]{douglas2018fusion}). Hence, translating between the graphical calculi for 1-morphisms corresponds to a reflection across the 1-axis.}

\[
\rmate{W}:=
 (\lcapweb_{t} \boxtimes s^\#) \hcompw (t^\# \boxtimes W \boxtimes s^\#) \hcompw (t^\# \boxtimes \lcupweb_{s}) \colon t^\# \to s^\#
\]
and the (left-)mate of $W$ to be the 1-morphism
\[
\lmate{W}:=
 (s^\#\boxtimes \rcapweb_{t}) \hcompw (s^\# \boxtimes W \boxtimes t^\#) \hcompw (\rcupweb_{s} \boxtimes t^\#) \colon t^\# \to s^\#.
\]
We draw the left- and right-mates as follows\footnote{Mnemonic: the superscript appears on the side, on which the wire connecting to the source object of $W$ has been bent upwards.}:
    \begin{align*}
        \leftmate{W}={}^{\#}W && \rightmate{W}=W^{\#}
    \end{align*}
\end{construction}

\begin{construction}[$*$ on 1-morphisms]\label{constr:starOn1morphs}
    We extend the two operation $*$ to the set of 1-morphisms $W$ of $\Foambicat{\labdata}$. We let $*\colon W \mapsto W^*$ be the involution on the set of 1-morphisms which is defined by acting on generating $\labdata$-webs by swapping
    \[
    \lcapweb_x \xleftrightarrow{*} \rcupweb_x, \quad \rcapweb_x \xleftrightarrow{*} \lcupweb_x, \quad \splitweb_{x,y} \xleftrightarrow{*} \mergeweb_{x,y}
    \]
    and then extending to basic $\labdata$-webs by setting
    \[(o_l\boxtimes W \boxtimes o_r)^*:= (o_l\boxtimes W^* \boxtimes o_r)\]
    and acting on general morphisms, i.e.\ $\hcompw$-composites of basic webs, term-wise and reversing the $\hcompw$-order\footnote{In the graphical calculus, this corresponds to a reflection in the $1$-direction, followed by changing the orientation of all web edges. 
    }:
    \[ (W_k\hcompw \cdots \hcompw W_1)^* := W_1^*\hcompw \cdots \hcompw W_k^*.\]
\end{construction}

\begin{remark}\label{rem:hashContravariantComp}
    The left- or right-mate operation on 1-morphisms will also act contravariantly in the $\hcompw$-direction, but only up to 2-isomorphisms. For the left-mate, these will be described in Remark~\ref{rem:HashAsFunctor}.
\end{remark}

\subsection{2-morphisms and vertical composition}
\label{sec:vertcomp}
In this subsection, we define the small $\kb$-linear category $\Foambicathom{\labdata}{t}{s}$ for any pair of objects $s,t$ of $\Foambicat{\labdata}$ which will form the category of 1-morphisms $s\to t$ in $\Foambicat{\labdata}$. The composition in $\Foambicathom{\labdata}{t}{s}$ will constitute the vertical composition of 2-morphisms in $\Foambicat{\labdata}$.

\begin{construction}[Internal hom]
\label{constr:inthom}
    Let $s,t$ be objects of $\Foambicat{\labdata}$ and $S,T\colon s \to t$ be 1-morphisms in $\Foambicat{\labdata}$. We declare the \emph{internal hom} from $S$ to $T$ to be the 1-morphism:
    \[
    \inthom(S,T):=  \rcapweb_t \hcompw (T \boxtimes t^\#) \hcompw (s \boxtimes S^{*\#}) \hcompw \lcupweb_s \;\colon\; \emptyset \to \emptyset
    \]
    In the graphical calculus, $\inthom(S,T)$ is depicted as:
    \begin{align*}
        \internalhom{S}{T}
    \end{align*}
\end{construction}

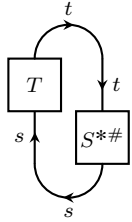
\begin{construction}[Vertices and edges of internal homs]
\label{constr:edgeinthom}
    Let $S,T\colon s\to t$ be 1-morphisms in $\Foambicat{\labdata}$. 
    There is a bijection 
    \begin{align*}
        V(\inthom(S,T)) \cong V(S^{*\#}) \sqcup V(T)
    \end{align*}
    for the vertices in the internal hom of $S$ and $T$. 
    For the edges, we write 
    \[ 
    E(S) = E(\topcirc{S})\sqcup E(\partial S)
    \]
    where $E(\topcirc{S})$ is the set of (internal) loops and internal arcs and $E(\partial S)$ is the set of boundary arcs, boundary turnbacks, and passing edges. We write \begin{align*}
        E(T|S):=E(\inthom(S,T)).
    \end{align*}
    Note that $E(T|S)$ only contains internal edges. There is a surjection
    \begin{equation}
        \label{eq:surjwebglue}
    \pi\colon E(T)\sqcup E(S^{*\#}) \rightarrow E(T|S) 
     \end{equation}
    which is injective on $E(\topcirc{T})\sqcup E(\topcirc{S^{*\#}})$, and sends $E(\partial T)\sqcup E(\partial S^{*\#})$ to the set of equivalence classes of the equivalence relation generated by the relation: The edges $e\in E(\partial T)$ and $e' \in E(\partial S^{*\#})$ are related if they share a position in $s$ or $t$.
    The image of $E(\partial T)\sqcup E(\partial S^{*\#})$ consists of the set of newly created internal loops and internal arcs. 
    We say that $\overline{e}\in E(T|S)$ \emph{contains} all edges $e\in E(T)\sqcup E(S^{*\#})$ in its preimage under \eqref{eq:surjwebglue}. 
    \end{construction}

\begin{lemma}\label{lem:SameAbsIntHom}
    Let $s,t$ be objects, and $S,T\colon s\to t$ be 1-morphisms in $\Foambicat{\labdata}$. The following internal homs have the same underlying combinatorial closed $\labdata$-webs:
    \begin{align*}
     \inthom(S,T),\quad \inthom(\id_s, S^*\hcompw T), \quad \inthom(\id_t, T\hcompw S^*), \\\inthom(T^*,S^*),\quad \inthom(S\hcompw T^*,\id_s),\quad \inthom(T^*\hcompw S,\id_t).
    \end{align*}
    Furthermore, we have
    \begin{align*}
        \abs{\inthom(S,T)^{*\#}} = \abs{\inthom(S^{*\#},T^{*\#})}.
    \end{align*}
\end{lemma}
\begin{proof}
    Applying Construction \ref{constr:absCWebOfMorphism} to the internal homs defined in  Construction \ref{constr:inthom}, the underlying combinatorial closed $\labdata$-webs are obtained from $S^*$ and $T$ by identifying their boundary edges that share a position in $s$ and $t$. A similar consideration applies to the last equation, here with $S$ and $T^*$.
\end{proof}
\begin{remark}\label{rem:rAndBend}
    For a 1-morphism $W\colon \emptyset \to \emptyset$ in $\Foambicat{\labdata}$, we have 
    \begin{align*}
        \abs{W^{*\#}} = \abs{W}^r,
    \end{align*}
    where $-^r$ is the involution on combinatorial closed webs from \Cref{def:combclosedweb2}, which reverses the orientation of all edges.
    With the last statement of Lemma~\ref{lem:SameAbsIntHom}, we obtain
    \begin{align*}
        \abs{\inthom(S,T)}^r = \abs{\inthom(S^{*\#},T^{*\#})}.
    \end{align*}
\end{remark}

\begin{construction}[Internal identity]
\label{constr:intid}
        Let $s,t$ be objects of $\Foambicat{\labdata}$ and $S\colon s \to t$ a 1-morphism in $\Foambicat{\labdata}$.  We shall describe a combinatorial $\labdata$-foam
    \[
    \mathrm{cup}_S\colon \emptyset \to \abs{\inthom(S,S)}.
    \]
    \textbf{Vertices:} The foam $\mathrm{cup}_S$ has no internal vertices. The boundary vertices are partitioned as 
    \[
    V^\partial = V(S) \sqcup V(S^{*\#}).
    \]
    \textbf{Seams and edges:} There are no loop seams. The canonical bijection $ V(S) \cong V(S^{*\#})$, associating to each $v$ its partner $v^{*\#}$, is used to define arc seams $E=E_{\mathrm{arc}}:=\{s_v\}_{v\in V(S)}$, where, for each merge/split vertex $v\in V^\partial_2$, we define $s_{v}$ as the seam with $V(s_{v})=\{v,v^{*\#}\}$, oriented from the merge vertex to the split vertex. By Construction \ref{constr:edgeinthom}, the boundary edges are given by
    \begin{align*}
        E^\partial = E(S|S) = E(\topcirc{S}) \sqcup E(\topcirc{S^{*\#}}) \sqcup \pi_{S|S}(E(\partial S)\sqcup E(\partial S^{*\#})).
    \end{align*}
    \textbf{Facets:} We use the canonical bijection 
    \[ E(S)\cong E(S^{*\#}), \quad e\mapsto e^{*\#}\]
    to describe the set of facets, parametrized by edges $e\in E(S)$:
    \begin{itemize}
        \item For every (internal) loop $e\in E(\topcirc{S})$, consider its paired partner $e^{*\#}\in E(\topcirc{S^{*\#}})$. We have a facet $f(e)\in F_{0,2}$ with label $\labelfunction(e)=\labelfunction(e^{*\#})$ and boundary words $((e), (e^{*\#}))$.
        \item For every arc $e\in E(S)$, there are at most two edges of $\abs{\inthom(S,S)}$ containing $e$ or $e^{*\#}$. Together with at most two seams in $E$, these form a minimal cycle which we equip with the orientation inherited from $e$ (cf.\ \Cref{rem:cyclicSign}) and use as boundary word for a facet $f(e)\in F_{0,1}$ with label $\labelfunction(e)$.\footnote{There are two cases: (1) a boundary arc creates a single internal arc under gluing, the endpoints of which are joined by a single seam, the facet attaches along this arc and seam. (2) an internal arc with its duplicate in $S^{*\#}$ gives rise to two internal arcs upon gluing, whose endpoints are joined in pairs by two seams. In this case the boundary word of the facet is of length four.}
        \item For every passing edge or boundary turnback $e\in E(\partial S)$, there is an internal loop $\overline{e}$ in $\abs{\inthom(S,S)}$ containing $e$ and $e^{*\#}$. We have a facet $f(e)\in F_{0,1}$ with label $\labelfunction(e)$, with boundary word $(\overline{e})$, and with orientation inherited from $e$.
    \end{itemize}
 The cyclic ordering of all facets around seams is inherited from $\inthom(S,S)$.
\end{construction}

  \begin{construction}[Internal vertical composition]  
  \label{constr:intcomp}
    Let $s,t$ be objects and $S,T,U\colon s \to t$ be 1-morphisms in $\Foambicat{\labdata}$. We shall describe a combinatorial $\labdata$-foam:
    \[
    \mathrm{vert}_{U,T,S} \colon \abs{\inthom(T,U)} \sqcup \abs{\inthom(S,T)} \to \abs{\inthom(S,U)},
    \]
    where the underlying internal homs in the source and target are depicted as
    \begin{align*}
        \internalhom{T}{U},\quad \internalhom{S}{T} \quad \quad \text{and}\quad \quad \internalhom{S}{U}
    \end{align*}
    \textbf{Vertices:} The foam $\mathrm{vert}_{U,T,S}$ has no internal vertices and the boundary vertices are partitioned as 
    \[
    V^\partial = V(U) \sqcup V(U^{*\#}) \sqcup V(T) \sqcup V(T^{*\#}) \sqcup V(S) \sqcup V(S^{*\#}).
    \]
    \textbf{Seams and edges:} There are no loop seams. The arc seams are given by $E=E_{\mathrm{arc}}=\{s_v\}_{v\in V(U) \sqcup V(T) \sqcup V(S)}$ where, for each merge/split vertex $v\in V^\partial_2$, we define $s_{v}$ as the seam with $V(s_{v})=\{v,v^{*\#}\}$, oriented from the merge vertex to the split vertex 
    The boundary edges of $\mathrm{vert}_{U,T,S}$ are 
    \begin{align*}
        E^\partial = E(\abs{\inthom(T,U)}^r) \sqcup E(\abs{\inthom(S,T)}^r) \sqcup E(\abs{\inthom(S,U)}).
    \end{align*}
    By Construction \ref{constr:edgeinthom} and Remark~\ref{rem:rAndBend}, we have 
    \begin{align*}
        E^\partial =&\ E(\topcirc{U}) \sqcup E(\topcirc{U^{*\#}}) \sqcup E(\topcirc{T}) \sqcup E(\topcirc{T^{*\#}})  \sqcup E(\topcirc{S}) \sqcup E(\topcirc{S^{*\#}}) \\
        &\sqcup \pi\left(E(\partial U^{*\#})\sqcup E(\partial T)\right) 
        \sqcup  \pi\left(E(\partial T^{*\#})\sqcup E(\partial S)\right) \sqcup 
        \pi\left(E(\partial U)\sqcup E(\partial S^{*\#})\right).
    \end{align*}
    \textbf{Facets:} The facets will be parametrized by $E^\partial$, using the 
    canonical bijections
    \[ E(U)\cong E(U^{*\#}), \quad E(T)\cong E(T^{*\#}), \quad E(S)\cong E(S^{*\#}), \quad e \mapsto e^{*\#}.\]
    For this, consider the equivalence relation generated by the following: two boundary edges $e,e'\in E(\partial S)\sqcup E(\partial T)\sqcup E(\partial U)$ are related if they share a position in $s$ or $t$. For each equivalence class $P$, consider the subset
    \[
            E^{\partial}_P:=\{\overline{e} \,|\, e\in P\}\sqcup \{\overline{e^{*\#}} \,|\, e\in P\}\subset E^\partial,
    \]
    where $\overline{e}\in E^\partial$ (resp.\ $\overline{e^{*\#}}$) is the image of $e$ (resp.\ $e^{*\#}$) under the appropriate surjection $\pi$ from \eqref{eq:surjwebglue}.
    Note that $E^\partial$ is given by the disjoint union of the internal edges of $S,T,U$ and their $*\#$ versions, and the set $\bigsqcup_{P}E^\partial_P$. The facets of $\mathrm{vert}_{U,T,S}$ are given by the following:
    \begin{itemize}
        \item For every (internal) loop $e\in E(\topcirc{U})\sqcup E(\topcirc{T})\sqcup E(\topcirc{S})$, consider its canonically identified partner $e^{*\#}\in E(\topcirc{U^{*\#}})\sqcup E(\topcirc{T^{*\#}})\sqcup E(\topcirc{S^{*\#}})$. We have a facet $f(e)\in F_{0,2}$ with label $\labelfunction(e)$ and boundary words $((e), (e^{*\#}))$.
        \item For every (internal) arc $e\in E(\topcirc{U})\sqcup E(\topcirc{T})\sqcup E(\topcirc{S})$, we have a facet $f(e)\in F_{0,1}$ with label $\labelfunction(e)$ and boundary word given by the cycle formed by $e$, $e^{*\#}$ and the two seams belonging to the vertices of $e$, with orientation inherited from $e$.        

        \item Each equivalence class $P$ corresponds to a facet $f(P)$ according to the following three cases:
        \begin{enumerate}[(i)]
            \item $E^\partial_P$ does not contain any loops, only arcs. Then, we have a facet $f(P)\in F_{0,1}$ with label $\labelfunction(e)$, for any $e\in P$, and boundary word given by the unique minimal cycle composed out of the arcs of $E^\partial_P$ and seams in $E$, with orientation inherited from $\overline{e}$.
         \item $E^\partial_P$ contains only loops and no arcs. Let $E^\partial_P=\{\overline{e_1},\ldots,\overline{e_{|E^\partial_P|}}\}$ be these loops. Then, we have a facet $f(P)\in F_{0,|E^\partial_P|}$ with label $\labelfunction(e_i)$, for any $i$, and boundary words $((\overline{e_1}),\ldots,(\overline{e_{|E^\partial_P|}}))$.
         \item $E^\partial_P$ contains both loops and arcs. Then we have a facet $f(P)\in F_{0,\ell+1}$, where $\ell$ is the number of loops in $E^\partial_P$. The facet has label $\labelfunction(e)$, for any $e\in P$, and boundary words given by the union of the boundary words from the previous two cases.
        \end{enumerate}
       \end{itemize}
    The cyclic ordering of all facets around seams is inherited from the boundary web.
\end{construction}

\begin{definition}\label{def:2homsAndComposition}
Let $s,t$ be objects of $\Foambicat{\labdata}$ and $S,T\colon s \to t$ 1-morphisms in $\Foambicat{\labdata}$. We declare 
\begin{equation}
\label{eqn:verthom}
\Hom_{\Foambicathom{\labdata}{t}{s}}(S,T):= \webeval(\abs{\inthom(S,T)}) 
\end{equation}

If $U\colon s\to t$ is another 1-morphism in $\Foambicat{\labdata}$, then we define the $\kb$-linear \emph{vertical composition} map:
\begin{equation}
\label{eqn:vertcomp}
\vcomp\colon \Hom_{\Foambicathom{\labdata}{t}{s}}(T,U) \otimes \Hom_{\Foambicathom{\labdata}{t}{s}}(S,T) \to \Hom_{\Foambicathom{\labdata}{t}{s}}(S,U)
\end{equation}
to be the map induced by the internal composition $\mathrm{vert}_{U,T,S}$ under the functor $\webeval$.

If $T=S$, we define the element
\begin{equation}
\label{eqn:vertid}
\id_S:=\webeval(\mathrm{cup}_S)(1)\in \Hom_{\Foambicathom{\labdata}{t}{s}}(S,S)
\end{equation}
and call this the \emph{identity foam} on $S$.
\end{definition}

\begin{prop}\label{prop:foamHomcat}
   Let $s,t$ be objects of $\Foambicat{\labdata}$, then the following data defines a small $\kb$-linear category, which we will denote by $\Foambicathom{\labdata}{t}{s}$: 
\begin{itemize}
   \item The objects are the 1-morphisms from $s$ to $t$ in $\Foambicat{\labdata}$ as defined in \Cref{def:foamonemorphisms}.
    \item The $\kb$-module of morphisms $\Hom_{\Foambicathom{\labdata}{t}{s}}(S,T)$ is defined in \eqref{eqn:verthom}.
   \item The composition law for morphisms is given by \eqref{eqn:vertcomp}.
   \item For each 1-morphism $S\colon s\to t$ in $\Foambicat{\labdata}$, the identity morphism $\id_S\in \Hom_{\Foambicathom{\labdata}{t}{s}}(S,S)$ on $S$ is given in \eqref{eqn:vertid}.
\end{itemize}
\end{prop}
\begin{proof}
We only need to show unitality and associativity of vertical composition. 
For unitality, let $s$ and $t$ be objects and $S,U\colon s\to t$ be 1-morphisms. The following equality of combinatorial $\labdata$-foams holds
    \begin{align}\label{eq:unitVertComp}
        \mathrm{vert}_{U,S,S} \circ (\id_{\abs{\inthom(S,U)}}\sqcup \mathrm{cup}_{S}) = \id_{\abs{\inthom(S,U)}}\colon \abs{\inthom(S,U)}\to \abs{\inthom(S,U)}.
    \end{align}

    To show the equality, we need to provide bijections between the sets of vertices, edges, seams, and facets, which are consistent with orientations, cyclic orderings and labels.

    For example, for the set of seams, the above equality can be verified as follows. The seams of $\mathrm{vert}_{U,S,S}$ and $\mathrm{cup}_S$ are given respectively by
    \[
        E_{\mathrm{vert}_{U,S,S}}=\{s'_v\}_{v\in V(U)}\sqcup \{s_v\}_{v\in V(S)} \sqcup \{s_v\}_{v\in V(S)}, \quad
        E_{\mathrm{cup}_{S}}=\{s_v\}_{v\in V(S)}.
    \]
    The composition of combinatorial $\labdata$-foams from Construction \ref{constr:labfoamcomposition} joins the three seams $s_v$ belonging to the same vertex $v\in V(S)$. Therefore, the sets of seams of the two sides of \eqref{eq:unitVertComp} agree.
    The equality of the sets of facets can be proved similarly, and it can be shown to be compatible with the orientation, cyclic ordering and labels.
    As a consequence, for any $\Sigma\in \Hom_{\Foambicathom{\labdata}{t}{s}}(S,U)$, we have
    \begin{align*}
        \webeval(\mathrm{vert}_{U,S,S})(\Sigma\otimes \id_S) = \Sigma \in \Hom_{\Foambicathom{\labdata}{t}{s}}(S,U).
    \end{align*}
    Unitality on the left is similar.

    For associativity, one has to check that for all $S,T,U,V\colon s\to t$, as combinatorial $\labdata$-foams, we have
    \[
        \mathrm{vert}_{V,T,S}\circ\mathrm{vert}_{V,U,T}=\mathrm{vert}_{V,U,S}\circ\mathrm{vert}_{U,T,S},
    \]
    that is, the sets of vertices, seams and facets agree, as done for the unitality above. Then, associativity follows by applying $\webeval$. 
\end{proof}

\subsection{Horizontal composition}\label{sec:horcomp}

\begin{construction}[Vertices and edges of composed webs]
\label{constr:edgesOfComp}
    Let $S\colon s\to t$ and $S'\colon t\to u$ be 1-morphisms in $\Foambicat{\labdata}$.
    We describe the sets of vertices and edges of $S'\hcompw S$ in terms of those of $S'$ and $S$.
    For the vertices, we have a bijection
    \[
        V(S'\hcompw S)\cong V(S')\sqcup V(S).
    \]
    For the edges, there is a surjection
    \begin{equation}
        \rho\colon E(S')\sqcup E(S) \rightarrow E(S'\hcompw S)
    \end{equation}
    which is injective on $E(\topcirc{S'})$ and $E(\topcirc{S})$, and sends $E(\partial S') \sqcup E(\partial S)$ to the set of equivalence classes with respect to the equivalence relation generated by the relation: Two edges $e'\in E(\partial S')$ and $e\in E(\partial S)$ are related if they share a position in $t$. 
\end{construction}

\begin{construction}[Internal horizontal composition]
\label{constr:intHorComp}
    Let $s,t,u$ be objects in $\Foambicat{\labdata}$ and $S,T\colon s\to t$ and $S',T'\colon t\to u$ be 1-morphisms in $\Foambicat{\labdata}$. In the following, we construct a combinatorial $\labdata$-foam
    \[
    \mathrm{hor}_{T,S}^{T',S'} \colon \abs{\inthom(S',T')} \sqcup \abs{\inthom(S,T)} \to \abs{\inthom(S'\hcompw S, T'\hcompw T)},
    \]
    where the underlying internal homs in the source and target are depicted as:
    \[
    \internalhomObj{S'}{T'}{t}{u} ,\quad \internalhomObj{S}{T}{s}{t} \quad\quad \text{and} \quad\quad \internalhomObjWide{(S'\hcompw S)}{T'\hcompw T}{s}{u}
    \]
    \textbf{Vertices:} There are no internal vertices. 
    By Remark~\ref{rem:rAndBend} and Constructions \ref{constr:edgeinthom} and \ref{constr:edgesOfComp}, the boundary vertices of $\mathrm{hor}_{T,S}^{T',S'}$ are partitioned as
    \[
    V^\partial = V(T') \sqcup V({T'}^{*\#}) \sqcup V(T) \sqcup V({T}^{*\#}) \sqcup V(S') \sqcup V({S'}^{*\#}) \sqcup V(S) \sqcup V(S^{*\#}).
    \]
    \textbf{Seams and edges:} 
    There are no loop seams. The arc seams are given by
    \[
    E = E_{\mathrm{arc}} = \{s_v\}_{v\in  V(T') \sqcup V(T) \sqcup V(S') \sqcup V(S)}
    \] 
    where, for each merge/split vertex $v\in V^\partial_2$, we define $s_{v}$ as the seam with $V(s_{v})=\{v,v^{*\#}\}$, oriented from the merge vertex to the split vertex. 
    By Constructions~\ref{constr:edgeinthom} and \ref{constr:edgesOfComp}, and Remark~\ref{rem:rAndBend},
    the set of boundary edges is given by 
    \begin{align*}
        E^\partial =&\ E(\abs{\inthom(S',T')}^r) \sqcup E(\abs{\inthom(S,T)}^r) \sqcup E(\abs{\inthom(S'\hcompw S, T'\hcompw T)}) \\
        =&\ E(\topcirc{T'}) \sqcup E(\topcirc{T'^{*\#}}) \sqcup E(\topcirc{T}) \sqcup E(\topcirc{T^{*\#}})\sqcup E(\topcirc{S'})\sqcup E(\topcirc{S'^{*\#}})\sqcup E(\topcirc{S})  \sqcup E(\topcirc{S^{*\#}}) \\
        & \sqcup \pi\left(E(\partial T'^{*\#})\sqcup E(\partial S')\right) \sqcup \pi\left(E(\partial T^{*\#})\sqcup E(\partial S)\right) \\
        & \sqcup 
        \pi\left(\rho(E(\partial T')\sqcup E(\partial T)) \sqcup \rho(E(\partial S'^{*\#})\sqcup E(\partial S^{*\#}))\right).
    \end{align*}
    \textbf{Facets:} The facets will be parametrized by $E^\partial$.
    For this, consider the equivalence relation generated by the following: two boundary edges $e,e'\in E(\partial S)\sqcup E(\partial S')\sqcup E(\partial T)\sqcup E(\partial T')$ are related if they share a position in $s$, $t$ or $u$. For each equivalence class $P$, consider the subset
    \[
            E^{\partial}_P:=\{\overline{e} \,|\, e\in P\}\sqcup \{\overline{e^{*\#}} \,|\, e\in P\}\subset E^\partial,
    \]
    where $\overline{e}\in E^\partial$ (resp.\ $\overline{e^{*\#}}$) is the image of $e$ (resp.\ $e^{*\#}$) under the appropriate surjection $\pi$ from \eqref{eq:surjwebglue}.
    
    Note that $E^\partial$ is given by the disjoint union of the internal edges of $S,S',T,T'$ and their $*\#$ versions, and the set $\bigsqcup_{P}E^\partial_P$. The facets of $\mathrm{hor}_{T,S}^{T',S'}$ are given by the following:
    \begin{itemize}
        \item For every loop $e\in E(\topcirc{T'})\sqcup E(\topcirc{T}) \sqcup E(\topcirc{S'})\sqcup E(\topcirc{S})$ there is a canonically identified partner $e^{*\#}\in E(\topcirc{T'^{*\#}})\sqcup E(\topcirc{T^{*\#}}) \sqcup E(\topcirc{S'^{*\#}})\sqcup E(\topcirc{S^{*\#}})$. We have a facet $f(e)\in F_{0,2}$ with label  $\labelfunction(e)=\labelfunction(e^{*\#})$ and boundary words $((e),(e^{*\#}))$.
        \item For every arc $e\in E(\topcirc{T'})\sqcup E(\topcirc{T}) \sqcup E(\topcirc{S'})\sqcup E(\topcirc{S})$ there is a canonically identified partner $e^{*\#}\in E(\topcirc{T'^{*\#}})\sqcup E(\topcirc{T^{*\#}}) \sqcup E(\topcirc{S'^{*\#}})\sqcup E(\topcirc{S^{*\#}})$. We have a facet $f(e)\in F_{0,1}$ with label $\labelfunction(e)=\labelfunction(e^{*\#})$. Its boundary word is the cycle composed out of $e$, $e^{*\#}$ and the two seams belonging to the vertices of $e$, with orientation inherited from $e$.
    \item Each equivalence class $P$ corresponds to a facet $f(P)$ according to the following three cases:
    \begin{enumerate}[(i)]
            \item $E^\partial_P$ contains only arcs and no loops. Then, we have a facet $f(P)\in F_{0,1}$ with label $\labelfunction(e)$, for any $e\in P$, and boundary word given by the unique minimal cycle composed out of the arcs of $E^\partial_P$ and seams in $E$, with orientation inherited from $\overline{e}$.
         \item $E^{\partial}_P$ contains only loops and no arcs. Let $E^\partial_P=\{\overline{e_1},\ldots,\overline{e_{|E^\partial_P|}}\}$ be the loops. Then, we have a facet $f(P)\in F_{0,|E^\partial_P|}$ with label $\labelfunction(e_i)$, for any $i$, and boundary words $((\overline{e_1}),\ldots,(\overline{e_{|E^\partial_P|}}))$.
         \item $E^{\partial}_P$ contains both loops and arcs. Then we have a facet $f(P)\in F_{0,\ell+1}$, where $\ell$ is the number of loops in $E^\partial_P$. The facet has label $\labelfunction(e)$, for any $e\in P$, and boundary words given by the union of the boundary words from the previous two cases.
        \end{enumerate}
    \end{itemize}
    The cyclic ordering of all facets around seams is inherited from the boundary web.
\end{construction}

\begin{definition}\label{def:horizontalComposition}
    Let $s,t,u$ be objects in $\Foambicat{\labdata}$. Define the $\kb$-linear \emph{horizontal composition} 
    \begin{equation}\label{eq:horcomp}
        \circ_h\colon \Hom_{\Foambicathom{\labdata}{u}{t}}(S',T') \otimes \Hom_{\Foambicathom{\labdata}{t}{s}}(S,T) \to \Hom_{\Foambicathom{\labdata}{u}{s}}(S'\circ S, T'\circ T)
    \end{equation}
    as the map induced by the internal horizontal composition map $\mathrm{hor}_{T,S}^{T',S'}$ under the functor $\webeval$.
\end{definition}

\begin{lemma}[Interchange law]\label{lem:interchangelaw}
    Let $s,t,u$ be objects and $S,T,U\colon s\to t$ and $S',T',U'\colon t\to u$ be 1-morphisms in $\Foambicat{\labdata}$. Consider 2-morphisms $\Sigma\colon S\to T$, $\Theta \colon T\to U$, $\Sigma'\colon S'\to T'$, and $\Theta'\colon T'\to U'$. Then,
    \begin{align}\label{eq:interchangeLaw}
        (\Theta'\circ_h \Theta) \circ_v (\Sigma'\circ_h \Sigma) = (\Theta'\circ_v \Sigma') \circ_h (\Theta\circ_v\Sigma).
    \end{align}
\end{lemma}
\begin{proof}
    For all objects $s,t,u$ in $\Foambicat{\labdata}$, and 1-morphisms $S,T,U\colon s\to t$ and $S',T',U'\colon t\to u$, one can check that the following equality of combinatorial $\labdata$-foams holds
    \[
        \mathrm{vert}_{U'\circ U,T'\circ T,S'\circ S}\circ \left( \mathrm{hor}_{U,T}^{U',T'} \sqcup \mathrm{hor}_{T,S}^{T',S'} \right) 
        =
        \mathrm{hor}_{U,S}^{U',S'}\circ \left( \mathrm{vert}_{U',T',S'} \sqcup \mathrm{vert}_{U,T,S} \right).
    \]
    Hence, the combinatorial $\labdata$-foams underlying both sides of \eqref{eq:interchangeLaw} are the same. Applying $\webeval$ yields the statement.
\end{proof}

\begin{thm}\label{thm:Foam2Cat}
    The above constructions endow $\Foambicat{\labdata}$ with the structure of a locally small, locally $\kb$-linear (strict) 2-category:
    \begin{itemize}
        \item The objects are given in Definition~\ref{def:obj2cat}.
        \item The $\kb$-linear category $\Foambicathom{\labdata}{t}{s}$ of 1-morphisms $s\to t$ is defined in Proposition~\ref{prop:foamHomcat}.
        \item Given objects $s,t,u$, the $\kb$-linear horizontal composition functor
        \begin{align*}
            \circ_h\colon \Foambicathom{\labdata}{u}{t} \otimes \Foambicathom{\labdata}{t}{s} \to \Foambicathom{\labdata}{u}{s}
        \end{align*}
        is given by the composition of 1-morphisms of $\Foambicat{\labdata}$ from Definition~\ref{def:foamonemorphisms}, on the level of objects. On the level of morphisms (i.e.\ 2-morphisms of $\Foambicat{\labdata}$), it is given by the $\kb$-linear horizontal composition map constructed in Definition~\ref{def:horizontalComposition}.
    \end{itemize}
\end{thm}
\begin{proof}
    The functoriality of horizontal composition functor $\circ_h$ under vertically composing 2-morphisms is the interchange law from Lemma~\ref{lem:interchangelaw}.
    For unitality of the horizontal composition law, let $s=t$ and $u$ be objects and $S',T'\colon t\to u$ be 1-morphisms. The following equality of combinatorial $\labdata$-foams holds
    \begin{align*}
        \mathrm{hor}_{\id_s,\id_s}^{T',S'} \circ (\id_{\abs{\inthom(S',T')}}\sqcup \mathrm{cup}_{\id_s}) = \id_{\abs{\inthom(S',T')}}\colon \abs{\inthom(S',T')}\to \abs{\inthom(S',T')}.
    \end{align*}
    Indeed, composing with $\mathrm{cup}_{\id_s}$ reduces the number of boundary components of the facets of $\mathrm{hor}_{\id_s,\id_s}^{T',S'}$ to match with the identity $\labdata$-foam $\id_{\abs{\inthom(S',T')}}$. As a consequence, for $\Sigma\in \Hom_{\Foambicathom{\labdata}{u}{t}}(S',T')$, we have
    \begin{align*}
        \webeval(\mathrm{hor}_{\id_s,\id_s}^{T',S'})(\Sigma\otimes \id_{\id_s}) = \Sigma \in \Hom_{\Foambicathom{\labdata}{u}{t}}(S',T').
    \end{align*}
    Unitality on the left is similar. For the associativity, one has to check that for all $S,T\colon s\to t$, $S',T'\colon t\to u$ and $S'',T''\colon u\to v$, we have
    \[
        \mathrm{hor}_{T,S}^{T''\circ T',S''\circ S'}\circ\mathrm{hor}_{T',S'}^{T'',S''}=\mathrm{hor}_{T'\circ T,S'\circ S}^{T'',S''}\circ\mathrm{hor}_{T,S}^{T',S'}
    \]
    as combinatorial $\labdata$-foams. Then, associativity follows by applying $\webeval$. 
\end{proof}

\subsection{Tensor product}
\label{sec:tencomp}
We follow Lemma 4 \cite{BaezNeuchlHDAI} to endow the 2-category $\Foambicat{\labdata}$ with a semistrict monoidal structure. 
\begin{construction}[Tensoring 2-functors]\label{constr:tensoringWithObj}
    We extend the operation $\boxtimes$, previously defined on objects, to (strict) 2-functors $x\boxtimes -$ and $-\boxtimes x$ for every object $x$ in $\Foambicat{\labdata}$. We prove the 2-functoriality properties in Theorem \ref{thm:semistrictMonStruct} below. Let $W\colon s\to t$ be a 1-morphism in $\Foambicat{\labdata}$ given by the sequence $W=(W_k,...,W_1)$. We define 1-morphisms
    \begin{align*}
       x\boxtimes W\colon x\boxtimes s\to x\boxtimes t \quad \text{ by } \quad x\boxtimes W &:= (x\boxtimes W_k,\dots, x\boxtimes W_1),\\
        W\boxtimes x\colon s\boxtimes x\to t\boxtimes x \quad \text{ by } \quad W\boxtimes x &:= (W_k\boxtimes x,\dots, W_1\boxtimes x),
    \end{align*}
    where $x\boxtimes W_i = x\boxtimes (o_l\boxtimes W'_i \boxtimes o_r) := ((x\boxtimes o_l)\boxtimes W'_i \boxtimes o_r) $ and $W_i\boxtimes x = (o_l\boxtimes W'_i \boxtimes o_r)\boxtimes x := (o_l\boxtimes W'_i \boxtimes (o_r\boxtimes x))$ for appropriate generating webs $W'_i$.
\smallskip 

    We now define the action of $x\boxtimes -$ and $-\boxtimes x$ on 2-morphisms. Let $s,t$ be objects of $\Foambicat{\labdata}$ and $S,T\colon s\to t$ be 1-morphisms in $\Foambicat{\labdata}$.
    For every object $x\in \Foambicat{\labdata}$, consider the combinatorial $\labdata$-foam 
    \[
        \id_{\abs{\inthom(S,T)}}\sqcup \mathrm{cup}_{\id_x}\colon \abs{\inthom(S,T)}\to \abs{\inthom(S,T)}\sqcup \abs{\inthom(\id_x,\id_x)}
    \]
    and observe that the following equality of combinatorial $\labdata$-webs holds
    \begin{equation*}\label{eq:combMonWebs}
        \abs{\inthom(x\boxtimes S,x\boxtimes T)} =
        \abs{\inthom(S,T)}\sqcup \abs{\inthom(\id_x,\id_x)}  = 
        \abs{\inthom(S\boxtimes x,T\boxtimes x)}.
    \end{equation*}
    The internal homs underlying the rightmost and leftmost webs in the equality above are represented below:
    \[
        \xMon{S}{T} ,
        \qquad
        \Monx{S}{T}.
    \]

    \noindent Let $\Sigma\in\Hom_{\Foambicathom{\labdata}{t}{s}}(S,T)$ be a 2-morphism. We define $x\boxtimes \Sigma$ as the image of $\Sigma$ under the map 
    \begin{align*}
        \Hom_{\Foambicathom{\labdata}{t}{s}}(S,T)\to \Hom_{\Foambicathom{\labdata}{x\boxtimes t}{x\boxtimes s}}(x\boxtimes S,x\boxtimes T)
    \end{align*}
    induced by the foam $\id_{\abs{\inthom(S,T)}}\sqcup \mathrm{cup}_{\id_x}$, seen as a foam 
    \begin{align*}
        \abs{\inthom(S,T)}\to \abs{\inthom(x\boxtimes S,x\boxtimes T)}
    \end{align*}
    under the functor $\webeval$. Similarly, we define $\Sigma\boxtimes x$ as the image of $\Sigma$ under the map 
    \begin{align*}
        \Hom_{\Foambicathom{\labdata}{t}{s}}(S,T)\to \Hom_{\Foambicathom{\labdata}{t\boxtimes x}{s\boxtimes x}}(S\boxtimes x,T\boxtimes x)
    \end{align*}
    induced by the foam $\id_{\abs{\inthom(S,T)}}\sqcup \mathrm{cup}_{\id_x}$
    under the functor $\webeval$. 
\end{construction}

\begin{lemma}\label{lem:tensAndSqcup}
    Let $S,T,U\colon s\to t$ and $S',T'\colon t\to u$ be 1-morphisms in $\Foambicat{\labdata}$, 
    and consider 2-morphisms $\Sigma\in\Hom_{\Foambicathom{\labdata}{t}{s}}(S,T)$, $\Theta\in\Hom_{\Foambicathom{\labdata}{t}{s}}(T,U)$ and $\Sigma'\in\Hom_{\Foambicathom{\labdata}{u}{t}}(S',T')$. 
    For all objects $x\in\Foambicat{\labdata}$, the following equalities of combinatorial $\labdata$-foams hold
    \begin{gather}
        \mathrm{vert}_{x\boxtimes U,x\boxtimes T,x\boxtimes S} = \mathrm{vert}_{U,T,S}\sqcup \mathrm{vert}_{\id_x,\id_x,\id_x} = \mathrm{vert}_{U\boxtimes x,T\boxtimes x,S\boxtimes x}, \\
        \mathrm{hor}_{x\boxtimes T,x\boxtimes S}^{x\boxtimes T',x\boxtimes S'} = \mathrm{hor}_{T,S}^{T',S'}\sqcup \mathrm{hor}_{\id_x,\id_x}^{\id_x,\id_x} = \mathrm{hor}_{T\boxtimes x,S\boxtimes x}^{T'\boxtimes x,S'\boxtimes x}, \\
        \mathrm{cup}_{x\boxtimes S} =\mathrm{cup}_{S}\sqcup \mathrm{cup}_{\id_{x}} = \mathrm{cup}_{S\boxtimes x}. \label{eq:cupTensId}
    \end{gather}
    Moreover, if $T''\colon s'\to t'$ is another 1-morphism in $\Foambicat{\labdata}$, one has
    \begin{equation}\label{eq:tensAndSqcupCup}
        \mathrm{cup}_{(t\boxtimes T'')\circ (S\boxtimes s')}= \mathrm{cup}_{T''}\sqcup \mathrm{cup}_S.
    \end{equation}
\end{lemma}

\begin{proof}
    The statement follows from Constructions \ref{constr:intcomp}, \ref{constr:intHorComp} and \ref{constr:intid}, i.e.\ by directly comparing the data defining the combinatorial $\labdata$-foams that appear in the identities.
\end{proof}

\begin{construction}[Tensorator]\label{constr:tensorator}
    Let $S\colon s\to t$ and $T\colon s'\to t'$ be 1-morphisms in $\Foambicat{\labdata}$.
    Observe that the following combinatorial $\labdata$-webs are the same:
    \begin{align}\label{eq:tensWebs}
    \begin{split}
        &\ \abs{\inthom((t\boxtimes T)\circ (S \boxtimes s'),(t\boxtimes T)\circ (S \boxtimes s'))}\\ 
        =&\ \abs{\inthom((S\boxtimes t')\circ(s\boxtimes T),(t\boxtimes T)\circ (S \boxtimes s'))}.
    \end{split}
    \end{align}
    Indeed, both $\labdata$-webs consist of the same combinatorial data describing $\abs{\inthom(S,S)} \sqcup\abs{\inthom(T,T)}$. In terms of graphical calculus, we have
    \begin{align*}
        (S\boxtimes t')\circ(s\boxtimes T)=\;\;  \SourceWebIntHomTensorator\quad\quad \text{and}\quad\quad \TargetWebIntHomTensorator\;\;=(t\boxtimes T)\circ (S \boxtimes s').
    \end{align*}
    By \eqref{eq:tensWebs}, we can view the combinatorial $\labdata$-foam $\mathrm{cup}_{(t\boxtimes T)\circ (S \boxtimes s')}$ defined in  Construction~\ref{constr:intid}, as a foam 
    \begin{align*}
        \emptyset\to
       \abs{\inthom((S\boxtimes t')\circ(s\boxtimes T),(t\boxtimes T)\circ (S \boxtimes s'))}.
    \end{align*}
    The \emph{tensorator} on $S$ and $T$ is defined as the 2-morphism
    \begin{align*}
        \otimes_{S,T}:= \webeval(\mathrm{cup}_{(t\boxtimes T)\circ (S \boxtimes s')})(1)\in \Hom_{\Foambicathom{\labdata}{t\boxtimes t'}{s\boxtimes s'}}((S\boxtimes t')\circ(s\boxtimes T),(t\boxtimes T)\circ (S \boxtimes s')).
    \end{align*}
\end{construction}

\begin{thm}\label{thm:semistrictMonStruct}
    Constructions \ref{constr:tensoringWithObj} and \ref{constr:tensorator} endow the 2-category $\Foambicat{\labdata}$ with a semistrict monoidal structure with monoidal unit given by the empty sequence $\emptyset$.
\end{thm}
\begin{proof}
    We verify (i)-(viii) of Lemma 4 in \cite{BaezNeuchlHDAI}. Let $x$ be an object in $\Foambicat{\labdata}$. For (i), we verify the 2-functoriality properties of $x\boxtimes -$ and $-\boxtimes x$.
    \begin{itemize}
        \item Composition of 1-morphisms: Let $s,t,u$ be objects and $S\colon s\to t$ and $T\colon t\to u$ be 1-morphisms in $\Foambicat{\labdata}$. We have $x\boxtimes (T\circ S) = (x\boxtimes T)\circ (x\boxtimes S)$ since in Construction~\ref{constr:tensoringWithObj}, $x\boxtimes -$ is defined as the (strictly associative) concatenation of objects in each element of the sequence.
        \item Identity 1-morphisms: For an object $s$ in $\Foambicat{\labdata}$, we have $\id_{x\boxtimes s} = x\boxtimes \id_s$ since both sides are the empty word.
        \item Identity 2-morphisms: For a 1-morphism $S\colon s\to t$ in, we have $x\boxtimes \id_S = \id_{x\boxtimes S}$. Indeed, by Constructions \ref{constr:intid} and \ref{constr:tensoringWithObj} 
        the combinatorial $\labdata$-foams underlying the two sides are $\mathrm{cup}_{\id_x}\sqcup \mathrm{cup}_{S}$ and $\mathrm{cup}_{x\boxtimes S}$ respectively, which are equal by construction.
        \item Functoriality for vertical composition: Given 1-morphisms $S,T,U\colon s\to t$ and 2-morphisms $\Sigma\in\Hom_{\Foambicathom{\labdata}{t}{s}}(S,T)$ and $\Theta\in\Hom_{\Foambicathom{\labdata}{t}{s}}(T,U)$, we have 
        \[
            x\boxtimes (\Theta \circ_v \Sigma) = (x\boxtimes \Theta)\circ_v (x\boxtimes \Sigma).
        \]
        The combinatorial $\labdata$-foams underlying the two sides of this equality are respectively 
        \[
            (\mathrm{cup}_{\id_x}\sqcup \id_{\abs{\inthom(S,U)}})\circ\mathrm{vert}_{U,T,S}=\mathrm{cup}_{\id_x}\sqcup \mathrm{vert}_{U,T,S}
        \]
        and 
        \[
            \mathrm{vert}_{x\boxtimes U,x\boxtimes T,x\boxtimes S} \circ ((\mathrm{cup}_{\id_x}\sqcup \id_{\abs{\inthom(T,U)}}) \sqcup (\mathrm{cup}_{\id_x}\sqcup \id_{\abs{\inthom(S,T)}})).
        \]
        In order to verify the equality, we have to show that the latter combinatorial $\labdata$-foam is equal to the former. This is a consequence of \Cref{lem:tensAndSqcup} and of the unitality of $\circ_v$, shown in \Cref{prop:foamHomcat}.
        \item Functoriality for horizontal composition: 
        Given parallel 1-morphisms $S,T\colon s\to t$ and $S',T'\colon t\to u$, and 2-morphisms $\Sigma\in\Hom_{\Foambicathom{\labdata}{t}{s}}(S,T)$ and $\Sigma'\in\Hom_{\Foambicathom{\labdata}{u}{t}}(S',T')$, we have 
        \[
            x\boxtimes (\Sigma' \circ_h \Sigma) = (x\boxtimes \Sigma')\circ_h (x\boxtimes \Sigma).
        \]
        The combinatorial $\labdata$-foams underlying the two sides of this equality are respectively 
        \[
            (\mathrm{cup}_{\id_x}\sqcup \id_{\abs{\inthom(S'\circ S,T'\circ T)}})\circ\mathrm{hor}_{T,S}^{T',S'}=\mathrm{cup}_{\id_x}\sqcup \mathrm{hor}_{T,S}^{T',S'}
        \]
        and 
        \[
            \mathrm{hor}_{x\boxtimes T,x\boxtimes S}^{x\boxtimes T',x\boxtimes S'} \circ ((\mathrm{cup}_{\id_x}\sqcup \id_{\abs{\inthom(S',T')}}) \sqcup (\mathrm{cup}_{\id_x}\sqcup \id_{\abs{\inthom(S,T)}})).
        \]
        Similarly to the proof of functoriality for the vertical composition, the equality of these two combinatorial $\labdata$-foams is a consequence of  \Cref{lem:tensAndSqcup} and of the unitality of $\circ_h$ from \Cref{thm:Foam2Cat}.
    \end{itemize}
    The 2-functoriality of $-\boxtimes x$ follows completely analogously by symmetry of the construction.
\smallskip

    For (ii), let $\emptyset$ be the object given by the empty sequence. For every object $x$, we have $\emptyset \boxtimes x = x\boxtimes \emptyset = x$ since $\boxtimes$ acts as concatenation on objects. For a 1-morphism $S\colon s\to t$, we have $\emptyset \boxtimes S = S\boxtimes \emptyset = S$ since the operation is element-wise concatenation with the empty sequence. Now, let $T\colon s\to t$ be another 1-morphism and $\Sigma\in\Hom_{\Foambicathom{\labdata}{t}{s}}(S,T)$ be a 2-morphism. The combinatorial $\labdata$-foam ${}_\emptyset\boxtimes$ induces the same map under $\webeval$ as the internal identity and hence, $\emptyset \boxtimes \Sigma = \Sigma$. Similarly, we also obtain $\Sigma\boxtimes \emptyset =\Sigma$.
\smallskip

    For (iii), associativity is clear between objects and 1-morphisms, since the $\boxtimes$ operation is (element-wise) concatenation. Let $S,T\colon s\to t$ be 1-morphisms and $\Sigma\in\Hom_{\Foambicathom{\labdata}{t}{s}}(S,T)$ be a 2-morphism. Given two objects $x,y$, we need to show that
    \begin{align*}
        x\boxtimes (y\boxtimes \Sigma) =  (x\boxtimes y)\boxtimes \Sigma, && x\boxtimes (\Sigma\boxtimes y) =  (x\boxtimes \Sigma)\boxtimes y, && \Sigma\boxtimes (x\boxtimes y) =  (\Sigma\boxtimes x)\boxtimes y.
    \end{align*}
    The combinatorial $\labdata$-foams underlying all the terms of the equalities above can be shown to be equal to $\id_{\abs{\inthom(S,T)}}\sqcup \mathrm{cup}_{\id_x}\sqcup \mathrm{cup}_{\id_y}$. This follows directly from Construction \ref{constr:tensoringWithObj}, \Cref{prop:foamcat} and  \Cref{lem:tensAndSqcup}.
    \smallskip

    For (iv) let $S\colon s\to s'$, $T\colon t\to t'$, $U\colon u\to u'$ be 1-morphisms. We need to prove the following equalities: 
    \begin{align*}
        \otimes_{s\boxtimes T,U}=s\boxtimes \otimes_{T,U}, &&
        \otimes_{S\boxtimes t,U}=\otimes_{S,t\boxtimes U}, &&
        \otimes_{S,T\boxtimes u}=\otimes_{S,T}\boxtimes u.
    \end{align*}
    All equalities follow from \Cref{lem:tensAndSqcup}.
\smallskip

    For (v), let $x,y$ be objects. We need to show that $\id_x \boxtimes y = x\boxtimes \id_y = \id_{x\boxtimes y}$, but this is clear since each of the expression is the empty sequence of objects. Now, let $S\colon s\to t$ and $T\colon s'\to t'$ be 1-morphisms. We need to show that $\otimes_{\id_x, T}=\id_{x\boxtimes T}$ and $\otimes_{S,\id_y}=\id_{S\boxtimes y}$. The underlying combinatorial $\labdata$-foams from Constructions~\ref{constr:intid} and \ref{constr:tensorator} are the same. Hence, applying $\webeval$, we obtain the equality.
\smallskip

    For (vi) let $S\colon s\to t$ and $T,T'\colon s'\to t'$ be 1-morphisms, and $\Sigma\in\Hom_{\Foambicathom{\labdata}{t'}{s'}}(T,T')$ be a 2-morphism. We need to show that 
    \begin{equation} \label{eq:tensProofvi}
        \otimes_{S,T'}\circ_v (\id_{S\boxtimes t'}\circ_h (s\boxtimes \Sigma)) = ((t\boxtimes \Sigma)\circ_h \id_{S\boxtimes s'})\circ_v \otimes_{S,T}.
    \end{equation}
    In order to prove this equality, we translate it in terms of combinatorial $\labdata$-foams. Since the combinatorial $\labdata$-foam associated to $\otimes_{S,T'}$ is $\mathrm{cup}_{(t\boxtimes T)\circ (S\boxtimes s')}$, which also corresponds to the identity of $\circ_v$ by \Cref{prop:foamHomcat}, the $\labdata$-foam associated to $\otimes_{S,T'}\circ_v -$ is equal to the identity. Analogously, the combinatorial $\labdata$-foam corresponding to $- \circ_v \otimes_{S,T}$ is equal to the identity.
    Hence, we only need to show that the combinatorial $\labdata$-foams underlying the remaining parts of \eqref{eq:tensProofvi} are equal, that is, by \Cref{lem:tensAndSqcup}:
    \[
        \mathrm{hor}_{s\boxtimes T',s\boxtimes T}^{S\boxtimes t',S\boxtimes t'} \circ (\mathrm{cup}_S\sqcup \mathrm{cup}_{\id_{t'}}\sqcup \mathrm{cup}_{\id_s})=
        \mathrm{hor}_{S\boxtimes s',S\boxtimes s'}^{t\boxtimes T',t\boxtimes T} \circ (\mathrm{cup}_{\id_t}\sqcup \mathrm{cup}_{S}\sqcup \mathrm{cup}_{\id_{s'}}).
    \]
    In fact, one can verify that the two sides of the identity above are equal, respectively, to 
    \[
        (\mathrm{hor}_{\id_s,\id_s}^{S,S} \circ (\mathrm{cup}_S\sqcup \mathrm{cup}_{\id_s})) \sqcup (\mathrm{hor}_{T',T}^{\id_{t'},\id_{t'}} \circ (\mathrm{cup}_{\id_{t'}}\sqcup \id_{\abs{\inthom(T,T')}}))
    \]
    and 
     \[
        (\mathrm{hor}^{\id_t,\id_t}_{S,S} \circ (\mathrm{cup}_{\id_t}\sqcup \mathrm{cup}_S)) \sqcup (\mathrm{hor}^{T',T}_{\id_{s'},\id_{s'}} \circ (\id_{\abs{\inthom(T,T')}}\sqcup \mathrm{cup}_{\id_{s'}})).
    \]
    By unitality of $\circ_h$, both these $\labdata$-foams are equal to $\id_{\abs{\inthom(T,T')}} \sqcup \mathrm{cup}_S$, which proves identity \eqref{eq:tensProofvi}.
\smallskip

    For (vii) let $S,S'\colon s\to t$ and $T\colon s'\to t'$ be 1-morphisms, and $\Sigma\in\Hom_{\Foambicathom{\labdata}{t}{s}}(S,S')$ be a 2-morphism. We need to show that 
    \begin{equation*}
        \otimes_{S',T}\circ_v ((\Sigma\boxtimes t') \circ_h \id_{s\boxtimes T}) = (\id_{t\boxtimes T}\circ_h (\Sigma\boxtimes s'))\circ_v \otimes_{S,T}.
    \end{equation*}
    The proof is analogous to the previous point.
\smallskip

    For (viii) let $S\colon s\to t$, $S'\colon t\to u$, $T\colon s'\to t'$ and $T'\colon t'\to u'$ be 1-morphisms. We need to show that the following identities hold:
    \begin{gather*}
        \otimes_{S,T'\circ T}= (\id_{t\boxtimes T'}\circ_h \otimes_{S,T})\circ_v (\otimes_{S,T'}\circ_h \id_{s\boxtimes T}), \\
        \otimes_{S'\circ S,T}= (\otimes_{S',T}\circ_h \id_{S\boxtimes s'})\circ_v (\id_{S'\boxtimes t'}\circ_h \otimes_{S,T}).
    \end{gather*}
    Similarly to the previous two points, the first identity can be verified as follows. First observe, using \Cref{lem:tensAndSqcup}, that the combinatorial $\labdata$-foam underlying each side of the equality splits as a disjoint union of a $\labdata$-foam involving only $T$ and $T'$ and a $\labdata$-foam involving only $S$. Then, the equality follows from unitality of $\circ_v$ and $\circ_h$ and from Constructions \ref{constr:inthom} and \ref{constr:intHorComp}. The second identity can be proved analogously.
\end{proof}

\subsection{Local gradings and examples}\label{subsec:locallygraded}
We show how to use \Cref{sec:Frob} to upgrade the construction of the locally linear semistrict monoidal 2-category $\Foambicat{\labdata}$ to a graded setting.

In principle, we can aim for two outcomes:
\begin{enumerate}
    \item A semistrict monoidal 2-category locally enriched in graded $\kb$-modules as in the \MainThm. We may call this a \emph{locally graded $\kb$-linear} semistrict monoidal 2-category.
    \item A locally $\kb$-linear semistrict monoidal 2-category with local $\Z$-action by grading shift autoequivalences on hom categories.
\end{enumerate}
In fact, it is most straightforward to first construct an example of type (1) and then pass to type (2) by the standard procedure of formally adjoining grading shifts of 1-morphisms and subsequently restricting 2-morphism modules to degree zero. \smallskip

    Suppose that $\kb$ is a graded commutative ring, $\labdata$ is a graded $\kb$-linear labeling datum, and we consider a graded foam evaluation, so that the graded version of property (A) is satisfied. Then $\FoamCat[\labdata]$ is equipped with a grading as in \Cref{sec:Frob} and $\webeval$ is a graded $\kb$-linear functor. Repeating Definition~\ref{def:2homsAndComposition} verbatim in this setting, the $\kb$-module $\Hom_{\Foambicathom{\labdata}{t}{s}}(S,T)$ of 2-morphisms between parallel 1-morphisms $S,T\colon s\to t$ naturally inherits a grading. However, the vertical composition $\vcomp$ of \eqref{eqn:vertcomp} and the horizontal composition map $\hcomp$ from \eqref{eq:horcomp} would in this case only be homogeneous $\kb$-linear maps of a certain degree determined by the abstract $\labdata$-foam realizing the respective operation. The key to obtaining a semistrict monoidal 2-category locally enriched in graded $\kb$-modules is thus to introduce systematic grading shifts of 2-morphism spaces, so that vertical and horizontal composition become grading-preserving on the level of 2-morphisms.

\begin{example}[Nilvalent graded case]\label{ex:BNgraded}
    Consider graded $\kb$-linear labeling data $\labdata$ for nilvalent foams. Let $s=(s_1,\ldots,s_k)$ be an object in $\Foambicat{\labdata}$ and, for all $i$, let $x_i$ be the element of $\labdata_1$ underlying $s_i\in\labdata_1\times \{\uparrow, \downarrow\}$. We define the \emph{shift} of $s$ to be 
    \begin{align*}
        \shift(s) :=\sum_i \shift(x_i)
    \end{align*}
    where $\shift(x_i)$ is the shift of the graded Frobenius algebra $\mathcal{Z}_{x_i}(\bigcirc)$ from Definition~\ref{def:gradingFrob}.
    
    Then, for 1-morphisms $S,T\colon s\to t$ in $\Foambicat{\labdata}$, we define
    \begin{align*}
        \Hom_{\Foambicathom{\labdata}{t}{s}}(S,T) := \webeval(\abs{\inthom(S,T)})\{\tfrac{1}{2}(\shift(s)+\shift(t))\}
    \end{align*}
    where on the right-hand side, we denote the upward shift of the graded $\kb$-module in curly brackets. One can check that with this shift the compositions $\vcomp$ and $\hcomp$ of $\Foambicat{\labdata}$ as well as the semistrict monoidal structure are grading-preserving.
    \end{example}
\begin{example}[Oriented decorated Khovanov--Bar-Natan surfaces]
    Specializing to the nilvalent case of Example~\ref{ex:BNgraded} with a single label and a graded Frobenius algebra $A$, we obtain a locally graded $\kb$-linear monoidal 2-category of \emph{oriented} decorated Khovanov--Bar-Natan surfaces. This yields an oriented version of part (1) of the \MainThm:
   \begin{align*}
    \overrightarrow{\bn_A}\in \Alg_{\Eone}(\Cat[\Cat[\kb\grmod]]).
\end{align*}
\end{example}

\begin{example}[Unoriented decorated Khovanov--Bar-Natan surfaces]
\label{exa:ubn}
    The statement (1) in the \MainThm\ relies on a completely parallel treatment of nilvalent combinatorial $\labdata$-foams and their evaluation as in Sections \ref{sec:closedFoamEval} and \ref{sec:TQFT} which do not carry orientation data. Indeed, this yields a locally graded $\kb$-linear monoidal 2-category of \emph{unoriented}\footnote{For unoriented decorated Khovanov--Bar-Natan surfaces, the natural input for the decorations is given by 2-dimensional TQFTs for \emph{orientable} cobordisms in the sense of \cite{GoertzWedrich}. These correspond to tuples $(A,\rmphi)$ of a commutative Frobenius algebra $A$ together with an involution $\rmphi$, which may be chosen as $\rmphi=\id$.} decorated Khovanov--Bar-Natan surfaces
\begin{align*}
    \bn_A\in\Alg_{\Eone}(\Cat[\Cat[\kb\grmod]]).
\end{align*}
 There exists a functor of monoidal 2-categories
\begin{align*}
    \overrightarrow{\bn_A}\to \bn_A
\end{align*}
forgetting the orientation. In particular, for objects this forgets the $\{\uparrow, \downarrow\}$ in the entries of the sequences. As a result, the objects of $\bn_A$ are just the $\boxtimes$-powers of the generating object corresponding to the unique label.
\end{example}

\begin{example}[Robert--Wagner $\glN$-foams]\label{ex:glNgraded}
    Given $N\in\N$, recall the case of the Robert--Wagner $\glN$-foams with graded labeling data $\labdata^N$ over $\kb$ from \Cref{ex:FAforRobertWagner}. We abbreviate the notation for the associated locally $\kb$-linear semistrict monidal 2-category by
    \[\Foambicat{N}:=\Foambicat{\labdata^N}\]
    For a 1-morphism $W\colon s\to t$, we note that the degree of the foam $\text{cup}_W$ from Construction~\ref{constr:intid}, which underlies the identity 2-morphism $\id_W$, is computed (using \eqref{eq:foamdeg}) as the weighted Euler characteristic
    \begin{align*}
        \deg(\text{cup}_W) = -\sum_{k\in \labdata_1}\sum_{\substack{\text{non-loop edges}\\\text{in $W$ labeled $k$}}} k(N-k) + \sum_{k,l\in \labdata_1}\sum_{\substack{\text{vertices in $W$}\\\text{labeled } k,l,k+l }}\left((k+l)(N-(k+l))+kl\right).
    \end{align*}
    Then, the graded $\kb$-module of 2-morphisms $\Hom_{\Foambicathom{N}{t}{s}}(S,T)$ is given by 
    \begin{align*}
        \Hom_{\Foambicathom{N}{t}{s}}(S,T) := \webeval(\abs{\inthom(S,T)})\{-\tfrac{1}{2}(\deg(\text{cup}_S)+\deg(\text{cup}_T))\}.
    \end{align*}
    In \cite{QR}, Queffelec--Rose describe a locally graded linear 2-category of progressive $\glN$-foams. In particular, the grading shifts of the 2-morphisms are given in \cite[Definition 3.3]{QR} in terms of the \emph{weighted Euler characteristic} of their source and target webs, which agree with the negative degree of the respective cup foams. The shift $\{-\tfrac{1}{2}(\deg(\text{cup}_S)+\deg(\text{cup}_T)\}$ corresponds to the shift of \cite{QR}, after a slight extension from the progressive $\glN$-foams, which are accessible via categorical skew Howe duality, to Robert--Wagner $\glN$-foams.
    In \cite[Lemma 3.4]{QR}, it is shown that this choice of grading is compatible with horizontal and vertical composition. Since $\deg(\text{cup}_W)$ is additive under disjoint union of webs $W$, it is clear that the grading shifts are also compatible with the monoidal structure. 
    \smallskip
    
    In conclusion, we have thus equipped the semistrict monoidal 2-category $\Foambicat{N}$ of $\glN$-foams with a local enrichment in graded $\kb$-modules, as stated in part (2) of the \MainThm. After trading the local graded $\kb$-linearity for local $\kb$-linearity with local $\Z$-action, we obtain the alternative version remarked upon in the introduction.
\end{example}

\begin{remark}
    The 2-endomorphism algebras of the generating objects of $\Foambicat{N}$ are Grassmannian cohomology rings $H^*(\mathrm{Gr}_{\mathbb{C}}(k,N);\Z)$ and thus non-semisimple for $0< k<N$. This obstructs $\Foambicat{N}$ from yielding a semisimple 2-category in the sense of \cite[Definition~1.4.1]{douglas2018fusion} even after change of ground ring and the usual 2-categorical completion operations. Analogous arguments may be applied to $\bn_A$ (and $\overrightarrow{\bn_A}$), for which the generating object has 2-endomorphism algebra $A$.
\end{remark}

\subsection{Duals and adjoints}
\label{subsec:duals}
\newcommand{\triangulator}{\triangledown}
Recall that semistrict monoidal 2-categories can be considered as Gray categories with one object. In this section, we define a duality structure in the sense of \cite{barrett2018gray} on $\Foambicat{\labdata}$ given by duals for objects and adjoints for 1-morphisms. Specifically, we prove the following theorem.
\begin{thm}\label{thm:dualityStruc}
The operations $*$ and $\#$ defined in \Cref{subsec:objsAndOneMorphs} extend to a spatial duality structure on $\Foambicat{\labdata}$ in the sense of \cite[Definitions 3.10 and 4.8]{barrett2018gray} in the following way.
\begin{enumerate}[(1)]
		\item \textbf{$*$-duality:}
		\begin{enumerate}[(a)]
			\item A (strict) 2-functor $*\colon \Foambicat{\labdata}\to \Foambicat{\labdata}^{\mathrm{hop},\mathrm{vop}}$ which is the identity on objects and reverses the order of vertical and horizontal composition. 
			\item A 2-morphism $\varepsilon_S\in \Hom_{\Foambicathom{\labdata}{t}{t}}(\id_{t}, S\circ S^*)$ for every 1-morphism $S\colon s \to t$ and objects $s,t$ in $\Foambicat{\labdata}$.
		\end{enumerate}
		satisfying the following properties
		\begin{enumerate}[(i)]
			\item $*\circ*$ is the identity 2-functor.
			\item For all 1-morphisms $U\colon r\to s$ and $S,T\colon s\to t$ and 2-morphisms $\Sigma \in \Hom_{\Foambicathom{\labdata}{t}{s}}(S,T)$, we have
			\begin{align}
				(\Sigma\hcomp \id_{S^*})\vcomp \varepsilon_S &= (\id_T\hcomp \Sigma^*) \vcomp \varepsilon_T,\\
				(\id_S\circ_h (\varepsilon_{S^*})^*)\circ_v(\varepsilon_S\circ_h \id_S)&=\id_S, \label{eq:zigzagForAdjoints}\\
				(\id_S \hcomp \varepsilon_U\hcomp \id_{S^*})\vcomp \varepsilon_S &= \varepsilon_{S\circ U}.
			\end{align}
            \item For all objects $x,y,s,t$ and 1-morphisms $S\colon s\to t$ we have
            \begin{itemize}
                \item $(x \boxtimes S \boxtimes y)^* = x \boxtimes  S^* \boxtimes y$
                \item $\varepsilon_{x \boxtimes S \boxtimes y}=x \boxtimes \varepsilon_S \boxtimes y$.
            \end{itemize}
		\end{enumerate}
		\item \textbf{$\#$-duality:} For every object $s$, there is a 1-morphism $\eta_s\colon \emptyset\to s\boxtimes s^\#$ called \emph{fold}, and an invertible 2-morphism $\triangulator_s\in \Hom_{\Foambicat{\labdata}}\bigl((\eta_s^*\boxtimes s)\circ (s\boxtimes \eta_{s^\#}),\id_s\bigr)$ called \emph{triangulator} such that:
		\begin{enumerate}[(i)]
			\item $s^{\#\#}=s$, for all objects $s$.
			\item $\emptyset^\#=\emptyset$, $\eta_\emptyset=\id_\emptyset$ and $\triangulator_\emptyset=\id_{\id_\emptyset}$.
			\item For all objects $s,t$:
			 \begin{itemize}
				\item $(s\boxtimes t)^\# = t^\# \boxtimes s^\# $ 
				\item $\eta_{s\boxtimes t} = (s\boxtimes \eta_t \boxtimes s^{\#}) \circ \eta_s$
				\item $\triangulator_{s\boxtimes t} = ((\triangulator_s\boxtimes t) \hcomp (s \boxtimes \triangulator_t)) \vcomp (\id_{\eta_{s}^*\boxtimes s\boxtimes t}\hcomp \otimes_{s\boxtimes \eta_t^*,\eta_{s^\#}\boxtimes t} \hcomp \id_{s\boxtimes t \boxtimes \eta_{t^\#}})$
			\end{itemize}
			\item For all objects $s$:
            \begin{align*}\id_{\eta_s^*} = (\id_{\eta_s^*}\hcomp (\triangulator_s \boxtimes s^\#))\vcomp (\otimes_{\eta_s^*,\eta_s^*} \hcomp \id_{s\boxtimes \eta_{s^\#}\boxtimes s^\#} ) \vcomp (\id_{\eta_{s}^*}\hcomp (s\boxtimes \triangulator_{s^\#}^* )).
            \end{align*}
		\end{enumerate}
        \item \textbf{Spatiality:} 
        The operation $\#$ extends to a weak 2-functor $\#\colon \Foambicat{\labdata} \to \Foambicat{\labdata}^{\text{hop,$\boxtimes$op}}$. It is weak in the sense that it respects horizontal composition and tensor product, reversing their orders, but only up to specified 2-isomorphisms. 
        
        The natural isomorphisms\footnote{These are natural isomorphisms in the sense of \cite[Definition A.10]{barrett2018gray}.} from \cite[Lemma 4.6 and Definition 4.8]{barrett2018gray}
        \begin{align*}
            \Delta\colon \#\to {*\# *} \qquad \text{and} \qquad {*\Delta *}\colon \#\to {*\# *}
        \end{align*}
        are equal.
	\end{enumerate}
\end{thm}

 \begin{remark}

The $*$-duality provides left- and right-adjoints for 1-morphisms. The equations
\begin{align*}
     (\id_S\hcomp (\varepsilon_{S^*})^*)\vcomp(\varepsilon_S\hcomp \id_S)&=\id_S \\
     ((\varepsilon_{S^*})^*\hcomp \id_{S^*})\vcomp(\id_{S^*}\hcomp \varepsilon_S) &= \id_{S^*}
\end{align*}
are the zig-zag identities for an adjunction between $S$ and its left-adjoint $S^*$, where the 2-morphisms $\varepsilon_S$ and $(\varepsilon_{S^*})^*$ act as unit and counit. 
The first equation is \eqref{eq:zigzagForAdjoints} and the second equation can be proven analogously. We refer to the proof of Theorem~\ref{thm:dualityStruc} below. Alternatively, note that the second equation follows from the duality structure by applying (13) in \cite[Lemma 3.3]{barrett2018gray} to $\id_S$.

Similarly, $S^*$ is right-adjoint to $S$ with unit and counit $\varepsilon_{S^*}$ and $(\varepsilon_S)^*$, and zig-zag identities
\begin{align*}
    (\id_{S^*}\hcomp (\varepsilon_{S})^*)\vcomp(\varepsilon_{S^*}\hcomp \id_{S^*})&=\id_{S^*}\\
    ((\varepsilon_{S})^*\hcomp \id_{S})\vcomp(\id_{S}\hcomp \varepsilon_{S^*}) &= \id_{S}.
\end{align*}
These equations can be obtained from the zig-zag identities for the left-adjoint above by interchanging $S$ with $S^*$, or by applying \cite[Lemma 3.3]{barrett2018gray} again.

The $\#$-duality provides duals for objects witnessed by the folds for which the triangulator acts as a zig-zag isomorphism. The operation $\#$ on 1-morphisms reverses the order of (horizontal) composition and of the monoidal structure.

\end{remark}
Recall that an object in a monoidal 2-category is called \emph{2-dualizable}, if there exist left dual and right dual objects, i.e.\ so that there exist pairs of unit and counit 1-morphisms such that the zig-zag relations hold up to 2-isomorphism; and furthermore the unit and counit 1-morphisms have left and right adjoints. As a direct consequence of Theorem~\ref{thm:dualityStruc}, we obtain the following statement. 
\begin{corollary}\label{cor:2dualizable}
    Every object of $\Foambicat{\labdata}$ is 2-dualizable.
\end{corollary}
In particular, every object of $\Foambicat{N}$, $\overrightarrow{\bn_A}$, and $\bn_A$ is 2-dualizable. Furthermore, in $\bn_A$, every object is self-dual. 
\begin{proof}
    The $\#$-duality computes left- and right-duals for objects, the corresponding unit and counit morphisms are given by appropriate cup and cap 1-morphisms, whose zig-zags are witnessed by versions of the triangulator. The left and right adjoints for these cup and cap 1-morphisms are again appropriate cap and cup 1-morphisms, respectively (with grading shifts in the version with local $\Z$-actions). The unit and counit 2-morphisms for these adjunctions can be expressed in terms of $\varepsilon_S$ and $*$ and visualized as 2-dimensional cup, saddle, and cap 2-morphisms. Their zig-zag relations are discussed in the remark above.
\end{proof}
Below, we prove \Cref{thm:dualityStruc}. For this, we first give the following constructions.

\begin{construction}[$*$-duality] 
\label{constr:adjointsFor1Morph}
In the following, we define the data for adjoints of 1-morphisms in $\Foambicat{\labdata}$ via the $*$-duality. \begin{enumerate}[(a)]
			\item For a 1-morphism $W\colon s\to t$ of $\Foambicat{\labdata}$, we have defined $W^*\colon t\to s$ in Construction~ \ref{constr:starOn1morphs}. Now, let $s,t$ be objects of $\Foambicat{\labdata}$ and consider 1-morphisms $S,T\colon s\to t$. By Lemma~\ref{lem:SameAbsIntHom}, we have
            \[
                \abs{\inthom(S,T)}=\abs{\inthom(T^*,S^*)},
            \]
            therefore we can view the combinatorial $\labdata$-foam $\id_{\abs{\inthom(S,T)}}$ from \Cref{def:idlabfoam} as a foam $\abs{\inthom(S,T)}\to \abs{\inthom(T^*,S^*)}$.
            We define the map
            \begin{align}\label{eq:defStarSualityOn2Morph}
                -^*\colon \Hom_{\Foambicathom{\labdata}{t}{s}}(S,T) \to \Hom_{\Foambicathom{\labdata}{s}{t}}(T^*,S^*)
            \end{align}
            as the map induced by $\id_{\abs{\inthom(S,T)}}$ under the functor $\webeval$. Let $\Sigma\in \Hom_{\Foambicathom{\labdata}{t}{s}}(S,T)$ be a 2-morphism, and denote the image of $\Sigma$ under $-^*$ by  $\Sigma^*\in \Hom_{\Foambicathom{\labdata}{s}{t}}(T^*,S^*)$.
			\item Given objects $s,t$ of $\Foambicat{\labdata}$ and a 1-morphism $S\colon s\to t$, again by Lemma \ref{lem:SameAbsIntHom}, we have
            \[
                \abs{\inthom(\id_t,S\circ S^*)} = \abs{\inthom(S,S)}.
            \]
            Then, the combinatorial $\labdata$-foam $\mathrm{cup}_S$ of \Cref{constr:intid} can be viewed as a foam $\emptyset \to \abs{\inthom(\id_t,S\circ S^*)}$.
            In terms of graphical calculus, we have for the internal hom in the target:
            \begin{align*}
                \TargetWebEpsilon{S}
            \end{align*}
            We define the 2-morphism
            \begin{equation} \label{eq:defEpsilon}
                \varepsilon_S := \webeval(\mathrm{cup}_S)(1)\in \Hom_{\Foambicathom{\labdata}{t}{t}}(\id_t,S\circ S^*).
            \end{equation}
		\end{enumerate}
        Note that by \cite[Lemma 3.3]{barrett2018gray}, the functor $*$ is uniquely determined on 2-morphisms by choice of the $\varepsilon_S$ defined in \eqref{eq:defEpsilon}. 
\end{construction}

\begin{construction}[\#-duality]\label{constr:dualsForObj}
In the following, we define the $\#$-duality which provides duals for objects. For any object $s$ of $\Foambicat{\labdata}$, the element $s^\#$ is defined in Construction \ref{constr:dualobj} and the \emph{fold} is defined as $\eta_s:=\lcupweb_s$, where $\lcupweb_s$ is the 1-morphism from Construction~\ref{constr:fold}. Then, we have $\eta_s^*=\rcapweb_s$ and $\eta_{s^\#}=\rcupweb_s$. Now observe that similarly to the argument in Lemma \ref{lem:SameAbsIntHom}, we also have
    \begin{equation} \label{eq:triangulatorIntHom}
        \abs{\inthom\bigl((\eta_s^*\boxtimes s)\circ (s\boxtimes \eta_{s^\#}),\id_s\bigr)} = \abs{\inthom(\id_s,\id_s)}.
    \end{equation}
We depict the internal hom on the left-hand side in terms of graphical calculus. For this, note that the source and target in the internal homs are depicted as
\begin{align*}
    (\eta_s^*\boxtimes s)\circ (s\boxtimes \eta_{s^\#})=\;\; \curledIdentity{s} \quad\quad \text{and} \quad\quad \idonobj{s}\;\;=\id_s .
\end{align*}
For the internal hom, we have:
\begin{align*}
    \TargetWebTriangulator{s}
\end{align*}
With \eqref{eq:triangulatorIntHom}, the combinatorial $\labdata$-foam $\mathrm{cup}_{\id_s}$ of \Cref{constr:intid} can be viewed as a foam $\emptyset \to \abs{\inthom\bigl((\eta_s^*\boxtimes s)\circ (s\boxtimes \eta_{s^\#}),\id_s\bigr)}$. We define the \emph{triangulator} $\triangulator_s$ as the 2-morphism
\begin{equation}
    \triangulator_s := \webeval(\mathrm{cup}_{\id_s})(1)\in \Hom_{\Foambicathom{\labdata}{s}{s}}\bigl((\eta_s^*\boxtimes s)\circ (s\boxtimes \eta_{s^\#}),\id_s\bigr).
\end{equation}
\end{construction}

\begin{remark}
    Constructions \ref{constr:adjointsFor1Morph} and \ref{constr:dualsForObj} can be extended to both types of graded settings from \Cref{subsec:locallygraded}. From the statements in Subsection~\ref{subsec:locallygraded}, we then deduce that the operations $*$ and $\#$ extend to a duality structure in the locally graded $\kb$-linear case. The $*$-duality defined on 2-morphisms in \eqref{eq:defStarSualityOn2Morph} becomes a map of graded $\kb$-modules and $\varepsilon_S$ for a 1-morphism $S\colon s\to t$ becomes a homogeneous 2-morphism, typically of nonzero degree. For the two Examples~\ref{ex:BNgraded} and \ref{ex:glNgraded}, the degree of the element $\varepsilon_S$ can be computed from the weights stated there.
    By Corollary~\ref{cor:2dualizable} and specializing to the relevant examples, Theorem \ref{thm:dualityStruc} thus proves the remaining claims of the \MainThm.
    
    To accommodate the setting with local $\Z$-actions, one would like to redefine the $*$-operation on 1-morphisms to include compensating grading shifts, so that all $\varepsilon_S$ become grading-preserving. However, the 2-functor $*$ will then typically no longer square to the identity, but rather to a grading shift of the identity.\end{remark}

\begin{proof}[Proof of Theorem \ref{thm:dualityStruc}]
\textbf{$*$-duality:} We start with showing that $*\colon \Foambicat{\labdata}\to \Foambicat{\labdata}^{\mathrm{hop},\mathrm{vop}}$ is indeed a 2-functor.
\begin{itemize}
    \item Composition of 1-morphisms: Let $s,t,u$ be objects and $S\colon s\to t$ and $T\colon t\to u$ be 1-morphisms in $\Foambicat{\labdata}$. We have $(T\circ S)^*=S^*\circ T^*$ by Construction~\ref{constr:starOn1morphs} and the definition of (horizontal) composition in Definition~\ref{def:foamonemorphisms}.
    \item Identity 1-morphisms: Let $s$ be an object, and consider the identity 1-morphism $\id_s\in \Foambicathom{\labdata}{s}{s}$ given by the empty sequence. Clearly, we have $(\id_s)^*=\id_{s^*}$.
    \item Identity 2-morphisms: Let $S\colon s\to t$ be a 1-morphism and consider $S^*\colon t\to s$. By Lemma \ref{lem:SameAbsIntHom}, the combinatorial closed $\labdata$-webs of $\inthom(S,S)$ and $\inthom(S^*,S^*)$ are equal. By Construction~\ref{constr:adjointsFor1Morph}, the combinatorial $\labdata$-foam underlying $-^*$ on 2-morphisms is the identity foam on $\abs{\inthom(S,S)}$. Hence, we obtain by Construction~\ref{constr:intid} that $(\id_S)^*=\id_{S^*}\in \Hom_{\Foambicathom{\labdata}{t}{s}}(S^*,S^*)$.
    \item Vertical composition of 2-morphisms: Given 2-morphisms $\Sigma\in\Hom_{\Foambicathom{\labdata}{t}{s}}(S,T)$ and $\Theta\in\Hom_{\Foambicathom{\labdata}{t}{s}}(T,U)$, we have to show that $(\Theta\circ_v \Sigma)^*=\Sigma^*\circ_v\Theta^*$. Since the combinatorial $\labdata$-foam describing $-^*$ on 2-morphisms is the identity, it follows that the combinatorial $\labdata$-foams underlying the two sides of the equality are both equal to $\mathrm{vert}_{U,T,S}\circ(\Theta\sqcup\Sigma)$, which proves the statement.
    \item Horizontal composition of 2-morphisms: Given 2-morphisms $\Sigma\in\Hom_{\Foambicathom{\labdata}{t}{s}}(S,T)$ and $\Sigma'\in\Hom_{\Foambicathom{\labdata}{u}{t}}(S',T')$, we have to show that $(\Sigma'\circ_h \Sigma)^*=\Sigma^*\circ_h\Sigma'^*$. Similarly to the previous point, this follows from the fact that at the level of combinatorial $\labdata$-foams, $-^*$ is given by the identity.
\end{itemize}
We now prove properties (i), (ii), and (iii).
\begin{enumerate}[(i)]
\item First, we show that $*\circ*$ is the identity 2-functor. On objects and 1-morphisms, this follows directly from the definition of $*$ in Constructions \ref{constr:dualobj} and \ref{constr:starOn1morphs}. For 2-morphisms, let $\Sigma\in \Hom_{\Foambicat{\labdata}}(S,T)$ where $S,T\colon s\to t$ are 1-morphisms. Then,
\begin{align*}
    (\Sigma^*)^*
    &=
    \webeval(\id_{\abs{\inthom(T^*,S^*)}})\circ \webeval(\id_{\abs{\inthom(S,T)}}) (\Sigma)
    \\ &= \webeval(\id_{\abs{\inthom(T^*,S^*)}}\circ\id_{\abs{\inthom(S,T)}}) (\Sigma)
    = 
    \webeval(\id_{\abs{\inthom(S,T)}}) (\Sigma)
    =
    \Sigma.
\end{align*}
\item We need to show three equations:
\begin{itemize}
    \item $(\Sigma\hcomp \id_{S^*})\vcomp \varepsilon_S = (\id_T\hcomp \Sigma^*) \vcomp \varepsilon_T$ for all 1-morphisms $S,T\colon s\to t$ and 2-morphisms $\Sigma \in \Hom_{\Foambicat{\labdata}}(S,T)$.
    This can be shown by verifying (using Constructions \ref{constr:inthom} and \ref{constr:intHorComp}) that the combinatorial $\labdata$-foams underlying both sides of the equation are equal to $\Sigma$.
    \item $(\id_S\circ_h (\varepsilon_{S^*})^*)\circ_v(\varepsilon_S\circ_h \id_S)=\id_S$ for all 1-morphisms $S\colon s\to t$. Indeed, by using Constructions \ref{constr:labfoamcomposition}, \ref{constr:intid}, \ref{constr:intcomp} and \ref{constr:intHorComp}, one can check that the combinatorial $\labdata$-foam underlying the left-hand side of the equation is equal to $\mathrm{cup}_S$. 
    \item $(\id_S \hcomp \varepsilon_U\hcomp \id_{S^*})\vcomp \varepsilon_S = \varepsilon_{S\circ U}$ for all 1-morphisms $S\colon s\to t$ and $U\colon r\to s$.
    Similarly to the previous point, this can be checked by going through Constructions \ref{constr:labfoamcomposition}, \ref{constr:intid}, \ref{constr:intcomp} and \ref{constr:intHorComp}.            
\end{itemize}
\item The equation $(x \boxtimes S \boxtimes y)^* = x \boxtimes  S^* \boxtimes y$ follows immediately from Constructions~\ref{constr:dualobj} and \ref{constr:starOn1morphs}. The equation  $\varepsilon_{x \boxtimes S \boxtimes y}=x \boxtimes \varepsilon_S \boxtimes y$ follows from \eqref{eq:defEpsilon} in Construction \ref{constr:adjointsFor1Morph} and \eqref{eq:cupTensId} in Lemma~\ref{lem:tensAndSqcup}.
\end{enumerate}
    \item \textbf{$\#$-duality:} We first show that $\triangulator_s$ is invertible. Given an object $s$, we define the 2-morphism 
    \begin{align*}
        \triangulator_s^{-1}:=\webeval(\mathrm{cup}_{\id_s})(1)\in\Hom_{\Foambicathom{\labdata}{s}{s}}\bigl(\id_s,(\eta_s^*\boxtimes s)\circ (s\boxtimes \eta_{s^\#})\bigr),
    \end{align*}
    where $\mathrm{cup}_{\id_s}$ is viewed as a foam 
    \begin{align*}
        \emptyset\to \abs{\inthom\bigl(\id_s,(\eta_s^*\boxtimes s)\circ (s\boxtimes \eta_{s^\#})\bigr)},
    \end{align*}
    by observing that $\abs{\inthom\bigl(\id_s,(\eta_s^*\boxtimes s)\circ (s\boxtimes \eta_{s^\#})\bigr)}=\abs{\inthom(\id_s,\id_{s})}$. Then it follows that $\triangulator_s\circ_v \triangulator_s^{-1}=\id_{\id_s}$ and $\triangulator_s^{-1}\circ_v \triangulator_s=\id_{(\eta_s^*\boxtimes s)\circ (s\boxtimes \eta_{s^\#})}$. Indeed, by unitality of $\circ_v$, the combinatorial $\labdata$-foams underlying the left-hand sides of both equalities are 
    \[
    \mathrm{vert}_{\id_s,\id_s,\id_s}\circ (\mathrm{cup}_{\id_s}\sqcup \mathrm{cup}_{\id_s})=\mathrm{cup}_{\id_s}=\mathrm{cup}_{(\eta_s^*\boxtimes s)\circ (s\boxtimes \eta_{s^\#})}.
    \]
    We now prove properties (i) through (iv).
    \begin{enumerate}[(i)]
        \item We have $s^{\#\#}=s$, for all objects $s$. This follows directly from Construction \ref{constr:dualobj}.
        \item The equations $\emptyset^\#=\emptyset$, $\eta_\emptyset=\id_\emptyset$ and $\triangulator_\emptyset=\id_{\id_\emptyset}$ are clear from the definition. 
        \item For all objects $s,t$:
         \begin{itemize}
            \item By definition, we have $(s\boxtimes t)^\# = t^\# \boxtimes s^\#$.
            \item From Construction~\ref{constr:fold}, it follows inductively that $\eta_{s\boxtimes t} = (s\boxtimes \eta_t \boxtimes s^{\#}) \circ \eta_s$.
            \item For $\triangulator_{s\boxtimes t} = ((\triangulator_s\boxtimes t) \hcomp (s \boxtimes \triangulator_t)) \vcomp (\id_{\eta_{s}^*\boxtimes s\boxtimes t}\hcomp \otimes_{s\boxtimes \eta_t^*,\eta_{s^\#}\boxtimes t} \hcomp \id_{s\boxtimes t \boxtimes \eta_{t^\#}})$ note that by \Cref{lem:tensAndSqcup}, the combinatorial $\labdata$-foam underlying the left-hand side is equal to $\mathrm{cup}_{\id_s}\sqcup\mathrm{cup}_{\id_t}$. For the right-hand side, observe that by \Cref{lem:tensAndSqcup} and unitality of $\circ_h$, the combinatorial $\labdata$-foam underlying $(\triangulator_s\boxtimes t) \hcomp (s \boxtimes \triangulator_t)$ is equal to $\mathrm{cup}_{\id_s}\sqcup\mathrm{cup}_{\id_t}$. Moreover, one can verify that the combinatorial $\labdata$-foams underlying 
        \begin{align*}
            \id_{\eta_{s}^*\boxtimes s\boxtimes t},\quad \id_{s\boxtimes t \boxtimes \eta_{t^\#}}, \quad \otimes_{s\boxtimes \eta_t^*,\eta_{s^\#}\boxtimes t}
        \end{align*}
        are equal, respectively, to 
        \begin{align*}
            (\mathrm{cup}_{\id_s})^{\sqcup 2}\sqcup\mathrm{cup}_{\id_t},\quad  \mathrm{cup}_{\id_s}\sqcup (\mathrm{cup}_{\id_t})^{\sqcup 2}, \quad (\mathrm{cup}_{\id_s})^{\sqcup 2}\sqcup(\mathrm{cup}_{\id_t})^{\sqcup 2}.
        \end{align*}
        Now, the third $\circ_h$ merges four copies of $\mathrm{cup}_{\id_t}$ into one, while the second $\circ_h$ merges four copies of $\mathrm{cup}_{\id_s}$ into one. By unitality of $\circ_v$, we conclude that the $\labdata$-foam on the right-hand side of the equation is equal to $\mathrm{cup}_{\id_s}\sqcup\mathrm{cup}_{\id_t}$, as desired.
        \end{itemize}
        \item Finally, we prove 
        \begin{align*}
            \id_{\eta_s^*} = (\id_{\eta_s^*}\hcomp (\triangulator_s \boxtimes s^\#))\vcomp (\otimes_{\eta_s^*,\eta_s^*} \hcomp \id_{s\boxtimes \eta_{s^\#}\boxtimes s^\#} ) \vcomp (\id_{\eta_{s}^*}\hcomp (s\boxtimes \triangulator_{s^\#}^*))
        \end{align*}
        for all objects $s$. First, observe that the combinatorial $\labdata$-foam underlying the left-hand side of the equation is equal to $\mathrm{cup}_{\id_s}$. Similarly to the previous point, one shows that the $\labdata$-foam describing the right-hand side is equal to $\mathrm{cup}_{\id_s}$ by decomposing each term as disjoint unions of $\mathrm{cup}_{\id_s}$ and using unitality of $\circ_h$ and $\circ_v$.
    \end{enumerate}
    \textbf{Spatiality:}
    The fact that $\#$ extends to a weak 2-functor, contravariant in the horizontal composition and in the monoidal directions, is shown in \cite[Theorem 4.3]{barrett2018gray}. For this, the operation $\#$ on 1- and 2-morphisms is defined in \cite[Equation (17)]{barrett2018gray}. On 1-morphisms $W\colon s\to t$,  this coincides with the left-mate 
        \begin{align*}
            \#(W) = \lmate{W}\colon t^\#\to s^\#
        \end{align*}
    from Construction~\ref{constr:mates}. Note that the operation ${*\#*}(W)=\rmate{W}$ assigns the right-mate to $W$. As in \cite[Equation (32)]{barrett2018gray}, we define the 2-morphism components 
    \begin{align*}
        \Delta_W \in \Hom_{\Foambicathom{\labdata}{s^\#}{t^\#}} (\lmate{W},\rmate{W})
    \end{align*} 
    of the natural isomorphism $\Delta$ as
    \begin{align*}
            \Delta_W =&\ (\id_{\rmate{W}} \hcomp (\triangulator^*_{t^\#})^{-1})\vcomp
            (\id_{\rmate{W}}\hcomp \id_{t^\#\boxtimes \eta_t^*}\hcomp t^\#\boxtimes \varepsilon_W^*\boxtimes t^\#\hcomp \id_{\eta_{t^\#}\boxtimes t^\#}) \\
            &\ \vcomp (\id_{\rmate{W}}\hcomp \id_{(t^\#\boxtimes \eta_t^*)\hcompw (t^\#\boxtimes W\boxtimes t^\#)}\hcomp t^\#\boxtimes \triangulator_s\boxtimes t^\#\hcomp \id_{(t^\#\boxtimes W^*\boxtimes t^\#)\hcompw (\eta_{t^\#}\boxtimes t^\#)}) \\
            &\ \vcomp (\id_{\rmate{W}}\hcomp \otimes_{t^\#\boxtimes \eta^*_s, \eta^*_t\hcompw (W\boxtimes t^\#)}\hcomp \id_{(t^\#\boxtimes s\boxtimes \eta_{s^\#}\boxtimes t^\#)\hcompw (t^\#\boxtimes W^*\boxtimes t^\#) \hcompw (\eta_{t^\#}\boxtimes t^\#)}) \\
            &\ \vcomp (\id_{\rmate{W}}\hcomp \id_{t^\#\boxtimes \eta^*_s}\hcomp \otimes_{(t^\#\boxtimes W^*)\hcompw \eta_{t^\#}, \lmate{W}}) 
            \vcomp (\varepsilon_{\rmate{W}}\hcomp \id_{\lmate{W}}).
    \end{align*}
    It remains to verify that the 2-morphisms components of the natural transformations $\Delta$ and ${*\Delta *}$ between right- and left-mate are equal. This relies on the fact that the underlying combinatorial $\labdata$-foams used in the construction of these 2-morphisms are equal to the identity $\labdata$-foam, and that applying $*$ is compatible with the combinatorial $\labdata$-foams describing vertical and horizontal composition.
\end{proof}

\begin{remark}\label{rem:HashAsFunctor}
   The weak 2-functor $\#\colon \Foambicat{\labdata} \to \Foambicat{\labdata}^{\text{hop,$\boxtimes$op}}$ \cite[Theorem 4.3]{barrett2018gray} used in \Cref{thm:dualityStruc}.(3) acts weakly contravariantly in the $\hcompw$-direction of composition of webs. In particular, there is a 2-isomorphism $ \lmate{W_1}\hcompw \lmate{W_2}\to \lmate{(W_2\hcompw W_1)}$ for every composable pair of 1-morphisms $W_1\colon s\to t$ and $W_2\colon t\to u$. The 2-isomorphism is constructed from identity 2-morphisms, the triangulator $\triangulator_t$ and two tensorators. Furthermore, the 2-morphism $(\triangulator_{s^{\#}})^*$ provides a 2-isomorphism $\id_{s^{\#}}\to (\id_s)^{\#}$. Compare the proof of \cite[Theorem 4.3]{barrett2018gray}.
\end{remark}
\begin{remark}[Strictification]
\label{rem:strict}
    Recall our 2-dimensional graphical calculus for 1-morphisms, and the 3-dimensional graphical calculus for 2-morphisms from Figure~\ref{fig:cubecoordinatedconvention}. In terms of the graphical calculus, the operation $*$ from Construction~\ref{constr:starOn1morphs} is given by a rotation around the $0$-axis by $\pi$ together with a reversal of the orientation of each edge. 
    The $\#$-operation---which assigns the left-mate $\lmate{W}$ to a 1-morphism $W$---does not correspond to a rotation on the graphical calculus, however. It would instead be more natural to consider the rotation by $\pi$ around the $2$-axis. In the following, we comment on how such an operation could be defined on $\Foambicat{\labdata}$.

    One option is to use the strictification procedure developed in \cite[Section 5]{barrett2018gray} for spatial Gray categories with duals. For this, one first adjoins formally rotated copies of objects and 1-morphisms\footnote{In the Gray category language these are 1-morphisms and 2-morphisms, respectively.}. Then an operation $\underline{\#}$ is defined to swap the original and rotated copies, which can then be extended to a strict 2-functor that strictifies the operation $\#$. This yields an equivalent Gray category with strict duals in the sense of \cite[Definition 5.1]{barrett2018gray}, at the expense of introducing duplicate objects and 1-morphisms. Compare \cite[Theorems 5.2 and 5.3]{barrett2018gray}. 

    Another option is to define an operation $\hash$ on $\Foambicat{\labdata}$ combinatorially which resembles a rotation. For a basic $\labdata$-web $o_l\boxtimes W' \boxtimes o_r$ with $W'$ a generating $\labdata$-web, we define
    \begin{align*}
        \hash(o_l\boxtimes W' \boxtimes o_r) := o_r^{\#}\boxtimes (\lmate{W'})\boxtimes o_l^{\#}
    \end{align*}
    and for a 1-morphism $W=(W_1,\ldots,W_k)$, we set $\hash(W):=(\hash(W_k),\ldots,\hash(W_1))$. 
    For consistency, we then also require $\hash(\id_s):=\id_{s^\#}$.
    Note that, unlike $\#$, the operation $\hash$ respects composition of 1-morphisms. With this definition, however, the operation $\hash$ is only involutive up to the 2-isomorphisms which trivialize taking left-mates twice. 
    
    Lastly, as a compromise between the two previous attempts to define an operation $\underline{\#}$ that truly corresponds to the rotation by $\pi$ around the 2-axis, it is possible to use an enlargement of $\Foambicat{\labdata}$ that is much milder than in \cite[Section 5]{barrett2018gray}. We can repeat the construction of $\Foambicat{\labdata}$ but additionally adjoin for every generating merge and split $\labdata$-web a rotated counterpart. Specifically, for all labels $x,y\in \labdata_1$, these additional generators are modeled as: 
    \begin{equation*}
        \splitdownxy{x}{y}
        \colon (\widecheck{x+y})\to (\widecheck{x},\widecheck{y})\;\;, \quad
         \mergedownxy{x}{y} \colon(\widecheck{x},\widecheck{y})\to (\widecheck{x+y}). 
     \end{equation*}
     The rest of the construction proceeds analogously. One can then prove that the new downward-oriented merge and split vertices are isomorphic to the left- and right-mates of their upward-oriented counterparts, so that the resulting semistrict monoidal 2-category is equivalent to $\Foambicat{\labdata}$. In the new model, however, the operation $\underline{\#}$ can be defined to interchange upward and downward versions of all merge, split, cup and cap webs. Thus one obtains a strict 2-functor $\underline{\#}$ on the modified $\Foambicat{\labdata}$, which is an involution, and the composition of $*$ and $\underline{\#}$---which then corresponds to a rotation by $\pi$ around the 1-axis---is also an involution. As a result, one obtains a Gray category with strict duals. 
\end{remark}

\begin{remark}[Sphericality]
    In the language of Douglas--Reutter \cite[Definitions~2.2.3 and 2.2.4]{douglas2018fusion}, we have constructed a \emph{pivotal structure} on the \emph{monoidal planar pivotal 2-category} $\Foambicat{\labdata}$. In this setting, one can consider various kinds of 2-spherical traces of a given 1-morphism $W\colon s\to t$, see \cite[Section~2.2.3]{douglas2018fusion}. As a consequence of pivotality, the so-called left and right traces always agree by \cite[Proposition~2.2.9]{douglas2018fusion}. Informally speaking, in the associated graphical calculus, this allows to freely manipulate 2-morphisms represented by diagrams on an embedded 2-sphere through isotopy on that sphere.
    
    On top of this, one can also consider an additional 3-sphericality property \cite[Definition~2.3.2]{douglas2018fusion}. To state it, note that 2-spherical traces for 2-endomorphisms of an object can be formed in two ways, closing the 2-sphere around the front or the back in the graphical calculus. The 3-sphericality property \cite[Definition~2.3.2]{douglas2018fusion} for a pivotal 2-category requires that the front and back traces of any 2-endomorphism agree. This is the case for $\Foambicat{\labdata}$ as we now argue.
    
    Indeed, the \emph{front trace} and \emph{back trace} of a 2-endomorphism $\Sigma\in \Hom_{\Foambicathom{\labdata}{s}{s}}(\id_s,\id_s)$ for an object $s$ in $\Foambicat{\labdata}$ are, respectively, defined as 
    \begin{align*}
        \Tr_F(\Sigma) &= (\varepsilon_{\rcapweb_s})^* \vcomp (\id_{\rcapweb_s}\hcomp (\Sigma \boxtimes s^\#) \hcomp \id_{\lcupweb_s}) \vcomp \varepsilon_{\rcapweb_s} \in \Hom_{\Foambicathom{\labdata}{\emptyset}{\emptyset}}(\id_\emptyset, \id_{\emptyset})\\
        \Tr_B(\Sigma) &= (\varepsilon_{\lcapweb_s})^* \vcomp (\id_{\lcapweb_s}\hcomp(s^\#\boxtimes \Sigma)\hcomp \id_{\rcupweb_s}) \vcomp \varepsilon_{\lcapweb_s}\in \Hom_{\Foambicathom{\labdata}{\emptyset}{\emptyset}}(\id_\emptyset, \id_{\emptyset}).
    \end{align*}
     These are equal since the underlying combinatorial $\labdata$-foams are the same.  
\end{remark}

\bibliographystyle{myamsalpha}
\bibliography{pw}

@misc{Rou2,
	author = {R.~Rouquier},
	note = {\arxiv{0812.5023}},
	title = {2-{K}ac-{M}oody algebras},
	year = {2008}}

@article{BHMV,
	author = {Blanchet, C. and Habegger, N. and Masbaum, G. and Vogel, P.},
	coden = {TPLGAF},
	doi = {10.1016/0040-9383(94)00051-4},
	fjournal = {Topology. An International Journal of Mathematics},
	issn = {0040-9383},
	journal = {Topology},
	mrclass = {57N10 (57M25 81T40)},
	mrnumber = {1362791 (96i:57015)},
	mrreviewer = {Justin D. Roberts},
	number = {4},
	pages = {883--927},
	title = {Topological quantum field theories derived from the {K}auffman bracket},
	volume = {34},
	year = {1995}}

@article{Bla,
	author = {Blanchet, Christian},
	doi = {10.1142/S0218216510007863},
	fjournal = {Journal of Knot Theory and its Ramifications},
	issn = {0218-2165},
	journal = {J. Knot Theory Ramifications},
	mrclass = {57M27 (57M25)},
	mrnumber = {2647055 (2011f:57011)},
	mrreviewer = {Paola Cristofori},
	note = {\arxiv{1405.7246}},
	number = {2},
	pages = {291--312},
	title = {An oriented model for {K}hovanov homology},
	url = {http://dx.doi.org/10.1142/S0218216510007863},
	volume = {19},
	year = {2010}}

@article{BN2,
	author = {Bar-Natan, Dror},
	fjournal = {Geometry and Topology},
	issn = {1465-3060},
	journal = {Geom. Topol.},
	mrclass = {57M27 (57M25)},
	mrnumber = {MR2174270},
	note = {\arxiv{math.GT/0410495}},
	pages = {1443--1499},
	title = {Khovanov's homology for tangles and cobordisms},
	volume = {9},
	year = {2005}}

@article{Cap,
	author = {Caprau, C.L.},
	doi = {10.2140/agt.2008.8.729},
	fjournal = {Algebraic \& Geometric Topology},
	issn = {1472-2747},
	journal = {Algebr. Geom. Topol.},
	mrclass = {57M27 (57M25 57R56)},
	mrnumber = {2443094 (2009g:57019)},
	mrreviewer = {Masahico Saito},
	note = {\arxiv{math/0707.3051}},
	number = {2},
	pages = {729--756},
	title = {{$\mathfrak{sl}(2)$} tangle homology with a parameter and singular cobordisms},
	url = {http://arxiv.org/abs/0707.3051},
	volume = {8},
	year = {2008}}

@article{CKM,
	author = {Cautis, S. and Kamnitzer, J. and Morrison, S.},
	doi = {10.1007/s00208-013-0984-4},
	fjournal = {Mathematische Annalen},
	issn = {0025-5831},
	journal = {Math. Ann.},
	mrclass = {17B37 (17B10 17B20)},
	mrnumber = {3263166},
	mrreviewer = {Iwan Praton},
	note = {\arxiv{1210.6437}},
	number = {1-2},
	pages = {351--390},
	title = {Webs and quantum skew {H}owe duality},
	url = {http://arxiv.org/abs/1210.6437},
	volume = {360},
	year = {2014}}

@article{MR0766964,
	author = {Vaughan F. R. Jones},
	coden = {BAMOAD},
	doi = {10.1090/S0273-0979-1985-15304-2},
	fjournal = {American Mathematical Society. Bulletin. New Series},
	issn = {0273-0979},
	journal = {Bull. Amer. Math. Soc. (N.S.)},
	mrclass = {57M25 (46L10)},
	mrnumber = {766964 (86e:57006)},
	mrreviewer = {J. S. Birman},
	note = {\mathscinet{MR766964} \doi{10.1090/S0273-0979-1985-15304-2}},
	number = {1},
	pages = {103--111},
	title = {A polynomial invariant for knots via von {N}eumann algebras},
	url = {http://dx.doi.org/10.1090/S0273-0979-1985-15304-2},
	volume = {12},
	year = {1985}}

@article{Kho,
	author = {Khovanov, M.},
	coden = {DUMJAO},
	doi = {10.1215/S0012-7094-00-10131-7},
	fjournal = {Duke Mathematical Journal},
	issn = {0012-7094},
	journal = {Duke Math. J.},
	mrclass = {57M27 (57R56)},
	mrnumber = {1740682 (2002j:57025)},
	note = {\arxiv{math.QA/9908171}},
	number = {3},
	pages = {359--426},
	title = {A categorification of the {J}ones polynomial},
	url = {http://arxiv.org/abs/math/9908171},
	volume = {101},
	year = {2000}}

@article{Kho3,
	author = {Khovanov$\phantom{a}\!\!\!$, M.},
	doi = {10.2140/agt.2004.4.1045},
	fjournal = {Algebraic \& Geometric Topology},
	issn = {1472-2747},
	journal = {Algebr. Geom. Topol.},
	mrclass = {57M27 (18G60 57R56)},
	mrnumber = {2100691 (2005g:57032)},
	mrreviewer = {Justin Sawon},
	note = {\arxiv{math.QA/0304375}},
	pages = {1045--1081},
	title = {$\mathfrak{sl}(3)$ link homology},
	url = {http://arxiv.org/abs/math/0304375},
	volume = {4},
	year = {2004}}

@article{KR,
	author = {Khovanov, Mikhail and Rozansky, Lev},
	fjournal = {Fundamenta Mathematicae},
	issn = {0016-2736},
	journal = {Fund. Math.},
	mrclass = {57M25 (57M27)},
	mrnumber = {MR2391017},
	note = {\arxiv{math.QA/0401268}},
	number = {1},
	pages = {1--91},
	title = {Matrix factorizations and link homology},
	volume = {199},
	year = {2008}}

@article{MOY,
	author = {Murakami, H. and Ohtsuki, T. and Yamada, S.},
	coden = {ENMAAR},
	fjournal = {L'Enseignement Math\'ematique. Revue Internationale. IIe S\'erie},
	issn = {0013-8584},
	journal = {Enseign. Math. (2)},
	mrclass = {57M27 (05C10)},
	mrnumber = {1659228},
	mrreviewer = {Leonid Plachta},
	number = {3-4},
	pages = {325--360},
	title = {Homfly polynomial via an invariant of colored plane graphs},
	volume = {44},
	year = {1998}}

@article{MSV,
	author = {Mackaay, M. and Sto{\v{s}}i{\'c}, M. and Vaz, P.},
	doi = {10.2140/gt.2009.13.1075},
	fjournal = {Geometry \& Topology},
	issn = {1465-3060},
	journal = {Geom. Topol.},
	mrclass = {57M27 (18G60 57M25)},
	mrnumber = {2491657 (2010h:57019)},
	mrreviewer = {Alberto Cavicchioli},
	note = {\arxiv{0708.2228}},
	number = {2},
	pages = {1075--1128},
	title = {{$\mathfrak{sl}_{N}$}-link homology {$(N\geq 4)$} using foams and the {K}apustin--{L}i formula},
	url = {http://arxiv.org/abs/0708.2228},
	volume = {13},
	year = {2009}}

@article{QR,
	author = {Queffelec, H. and Rose, D.E.V.},
	doi = {10.1016/j.aim.2016.07.027},
	fjournal = {Advances in Mathematics},
	issn = {0001-8708},
	journal = {Adv. Math.},
	mrclass = {57M27 (17B10 17B37 18D05 18G60)},
	mrnumber = {3545951},
	mrreviewer = {Michael Abel},
	note = {\arxiv{1405.5920}},
	pages = {1251--1339},
	title = {The $\mathfrak{sl}_n$ foam $2$-category: a combinatorial formulation of {K}hovanov--{R}ozansky homology via categorical skew {H}owe duality},
	url = {http://arxiv.org/abs/1405.5920},
	volume = {302},
	year = {2016}}

@article{RoW,
	author = {Robert, Louis-Hadrien and Wagner, Emmanuel},
	doi = {10.4171/qt/139},
	fjournal = {Quantum Topology},
	issn = {1663-487X},
	journal = {Quantum Topol.},
	mrclass = {57R56 (17B10 17B37 57K18)},
	mrnumber = {4164001},
	note = {\arxiv{1702.04140}},
	number = {3},
	pages = {411--487},
	title = {A closed formula for the evaluation of foams},
	url = {https://doi.org/10.4171/qt/139},
	volume = {11},
	year = {2020}}

@article{TVW,
	adsurl = {http://adsabs.harvard.edu/abs/2015arXiv150405069T},
	archiveprefix = {arXiv},
	author = {{Tubbenhauer}, D. and {Vaz}, P. and {Wedrich}, P.},
	eprint = {1504.05069},
	journal = {ArXiv e-prints},
	month = apr,
	primaryclass = {math.QA},
	title = {{Super $q$-Howe duality and web categories}},
	year = {2015},
    pages= {3703--3749}}

@article {2019arXiv190712194M,
    AUTHOR = {Morrison, Scott and Walker, Kevin and Wedrich, Paul},
     TITLE = {Invariants of 4-manifolds from {K}hovanov-{R}ozansky link
              homology},
   JOURNAL = {Geom. Topol.},
  FJOURNAL = {Geometry \& Topology},
    VOLUME = {26},
      YEAR = {2022},
	  note = {\arxiv{1907.12194}},
    NUMBER = {8},
     PAGES = {3367--3420},
      ISSN = {1465-3060},
   MRCLASS = {57K18 (57R56)},
  MRNUMBER = {4562565},
       DOI = {10.2140/gt.2022.26.3367},
       URL = {https://doi.org/10.2140/gt.2022.26.3367},
}

@article{MWBlob,
	title={Blob homology},
	volume={16},
	ISSN={1465-3060},
	url={http://dx.doi.org/10.2140/gt.2012.16.1481},
	DOI={10.2140/gt.2012.16.1481},
	number={3},
	journal={Geometry \& Topology},
	publisher={Mathematical Sciences Publishers},
	author={Morrison, Scott and Walker, Kevin},
	year={2012},
	month=jul, 
	pages={1481–1607} }

@misc{HRWBorderedInvariantsKh,
	title={Bordered invariants from {K}hovanov homology}, 
	author={Matthew Hogancamp and David E. V. Rose and Paul Wedrich},
	year={2024},
	eprint={2404.06301},
	archivePrefix={arXiv},
	primaryClass={math.GT},
	note={\arxiv{2404.06301}} 
}

@article{MR1185809,
	author = {Lickorish, W. B. R.},
	doi = {10.1007/BF02566519},
	fjournal = {Commentarii Mathematici Helvetici},
	issn = {0010-2571},
	journal = {Comment. Math. Helv.},
	mrclass = {57N10 (57M25)},
	mrnumber = {1185809},
	mrreviewer = {Louis H. Kauffman},
	note = {\mathscinet{MR1185809} \doi{10.1007/BF02566519}},
	number = {4},
	pages = {571--591},
	title = {Calculations with the {T}emperley-{L}ieb algebra},
	url = {https://doi.org/10.1007/BF02566519},
	volume = {67},
	year = {1992}}

@article {barrett2018gray,
    AUTHOR = {Barrett, John W. and Meusburger, Catherine and Schaumann,
              Gregor},
     TITLE = {Gray categories with duals and their diagrams},
   JOURNAL = {Adv. Math.},
  FJOURNAL = {Advances in Mathematics},
    VOLUME = {450},
      YEAR = {2024},
     PAGES = {Paper No. 109740, 129},
      ISSN = {0001-8708},
       DOI = {10.1016/j.aim.2024.109740},
       URL = {https://doi.org/10.1016/j.aim.2024.109740},
}

@misc{douglas2018fusion,
	archiveprefix = {arXiv},
	author = {Christopher L. Douglas and David J. Reutter},
	eprint = {1812.11933},
	note = {\arxiv{1812.11933}},
	primaryclass = {math.QA},
	title = {Fusion 2-categories and a state-sum invariant for 4-manifolds},
	year = {2018}}

@article{MR1928174,
	author = {Khovanov, Mikhail},
	fjournal = {Algebraic \& Geometric Topology},
	issn = {1472-2747},
	journal = {Algebr. Geom. Topol.},
	mrclass = {57M27 (57R56)},
	mrnumber = {MR1928174 (2004d:57016)},
	mrreviewer = {Jacob Andrew Rasmussen},
	note = {\arxiv{math.GT/0103190}},
	pages = {665--741},
	title = {A functor-valued invariant of tangles},
	volume = {2},
	year = {2002}}

@misc{liu2024braided,
	archiveprefix = {arXiv},
	author = {Yu Leon Liu and Aaron Mazel-Gee and David Reutter and Catharina Stroppel and Paul Wedrich},
	note = {\arxiv{2401.02956}},
	primaryclass = {math.QA},
	title = {A braided monoidal $(\infty,2)$-category of {S}oergel bimodules},
	year = {2024}}

@article {DSPS,
    AUTHOR = {Douglas, Christopher L. and Schommer-Pries, Christopher and
              Snyder, Noah},
     TITLE = {Dualizable tensor categories},
   JOURNAL = {Mem. Amer. Math. Soc.},
  FJOURNAL = {Memoirs of the American Mathematical Society},
    VOLUME = {268},
      YEAR = {2020},
    NUMBER = {1308},
     PAGES = {vii+88},
      ISSN = {0065-9266},
      ISBN = {978-1-4704-4361-0; 978-1-4704-6347-2},
   MRCLASS = {57R56 (16D90 17B37 18M15 55U30 57K35 57R15)},
  MRNUMBER = {4254952},
MRREVIEWER = {J\v{i}r\'{\i} Rosick\'{y}},
       DOI = {10.1090/memo/1308},
       URL = {https://doi.org/10.1090/memo/1308},
}

@misc{khovanov2020universal,
    title={Universal construction of topological theories in two dimensions},
    author={Mikhail Khovanov},
    year={2020},
    eprint={2007.03361},
    archivePrefix={arXiv},
    primaryClass={math.QA}
}

@article {BaezNeuchlHDAI,
	AUTHOR = {Baez, John C. and Neuchl, Martin},
	TITLE = {Higher-dimensional algebra. {I}. {B}raided monoidal
	{$2$}-categories},
	JOURNAL = {Adv. Math.},
	FJOURNAL = {Advances in Mathematics},
	VOLUME = {121},
	YEAR = {1996},
	NUMBER = {2},
	PAGES = {196--244},
	ISSN = {0001-8708},
	MRCLASS = {18D05 (18D20 58D29 81T99)},
	MRNUMBER = {1402727},
	MRREVIEWER = {Pilar C. Carrasco},
	DOI = {10.1006/aima.1996.0052},
	URL = {https://doi.org/10.1006/aima.1996.0052},
}

@misc{leonsurfaceskein,
      title={A construction of surface skein {TQFT}s and their extension to 4-dimensional 2-handlebodies}, 
      author={Leon J. Goertz},
      year={2025},
      eprint={2511.19352},
      archivePrefix={arXiv},
      primaryClass={math.QA},
      note={\arxiv{2511.19352}}, 
}

@misc{stroppel2024braidingtypesoergelbimodules,
      title={Braiding on type A Soergel bimodules: semistrictness and naturality}, 
      author={Catharina Stroppel and Paul Wedrich},
      year={2024},
	  note = {\arxiv{2412.20587}},
      archivePrefix={arXiv},
      primaryClass={math.QA},
      url={https://arxiv.org/abs/2412.20587}, 
}

@misc{WedrichSurveyLinkHomTQFT,
      title={From Link Homology to Topological Quantum Field Theories}, 
      author={Paul Wedrich},
      year={2025},
      eprint={2509.08478},
      archivePrefix={arXiv},
      primaryClass={math.QA},
      note={\arxiv{2509.08478}}, 
}

@article {ClivioGrFA,
    AUTHOR = {Clivio, Jonathan},
     TITLE = {Graded {F}robenius {A}lgebras},
   JOURNAL = {Q. J. Math.},
  FJOURNAL = {The Quarterly Journal of Mathematics},
    VOLUME = {76},
      YEAR = {2025},
    NUMBER = {4},
     PAGES = {1105--1157},
      ISSN = {0033-5606,1464-3847},
   MRCLASS = {18 (55P50 57)},
  MRNUMBER = {5003282},
       DOI = {10.1093/qmath/haaf029},
       URL = {https://doi.org/10.1093/qmath/haaf029},
}

@article {GPStricat,
    AUTHOR = {Gordon, R. and Power, A. J. and Street, Ross},
     TITLE = {Coherence for tricategories},
   JOURNAL = {Mem. Amer. Math. Soc.},
  FJOURNAL = {Memoirs of the American Mathematical Society},
    VOLUME = {117},
      YEAR = {1995},
    NUMBER = {558},
     PAGES = {vi+81},
      ISSN = {0065-9266,1947-6221},
   MRCLASS = {18D05},
  MRNUMBER = {1261589},
MRREVIEWER = {Kimmo\ I.\ Rosenthal},
       DOI = {10.1090/memo/0558},
       URL = {https://doi.org/10.1090/memo/0558},
}

@misc{khovanov2025lecturessl3foamslink,
      title={Lectures on SL(3) foams and link homology}, 
      author={Mikhail Khovanov and Dmitry Solovyev},
      year={2025},
      eprint={2507.17119},
      archivePrefix={arXiv},
      primaryClass={math.QA},
      url={https://arxiv.org/abs/2507.17119}, 
}

@article {MR3880205,
    AUTHOR = {Kronheimer, P. B. and Mrowka, T. S.},
     TITLE = {Tait colorings, and an instanton homology for webs and foams},
   JOURNAL = {J. Eur. Math. Soc. (JEMS)},
  FJOURNAL = {Journal of the European Mathematical Society (JEMS)},
    VOLUME = {21},
      YEAR = {2019},
    NUMBER = {1},
     PAGES = {55--119},
      ISSN = {1435-9855,1435-9863},
   MRCLASS = {57R58 (05C15)},
  MRNUMBER = {3880205},
MRREVIEWER = {Nikolai\ N.\ Saveliev},
       DOI = {10.4171/JEMS/831},
       URL = {https://doi.org/10.4171/JEMS/831},
}

@book {MR3076451,
    AUTHOR = {Gurski, Nick},
     TITLE = {Coherence in three-dimensional category theory},
    SERIES = {Cambridge Tracts in Mathematics},
    VOLUME = {201},
 PUBLISHER = {Cambridge University Press, Cambridge},
      YEAR = {2013},
     PAGES = {viii+278},
      ISBN = {978-1-107-03489-1},
   MRCLASS = {18D05 (18-02)},
  MRNUMBER = {3076451},
MRREVIEWER = {Josep\ Elgueta},
       DOI = {10.1017/CBO9781139542333},
       URL = {https://doi.org/10.1017/CBO9781139542333},
}

@article {MR2039036,
    AUTHOR = {Kapustin, Anton and Li, Yi},
     TITLE = {Topological correlators in {L}andau-{G}inzburg models with
              boundaries},
   JOURNAL = {Adv. Theor. Math. Phys.},
  FJOURNAL = {Advances in Theoretical and Mathematical Physics},
    VOLUME = {7},
      YEAR = {2003},
    NUMBER = {4},
     PAGES = {727--749},
      ISSN = {1095-0761,1095-0753},
   MRCLASS = {81T45 (14D21 18E30 81T30)},
  MRNUMBER = {2039036},
MRREVIEWER = {Christopher\ P.\ Herzog},
       DOI = {10.4310/atmp.2003.v7.n4.a5},
       URL = {https://doi.org/10.4310/atmp.2003.v7.n4.a5},
}

@article {MR2322554,
    AUTHOR = {Khovanov, Mikhail and Rozansky, Lev},
     TITLE = {Topological {L}andau-{G}inzburg models on the world-sheet
              foam},
   JOURNAL = {Adv. Theor. Math. Phys.},
  FJOURNAL = {Advances in Theoretical and Mathematical Physics},
    VOLUME = {11},
      YEAR = {2007},
    NUMBER = {2},
     PAGES = {233--259},
      ISSN = {1095-0761,1095-0753},
   MRCLASS = {81T45 (81T40)},
  MRNUMBER = {2322554},
MRREVIEWER = {Johannes\ Walcher},
       DOI = {10.4310/atmp.2007.v11.n2.a2},
       URL = {https://doi.org/10.4310/atmp.2007.v11.n2.a2},
}

@misc{dyckerhoff2025perverseschoberscoxetertype,
      title={Perverse schobers of Coxeter type $\mathbb{A}$}, 
      author={Tobias Dyckerhoff and Paul Wedrich},
      year={2025},
      eprint={2504.08496},
      archivePrefix={arXiv},
      primaryClass={math.QA},
      url={https://arxiv.org/abs/2504.08496}, 
}

@misc{GoertzWedrich,
      title={A note on {TQFT}s for orientable 2-dimensional cobordisms}, 
      author={Leon J. Goertz and Paul Wedrich},
      year={2025},
      eprint={2511.19373},
      archivePrefix={arXiv},
      primaryClass={math.QA},
      note={\arxiv{2511.19373}}, 
}
\end{document}